\documentclass[preprint,3p,times,sort&compress]{elsarticle}

\usepackage{amssymb}
\usepackage{amsmath}
\usepackage{amsthm}
\usepackage{lineno}

\newcommand\dd{\,\mathrm{d}}

\newcommand\e{\mathrm{e}}

\newcommand\bbR{\mathbb{R}}
\newcommand\bbN{\mathbb{N}}

\newcommand\He{\mathit{He}}
\newcommand\Kn{\mathit{Kn}}

\newcommand\mA{\mathcal{A}}
\newcommand\mM{\mathcal{M}}
\newcommand\mF{\mathcal{F}}

\newcommand\mH{\mathcal{H}}

\newcommand\mR{\mathcal{R}}

\newcommand\mQ{\mathcal{Q}}

\newcommand\mK{\mathcal{K}}
\newcommand\mI{\mathcal{I}}

\newcommand\bI{{\bf{I}}}

\newcommand\bs{\boldsymbol{s}}
\newcommand{\bbf}{\boldsymbol{f}}

\newcommand\bJ{\boldsymbol{J}}
\newcommand\bR{\boldsymbol{R}}

\newcommand\bxi{\boldsymbol{\xi}}
\newcommand\bx{\boldsymbol{x}}

\newcommand\bq{\boldsymbol{q}}

\newcommand\bu{\boldsymbol{u}}
\newcommand\bU{\boldsymbol{U}}
\newcommand\bv{\boldsymbol{v}}
\newcommand\be{\boldsymbol{e}}

\newcommand\bn{\boldsymbol{n}}

\newcommand\bvpi{\boldsymbol{\varpi}}
\newcommand\bvrho{\boldsymbol{\varrho}}
\newcommand\bsigma{\boldsymbol{\sigma}}

\newcommand\pd[2]{\dfrac{\partial {#1}}{\partial {#2}}}

\newcommand\bra[1]{\left( {#1} \right)}
\newcommand\abs[1]{| {#1} |}

\newcommand\norm[1]{\Vert {#1} \Vert}

\usepackage{mathtools}
\usepackage{latexsym}
\usepackage{color}
\usepackage[colorlinks=true, linkcolor=blue, citecolor=blue, urlcolor=blue]{hyperref}
\usepackage{cleveref}
\usepackage{lipsum}
\usepackage{amsfonts}
\usepackage{graphicx}
\usepackage{epstopdf}
\usepackage{algorithm}
\usepackage{algorithmic}
\usepackage{tikz}
\usepackage{booktabs}
\usepackage{siunitx}
\usepackage{subfig}
\usepackage{multirow}
\usepackage{diagbox}

\ifpdf
  \DeclareGraphicsExtensions{.eps,.pdf,.png,.jpg}
\else
  \DeclareGraphicsExtensions{.eps}
\fi

\usetikzlibrary{calc} 
\tikzstyle{start} = [rectangle, draw, fill=pink!20, 
    text width=8em, text centered, rounded corners, minimum height=5em]
\tikzstyle{Euler} = [rectangle, draw, fill=blue!20, 
    text width=8em, text centered, rounded corners, minimum height=5em]
\tikzstyle{Moment} = [rectangle, draw, fill=yellow!20, 
    text width=8em, text centered, rounded corners, minimum height=5em]
\tikzstyle{line} = [draw, -stealth, line width=2pt]

\graphicspath{{figs/}}

\theoremstyle{remark}
\newtheorem{remark}{Remark}
\theoremstyle{theorem}

\newtheorem{proposition}{Proposition}
\crefname{figure}{Figure}{Figure}
\crefname{table}{Table}{Table}
\crefname{equation}{}{}
\crefname{appendix}{}{}

\begin{document}

\begin{frontmatter}



\title{An improved fast iterative moment method for the simulation of
	near-continuum gas flows with Shakhov model \tnoteref{tnote1}}%
\tnotetext[tnote1]{
	The work of Zhenning Cai was supported by the Academic Research Fund of the
	Ministry of Education of Singapore under grant A-8002392-00-00. 
	The work of others was partially supported by the National Natural
    Science Foundation of China, No. 12171240 and the National Science
    and Technology Major Project, No. J2019-II-0007-0027.
    The
    computational resources were supported by High Performance
    Computing Platform of Nanjing University of Aeronautics and
    Astronautics, China.}

\author[1]{Zhenning Cai} 
\ead{matcz@nus.edu.sg}
\author[2,3]{Zhicheng Hu\corref{cor3}}
\cortext[cor3]{Corresponding author: Zhicheng Hu}
\ead{huzhicheng@nuaa.edu.cn}
\author[2,3]{Guanghan Li} 
\ead{liguanghan@nuaa.edu.cn}
\author[2,3]{Chunwu Wang}
\ead{wangcw@nuaa.edu.cn}
\author[2,3]{Xiaowen Wang}
\ead{wangxw@nuaa.edu.cn}

\affiliation[1]{organization={Department of Mathematics, National University of
		Singapore},
	addressline={Level 4, Block S17, 10 Lower Kent Ridge Road}, 
	postcode={119076}, 
	country={Singapore}}
\affiliation[2]{organization={School of Mathematics, Nanjing University of
Aeronautics and Astronautics},
            city={Nanjing},
            postcode={211106}, 
            country={China}}
\affiliation[3]{organization={Key Laboratory of
Mathematical Modelling and High Performance Computing of Air Vehicles (NUAA)},
            city={Nanjing},
            postcode={211106}, 
            country={China}}

\begin{abstract}
  The fast iterative moment (FIM) method developed in [G. Li, C. Wang,
  and Z. Hu, \textit{Comput. Fluids}, 311 (2026), 107031] achieves
  fast convergence for near-continuum gas flows modeled by the
  Boltzmann-BGK equation through alternating iterations between the
  high-order moment system and the consistent hydrodynamic
  equations. In this paper, we extend the FIM solver to the Shakhov
  model and improve it in the following aspects. A
  finite volume method with linear reconstruction is employed for
  high-resolution spatial discretization of both
  systems. Following the original FIM method, an
  implicit-explicit scheme treating the Shakhov collision term
  implicitly is adopted for the high-order moment
  system. A local Fourier analysis of a linearized
    one-dimensional BGK model reveals the need for more hydrodynamic
    iterations at smaller Knudsen numbers and their slower convergence
    on refined grids. Motivated by these findings, a Newton-multigrid
    method with a new regularization strategy is introduced for the
    hydrodynamic equations, generally providing an effective
    hydrodynamic correction with one step per alternating iteration,
    thereby reducing the need for empirical tuning.  Numerical
  experiments for planar Couette flow, lid-driven cavity flow, and
  heat transfer in a bottom-heated cavity are carried out to
  investigate the performance of the improved FIM solver. The results
  demonstrate improved efficiency and robustness over the
    original FIM method, particularly at small Knudsen numbers, and
    favorable performance with both first- and second-order spatial
  discretizations.
\end{abstract}




\end{frontmatter}


\section{Introduction}
\label{sec:intro}

In rarefied gas flows,
the degree of gas rarefaction is characterized by the Knudsen number
$\Kn$, which is defined as the ratio between the mean free path $\lambda$ of
gas molecules and the characteristic flow length $L$. 
The Navier-Stokes-Fourier (NSF) equations based on the continuum fluid
hypothesis are usually applicable up to $\Kn \sim 0.001$ \cite{Wu2014}.
When $\Kn$ is larger, the continuum hypothesis breaks down and the NSF
equations fail to capture the non-conventional behaviour of the rarefied flow. 
To model rarefied gas dynamics accurately, 
mesoscopic kinetic models are
generally adopted, among which the Boltzmann equation is the most fundamental one.
However, despite the rapid development of modern high-performance computing,
the numerical simulation of this equation remains computationally
challenging due to its high dimensionality and nonlinear integro-differential
structure.
To balance efficiency and accuracy, a popular strategy in the near-continuum
flow regime ($0.001 < \Kn < 0.1$) is to replace the
complex collision integral with simplified relaxation terms, such as the
Bhatnagar-Gross-Krook (BGK) model \cite{Bhatnagar1954}, the ellipsoidal
statistical BGK (ES-BGK) model \cite{HolwayJr1966}, and the Shakhov model \cite{Shakhov1972}.
Nevertheless, even with these simplified collision models, numerical
simulation is still computationally demanding.
In the past few decades, 
considerable effort has been devoted to developing
efficient and accurate numerical methods for the Boltzmann equation and its
simplified collision models; see, e.g., \cite{Mieussens2000,Xu2010,Cai2010,Su2020,Yang2022,Boscarino2022}. 

The moment method provides a hierarchy of macroscopic transport models capable of describing rarefaction effects beyond the NSF equations \cite{Struchtrup2005}.
The method, originally proposed by Grad \cite{Grad1949}, approximates the distribution function by a Hermite expansion in velocity space with expansion parameters determined by the local physical quantities. When an appropriate closure ansatz is adopted \cite{Cai2014,Cai2015framework}, the global hyperbolicity of moment systems of arbitrary order can be guaranteed.
As the moment order increases, these systems often provide increasingly accurate approximations to the underlying Boltzmann equation \cite{Torrilhon2017,Hu2016}. However, the rapidly increasing number of moment equations makes the direct formulation and implementation of high-order systems cumbersome. 
To address this issue, a unified numerical framework for high-order moment systems is developed in \cite{Cai2010,Cai2012,Cai2013microflow}, where the full discretization can be formally obtained without explicitly writing out each equation. 
While this framework facilitates the practical implementation of high-order moment systems, efficiently computing their steady-state solutions remains challenging.

To accelerate the convergence of steady-state solutions of hyperbolic moment systems, we have previously developed two iterative solvers under the nonlinear multigrid (NMG) framework \cite{Brandt2011}: an NMG iteration using spatial coarse-grid corrections \cite{Hu2014nmg,Hu2023} and a nonlinear
multilevel moment (NMLM) iteration using lower-order model corrections \cite{Hu2016,Hu2019}.
Both methods have shown significant improvement in convergence in comparison to the direct time-integration methods.
However, the convergence of both methods still deteriorates significantly as $\Kn$ decreases, leading to excessively high computational cost in the near-continuum flow regime. 
Moreover, integrating the two types of
coarse-level corrections into a more efficient and robust multilevel solver remains nontrivial.

Motivated by the slow convergence of these methods in the near-continuum flow regime, we have developed a fast iterative moment (FIM) method in \cite{Li2026} by alternately solving the high-order moment system and the hydrodynamic equations with consistent constitutive relations and boundary conditions. Since the hydrodynamic equations constitute a subsystem of the moment system for $M\ge 2$, they can efficiently predict the evolution of the primary macroscopic quantities at a much lower computational cost and thereby accelerate the overall iteration. This alternating framework can be further incorporated into the spatial NMG method, leading to an FIM-based multigrid solver that combines coarse-level corrections in the velocity and spatial spaces.
From a broader perspective, 
the FIM method shares some conceptual similarities with several recently developed micro-macro acceleration methods, including the general synthetic iterative scheme (GSIS) \cite{Zhang2025}, the unified gas-kinetic scheme (UGKS) \cite{Long2024}, and the improved discrete velocity method (IDVM) \cite{Yang2018,Yang2022}, although their formulations and implementations are fundamentally different.

Nevertheless, there is still room for improvement in the accuracy and efficiency of the FIM method.
The original implementation is restricted to the Boltzmann-BGK equation and employs only first-order spatial discretization, which introduces considerable numerical dissipation and limits the spatial resolution. Moreover, the time-integration-based hydrodynamic solver in each alternating iteration usually requires a large number of iterations to sufficiently evolve the macroscopic quantities when $\Kn$ is small.
Its convergence may also slow down as the grid is refined, leading to increased computational cost for high-resolution simulations.
Motivated by these limitations, in this work we extend the FIM solver to the Shakhov model, with the following main contributions.
\begin{itemize}
\item
	The spatial discretization is extended to second-order accuracy by employing a finite volume method with linear reconstruction for both the moment system and the consistent hydrodynamic equations.
\item 
  An implicit-explicit (IMEX) scheme is adopted for the moment system, in which the Shakhov collision term is treated implicitly to improve numerical stability at small Knudsen numbers, as in the original FIM method for the BGK model.
\item
	A local Fourier analysis of a linearized one-dimensional BGK model is performed to explain the need for more hydrodynamic iterations at smaller Knudsen numbers and their slower convergence on refined grids.
\item 
  A Newton-multigrid method equipped with a new regularization strategy is introduced for the hydrodynamic equations, replacing the time-integration-based methods used in \cite{Li2026}.
\end{itemize}
In general, one Newton-multigrid step per alternating iteration is sufficient to provide an effective hydrodynamic correction, substantially reducing the need for parameter tuning required in \cite{Li2026}. The resulting framework also provides an alternative way to integrate coarse-level corrections in the velocity and spatial spaces through the hydrodynamic equations and the Newton-multigrid method, respectively. With these improvements, the FIM solver achieves enhanced efficiency and robustness, particularly at small Knudsen numbers.

The rest of this paper is arranged as follows.
In \cref{sec:model}, we introduce the model equations, their spatial discretization with linear reconstruction, and the FIM framework with the time-integration-based hydrodynamic solver.
The local Fourier analysis is presented in \cref{sec:fim3-convergence-analysis}, and in \cref{sec:numerical-scheme} both the IMEX scheme for the moment system with the Shakhov model and the Newton-multigrid method for the hydrodynamic equations are presented.
Numerical experiments are carried out in \cref{sec:experiments} to demonstrate the performance of the proposed method. Finally, some concluding remarks are given in the last section.

\section{Governing equations and FIM framework}
\label{sec:model}
In this section, we first give a brief review of the steady-state Boltzmann-Shakhov equation and the high-order moment models, then present a unified spatial discretization with linear reconstruction for the given models, and finally introduce the FIM framework for solving the resulting discrete system.

\subsection{Boltzmann equation with Shakhov collision term}
\label{sec:boltzmann-eq}

The Boltzmann equation with a BGK-type collision term provides a fundamental description of gas flows in the near-continuum regime. For the steady-state problems considered in this work, it takes the form
\begin{align}
  \label{eq:boltzmann}
	 \bv \cdot \nabla_{\bx} f  = \bar{\nu} (f^{\text{E}} - f) .
\end{align}
Here, $f(\bx, \bv)$ is the distribution function, in which $\bx =
(x_1, \dots, x_D) \in \Omega \subset \bbR^D$ ($D = 1, 2$ or $3$) and $\bv =
(v_1, v_2, v_3) \in \bbR^3$ represent the spatial location and velocity of the
gas molecules, respectively.
The left-hand side of \cref{eq:boltzmann} describes the change of $f$ due to the
free streaming, while the right-hand side models the interactions between gas
molecules via a simplified relaxation process.
The average collision frequency $\bar{\nu}$ determines how fast the distribution
function approaches equilibrium, assumed to be independent of $\bv$ and inversely proportional to the Knudsen number, i.e., $\bar{\nu} \propto 1/\Kn$. In the BGK model, the function $f^{\text{E}}$ equals the local Maxwellian  $f^{\mM}$ defined by
\begin{align}
	 \label{eq:maxwellian}
	 f^{\mM}(\bx,\bv) = \frac{\rho(\bx)}{m_* (2 \pi \theta(\bx))^{3/2}}
	 \exp \left( - \frac{\abs{\bv - \bu(\bx)}^2}{2 \theta(\bx)} \right),
\end{align}
where $m_{*}$ is the mass of a single gas molecule. 
The macroscopic quantities in \cref{eq:maxwellian}, including the mass density
$\rho(\bx)$, the mean velocity $\bu(\bx) = (u_1, u_2, u_3)$ and the temperature
$\theta(\bx)$, can be computed from the distribution function $f$ as
\begin{align}
  \label{eq:rho-u-theta}
	 \begin{aligned}
	 \rho(\bx) = m_* \int_{\bbR^3} f(\bx,\bv) \dd \bv, \quad 
	 \rho(\bx) \bu(\bx) = m_* \int_{\bbR^3} \bv f(\bx,\bv) \dd \bv, \quad
	 \rho(\bx) \vert \bu(\bx) \vert^2 + 3 \rho(\bx) \theta(\bx) = m_*
	 \int_{\bbR^3} \vert \bv \vert^2 f(\bx, \bv) \dd \bv. 
	 \end{aligned}
\end{align}

It is well known that the BGK model does not produce the correct Prandtl number.
A commonly used fix is the Shakhov model \cite{Shakhov1972},  
in which the equilibrium distribution $f^{\text{E}}$ is given by
\begin{align}
	\label{eq:Shakhov-model}
	f^{\text{E}}(\bx,\bv) = \left[ 1 + \frac{\bra{1-\Pr} \bra{\bv-\bu(\bx)}
	\cdot \bq(\bx)}{5 \rho(\bx) \left[ \theta(\bx) \right]^2} \bra{\frac{|\bv -
	\bu(\bx)|^2}{\theta(\bx)} -5} \right] f^{\mM}(\bx,\bv),
\end{align}
where the Prandtl number can be specified by the parameter $\Pr$, and
$\bq$ denotes the heat flux of the gas to be defined later. When ${\Pr} =
1$, the Shakhov model reduces to the BGK model. The nonequilibrium moments, including the stress tensor $\sigma_{ij}$, $i,j=1,2,3$, and the heat flux $\bq$, can be computed from the distribution function as
\begin{align}
  \label{eq:sigma-q}
	 \begin{aligned}
	 & \sigma_{ij}(\bx) = m_* \int_{\bbR^3} (v_i - u_i(\bx))(v_j -
	 u_j(\bx)) f(\bx,\bv) \dd \bv - \rho(\bx) \theta(\bx) \delta_{ij},
         \qquad i,j = 1,2,3,
	 \\ & \bq(\bx) = \frac{m_*}{2} \int_{\bbR^3} \vert
	 \bv-\bu(\bx) \vert^2 (\bv-\bu(\bx)) f(\bx,\bv) \dd \bv.
	 \end{aligned}
\end{align}
Here $\delta_{ij}$ is Kronecker's delta symbol. Note that $\sigma_{ij}$ is a
trace-free tensor satisfying $\sigma_{11} + \sigma_{22} + \sigma_{33} = 0$.

\subsection{High-order hyperbolic moment system and hydrodynamic equations}
\label{sec:vel-dis}

To obtain moment models for the Boltzmann equation \cref{eq:boltzmann}, the
distribution function is first expanded into a series of Hermite functions as
\begin{align}
  \label{eq:dis-expansion}
  f(\bx, \bv) = \sum_{\alpha \in \bbN^{3}} f_\alpha(\bx)
  \mH_{\alpha}^{[\bvpi, \vartheta]} (\bv), 
\end{align}
where $f_\alpha(\bx)$ are the coefficients, and $\mH_{\alpha}^{[\bvpi,
\vartheta]}(\cdot)$ are the basis functions defined by
\begin{align*}
  \mH_{\alpha}^{[\bvpi, \vartheta]}(\bv) =
  \frac{1}{m_*(2\pi\vartheta)^{3/2}\vartheta^{|\alpha|/2}} \prod\limits_{d=1}^3\He_{\alpha_d}({\xi}_d) \exp \left(-\frac{{\xi}_d^2}{2} \right), \quad {\bxi} =
  \frac{\bv-\bvpi}{ \sqrt{\vartheta}},~~\forall \bv \in \bbR^3.
\end{align*}
In the above expression, $\alpha=(\alpha_1,\alpha_2,\alpha_3)$ is a multi-index with order $|\alpha|=\alpha_1+\alpha_2+\alpha_3$, 
and $\He_n(\cdot)$ is the Hermite polynomial of degree $n$ given by
\begin{align*}
  \He_n(x) = (-1)^n\exp \left(\frac{x^{2}}{2} \right) \frac{\dd^n}{\dd x^n} 
  \exp \left(-\frac{x^{2}}{2}\right).
\end{align*}
The parameters $\bvpi \in \bbR^3$ and $\vartheta \in \bbR^{+}$ in the basis
function could be either constants as in the Hermite spectral method
\cite{Hu2020}, or variables depending on $\bx$. 
In Grad's moment method
\cite{Grad1949}, $\bvpi$ and $\vartheta$ are chosen adaptively as the local
mean velocity $\bu$ and the local temperature $\theta$, which are determined
from $f$ itself via \cref{eq:rho-u-theta}. With this choice, the relationship
between the moment coefficients $f_{\alpha}$ and the macroscopic variables can
be deduced from \cref{eq:rho-u-theta} and \cref{eq:sigma-q} as
\begin{align}
  \label{eq:moments-relation}
  \begin{aligned}
    &f_0 = \rho, \qquad f_{e_1} = f_{e_2} = f_{e_3} = 0, \qquad
    \sum_{d=1}^3 f_{2e_d} = 0,\\
    &\sigma_{ij} = (1+\delta_{ij}) f_{e_i+e_j}, \quad q_i = 2
    f_{3e_i} + \sum_{d=1}^3 f_{2e_d+e_i}, \quad i,j=1,2,3,
  \end{aligned}
\end{align}
where $e_1$, $e_2$, $e_3$ represent the multi-indices $(1,0,0)$, $(0,1,0)$ and
$(0,0,1)$, respectively.
Since the first term of series \cref{eq:dis-expansion} already recovers the local Maxwellian, Grad's method often provides an efficient approximation of the distribution function for flows close to thermodynamic equilibrium.
Accordingly, we adopt Grad's choice $\bvpi=\bu$ and $\vartheta=\theta$ for the representation of the distribution function in the near-continuum flow regime considered in this paper.

Numerically, the distribution function is approximated by a truncated series of
order $M$, that is,
\begin{align}
  \label{eq:truncated-expansion}
  f(\bx, \bv) \approx \sum_{|\alpha| \leq M} f_\alpha(\bx)
  \mH_{\alpha}^{[\bu, \theta]} (\bv) \in \mF_{M}^{[\bu,\theta]}, 
\end{align}
in which $\mF_{M}^{[\bu,\theta]}$ is the
linear space spanned by the basis functions $\mH_{\alpha}^{[\bu,\theta]}(\bv)$ with $|\alpha| \leq M$. 
Substituting the expansion \cref{eq:truncated-expansion} into the Boltzmann equation \cref{eq:boltzmann} and following the procedures in \cite{Cai2012,Cai2014,Cai2015}, we obtain the globally hyperbolic moment system of order $M$. With $\bvrho$ denoting the vector of all independent variables $\bu$, $\theta$, and $f_{\alpha}$ with $\abs{\alpha}\leq M$,
the resulting system can be written in the quasi-linear form
\begin{align}
  \label{eq:moment-system}
	 \sum_{d=1}^{D} {\bf{A}}_{d}(\bvrho) \pd{\bvrho}{x_{d}} = \mQ(\bvrho),
\end{align}
where the left-hand side is derived from the convection part of the Boltzmann equation, and the right-hand side corresponds to the collision term.
The detailed expression of the above system can be found for the
Boltzmann-Shakhov equation in \ref{sec:expression_HME}.
For $M\ge2$, the factor $1/\Kn$ enters the right-hand side of \cref{eq:moment-system} through the collision frequency $\bar{\nu}$, making the system increasingly stiff as $\Kn\to0$.

Moreover, the conservation laws of mass, momentum and energy form a natural subsystem of the moment system \cref{eq:moment-system} for $M\ge2$. The corresponding hydrodynamic equations for the
primary macroscopic quantities $\rho$, $\bu$ and $\theta$ take the form
\begin{align}
	\label{eq:HEs}
	\left \{
	\begin{aligned}
		& \nabla_{\bx} \cdot \bra{\rho \bu} = 0, \\
		& \nabla_{\bx} \cdot \bra{\rho \bu \otimes \bu
		+ p \bI + \bsigma} = 0, \\
		& \nabla_{\bx} \cdot \bra{\rho E \bu + p \bu + \bu
		\cdot \bsigma + \bq} = 0,
	\end{aligned} \right.
\end{align}
where $p = \rho \theta$ is the pressure, $\bI$ is the identity matrix, and $E = 3 \theta / 2  +\vert \bu \vert^2 /2$ is the specific total energy.

\subsection{Spatial discretization of the moment system and linear reconstruction}
\label{sec:spa-dis}
The hyperbolic moment system \cref{eq:moment-system} of order $M$ is discretized following the unified finite volume method developed in \cite{Cai2010,Cai2012,Cai2012high,Cai2013microflow}.
For convenience, only the one-dimensional formulation is presented here, while the extension to two-dimensional problems on rectangular grids is straightforward.

Suppose the spatial domain $\Omega = (x_L, x_R)$ is partitioned into $N_{x}$ uniform grid cells of size $\Delta x$, where the $i$th cell is denoted by $\mK_i = [x_{i-1/2},x_{i+1/2}]$. 
In each cell $\mK_i$, we use $f_i(\bv)$ and
$f_{i}^{\pm}(\bv)$ to denote the cell-averaged distribution function and 
the reconstructed interface distribution functions at $x_{i\pm1/2}$, respectively.
Then the finite volume discretization of the Boltzmann equation \cref{eq:boltzmann} is given by
\begin{align}
	 \label{eq:x-fv-dis}
	 - \frac{1}{\Delta x} \bra{ \bar{F}_{i+1/2} \bra{f_{i}^{+}, f_{i+1}^{-}} -
	 \bar{F}_{i-1/2} \bra{f_{i-1}^{+}, f_{i}^{-}} } + \mQ(f_{i}) = 0.
\end{align}
Here, $\bar{F}_{i \pm 1/2} \bra{\cdot, \cdot}$
are the numerical fluxes at the boundaries of
the $i$th grid cell, 
and $\mQ(f_i)$ denotes the collision term evaluated using the cell-averaged distribution function $f_i$.
Let us assume that $f_{i}(\bv)$ belongs to the function space $\mF_{M}^{[\bu_{i},
\theta_{i}]}$, i.e.,
\begin{align}
  \label{eq:truncated-dis-i}
	 f_{i} (\bv)  = \sum_{|\alpha| \leq M} f_{i,\alpha} \mH_{\alpha}^{[\bu_{i},
	 \theta_{i}]} (\bv) ,
\end{align}
where $\bu_{i}$ and $\theta_{i}$ are the local mean velocity and temperature, respectively, so that the relations \cref{eq:moments-relation}
hold for the coefficients $f_{i,\alpha}$. 
Similarly, the numerical fluxes $\bar{F}_{i \pm 1/2} \bra{\cdot, \cdot}$ and the collision term $\mQ(f_{i})$ in \eqref{eq:x-fv-dis} should be expressed as functions in the same space $\mF_{M}^{[\bu_{i}, \theta_{i}]}$ as follows
\begin{align}
\label{eq:dis-expansion-F-Q}
	\begin{aligned}
		 \bar{F}_{i \pm 1/2} \bra{\cdot, \cdot} = \sum_{|\alpha| \leq M} \bar{F}
		_{i \pm 1/2,\alpha} \bra{\cdot, \cdot} \mH_{\alpha} ^{[\bu_{i},
		\theta_{i}]} (\bv), \quad
		 \mQ(f_{i})  = \sum_{|\alpha| \leq M} \mQ_{i,\alpha} \mH_{\alpha}
		^{[\bu_{i}, \theta_{i}]} (\bv),
	\end{aligned}
\end{align}
where $\bar{F}_{i \pm 1/2,\alpha}\bra{\cdot, \cdot}$ and $\mQ_{i,\alpha}$ are
the coefficients. Substituting the expansions \cref{eq:truncated-dis-i} and
\cref{eq:dis-expansion-F-Q} into the system \cref{eq:x-fv-dis} and matching
the coefficients of the same basis function $\mH_{\alpha}^{[\bu_{i},
\theta_{i}]} (\bv)$, 
we obtain the finite volume discretization for the hyperbolic moment system
\cref{eq:moment-system}.

In the present numerical experiments, a modified Harten, Lax and van Leer (HLL) flux, 
as used in \cite{Cai2013microflow}, is adopted for $\bar{F}_{i \pm 1/2} \bra{\cdot,
	\cdot}$.
In order to write the numerical flux in the form \cref{eq:dis-expansion-F-Q}, all distribution functions involved in \cref{eq:x-fv-dis}, including those in neighboring grid cells, need to be represented as members of $\mF_M^{[\bu_i,\theta_i]}$.
This requires a projection onto $\mF_{M}^{[\bu_{i}, \theta_{i}]}$ for any function in $\mF_{M}^{[\bvpi,\vartheta]}$ if $\bvpi\neq \bu_{i}$ or $\vartheta\neq \theta_{i}$, regardless of the choice of the numerical flux. An efficient implementation of the required projection, which preserves all moments up to the $M$th order, can be found in \cite{Hu2020}.

It remains to determine the distribution functions on cell boundaries
$f_{i}^{-}$ and $f_{i}^{+}$, which are to be constructed as functions in
$\mF_{M}^{[\bu_{i}^{-}, \theta_{i}^{-}]}$ and $\mF_{M}^{[\bu_i^{+},
\theta_{i}^{+}]}$. 
Obviously, it is equivalent to determine the parameters
$\bu_{i}^{\pm}$, $\theta_{i}^{\pm}$ and all coefficients $f_{i,\alpha}
^{\pm}$ with $\abs{\alpha} \leq M$. In our method, we first calculate the
reconstructed vectors of conservative variables on the cell boundaries as
\begin{align}
	\label{eq:recon-U}
	U_{i}^{+} = U_{i} + \frac{\Delta x}{2} \bs_{i}, \quad U_{i}^{-} = U_{i} -
	\frac{\Delta x}{2} \bs_{i},
\end{align}
where $U_{i} = [\rho_{i}, \rho_{i} \bu_{i}, \rho_{i} E_{i}]^{\mathsf{T}}$ is the vector
on the $i$th cell.
The parameters $\bu_i^\pm$ and $\theta_i^\pm$ are then extracted from
$U_i^\pm$.
Following \cite{Hu2019,Hu2023}, the coefficients are reconstructed as
\begin{align}
	\label{eq:recon-f-alpha}
	f_{i,\alpha}^{+} = f_{i,\alpha} + \frac{\Delta x}{2} d_{i, \alpha}, \quad &
    f_{i,\alpha}^{-} = f_{i,\alpha} - \frac{\Delta x}{2} d_{i, \alpha}. 
\end{align}
A first-order scheme can be obtained by setting all slopes $\bs_{i}$ and
$d_{i,\alpha}$ to $0$, 
while a second-order discretization is achieved by employing
\begin{align}
	\label{eq:slopes}
	\bs_{i} = \frac{U_{i+1} - U_{i-1}}{2 \Delta x}, \quad 
	d_{i,\alpha} = \frac{f_{i+1, \alpha} - f_{i-1, \alpha}}{2 \Delta x}.
\end{align}
Finally, appropriate boundary conditions are necessary to complete the discretization. The Maxwell boundary condition is employed throughout our numerical experiments.
We refer to \cite{Cai2012,Cai2013microflow} for details on such boundary
conditions.

\subsection{FIM framework}
\label{sec:alternating}
For convenience, the discretization \cref{eq:x-fv-dis} over the $i$th cell is rewritten in the residual form
\begin{align}
	\label{eq:res-HMS}
	\mR_i(\bbf)=0,
\end{align}
where $\mR_i(\bbf)$ denotes the local residual corresponding to the left-hand side of \cref{eq:x-fv-dis}, and $\bbf$ is the global vector composed of all cell-averaged distribution functions $f_i$.
The above discretization results in a nonlinear system coupling all unknowns, including $\bu_i$, $\theta_i$, and $f_{i,\alpha}$ for $i=0,1,\dots,N_x-1$ and $\abs{\alpha}\leq M$.
A simple relaxation method for solving this nonlinear system is the Richardson
iteration given by  
\begin{align}
	\label{eq:Richardson-HMS}
	f^{n+1}_{i} = f^{n}_{i} + \omega \mR_{i} \bra{\bbf^n},
\end{align}
where the superscript $n$ is the index for iteration steps. 
The relaxation parameter $\omega$ is selected according to the CFL condition
\begin{align}
	\label{eq:dt}
	\omega = \min_i \omega_i,
\end{align}
where the local relaxation parameter $\omega_i$ satisfies
\begin{align}
	\label{eq:dt-local}
	\omega_{i} \bra{\frac{|u_{i,1}^n| + C_{M+1} \sqrt{\theta_{i}^n}}{\Delta x}} < 1.
\end{align}
Here, $u_{i,1}^n$ is the first component of the vector $\bu_{i}^n$, and
$C_{M+1}$ is the maximal root of the Hermite polynomial of degree $M+1$.
Under Grad's expansion, the left- and right-hand sides of \cref{eq:Richardson-HMS} are represented in different function spaces, namely, $\mF_{M}^{[\bu_i^{n+1},\theta_i^{n+1}]}$ and $\mF_{M}^{[\bu_i^n,\theta_i^n]}$, respectively. Therefore, the evaluated right-hand side should be projected onto $\mF_{M}^{[\bu_i^{n+1},\theta_i^{n+1}]}$ using the projection mentioned in the previous subsection. The Richardson iteration \cref{eq:Richardson-HMS} can be interpreted as an explicit Euler scheme with the relaxation parameter $\omega$ serving as the pseudo-time step. As $\Kn$ decreases, its convergence becomes increasingly slower, resulting in poor efficiency in the near-continuum flow regime.

To accelerate the convergence, the FIM method proposed in \cite{Li2026} employs an alternating iteration framework consisting of a moment stage for the high-order moment system \cref{eq:moment-system} and a hydrodynamic stage for the hydrodynamic equations \cref{eq:HEs}.
Specifically, in the moment stage, $\gamma_1$ moment iterations are performed starting from the current moment solution $\bbf^n$ to obtain an intermediate approximation $\tilde{\bbf}$.
In the hydrodynamic stage, using the macroscopic quantities $\tilde{\rho}$, $\tilde{\bu}$, and $\tilde{\theta}$ extracted from $\tilde{\bbf}$ as the initial guess, at most $\gamma_2$ hydrodynamic iterations are performed. 
The resulting macroscopic solution is subsequently used to update $\tilde{\bbf}$ and construct the new moment solution $\bbf^{n+1}$.
The above procedure is summarized in \textbf{\Cref{alg:FIM}}. 
In \cite{Li2026}, the Richardson iteration \cref{eq:Richardson-HMS} or its symmetric Gauss-Seidel (SGS) version is used in the moment stage, while a similar iterative scheme with a consistent spatial discretization is employed in the hydrodynamic stage, as described below.

\begin{algorithm}[!htb]
	\renewcommand{\algorithmicrequire}{\textbf{Input:}}
	\renewcommand{\algorithmicensure}{\textbf{Output:}}
	\renewcommand{\algorithmicloop}{\textbf{begin}}
	\renewcommand{\algorithmicendloop}{\algorithmicend}
\caption{One step of the FIM alternating iteration \cite{Li2026} }
\label{alg:FIM}
\begin{algorithmic}[1]
	\REQUIRE Grid, $\bbf^{n}$
	\ENSURE The new approximation $\bbf^{n+1}$ denoted by $\bbf^{n+1}=\text{FIM}\bra{\bbf^{n}}$
	\vspace*{0.7em}
	\STATE Perform $\gamma_1$ moment iterations for
	\cref{eq:moment-system} to obtain a new approximation $\tilde{\bbf}$ from
	$\bbf^{n}$;
	\STATE Compute $\tilde{\rho}$, $\tilde{\bu}$ and $\tilde{\theta}$ on all grid
	cells from $\tilde{\bbf}$;
	\STATE Prepare hydrodynamic initial guess $\bU^{0}$ with $U_i^{0} =
	\left[\tilde{\rho}_{i}, \tilde{\rho}_{i} \tilde{\bu}_{i}, \frac{1}{2}
	\tilde{\rho_{i}} \bra{3\tilde{\theta}_i + \abs{\tilde{\bu}_i}^2} \right]
	^{\mathsf{T}}$ and set $m=0$;
	\STATE Calculate $\bU^{m+1}$ from $\bU^{m}$ by performing one hydrodynamic iteration for \cref{eq:HEs} 
	with compatible constitutive relations $\bsigma$, $\bq$ and boundary conditions.
	If $(m+1) \geq \gamma_2$ or $\norm{\bU^{m+1} - \bU^{m}}_1 \leq tol$, then
	break; otherwise, let $m:=m+1$, and repeat the current step;
	\STATE Extract $\rho^{m+1}$, $\bu^{m+1}$ and $\theta^{m+1}$ from $\bU^{m+1}$;  	
	\STATE Construct $\bbf^{n+1}$  from $\tilde{\bbf}$ by replacing $\tilde{\rho}$,
	$\tilde{\bu}$ and $\tilde{\theta}$ with $\rho^{m+1}$, $\bu^{m+1}$ and
	$\theta^{m+1}$, respectively.
\end{algorithmic}
\end{algorithm}

Let $U_i$ represent the cell average of $U$ over the $i$th grid cell $\mK_{i}$. 
Applying the finite volume method to \cref{eq:HEs} gives
\begin{align}
	\label{eq:res-HEs}
	\mR_{i}^{\text{hydro}} \bra{\bU} \coloneqq - \frac{1}{\Delta x}\left[ \hat{F}_{i+1/2} \bra{U_{i}^{+},
	U_{i+1}^{-}, \bsigma_{i}^{+}, \bsigma_{i+1}^{-}, \bq_{i}^{+}, \bq_{i+1}^{-} }
	- \hat{F}_{i-1/2} \bra{U_{i-1}^{+}, U_{i}^{-}, \bsigma_{i-1}^{+}, \bsigma_{i}
	^{-}, \bq_{i-1}^{+}, \bq_{i}^{-}} \right] = 0,
\end{align}
where $\mR_i^{\mathrm{hydro}}(\bU)$ denotes the local hydrodynamic residual, $\bU$ consists of all cell averages $U_i$, and $U_i^\pm$ are reconstructed according to \cref{eq:recon-U}.
Similarly, the interface values of the stress tensor $\bsigma$ and heat flux $\bq$ are reconstructed as
\begin{align}
	\label{eq:recon_sigma_q}
	\begin{aligned}
		\bsigma_{i}^{+}  = \tilde{\bsigma}_{i} + \frac{\Delta x}{2}
		\tilde{\bs_i}, \quad  \bsigma_{i}^{-}  = \tilde{\bsigma}_{i} -
		\frac{\Delta x}{2} \tilde{\bs_i}, \quad
		\bq_{i}^{+}  = \tilde{\bq}_{i} + \frac{\Delta x}{2}
		\bar{\bs_i}, \quad  \bq_{i}^{-}  = \tilde{\bq}_{i} -
		\frac{\Delta x}{2} \bar{\bs_i},
	\end{aligned}
\end{align}
where $\tilde{\bsigma}_{i}$ and $\tilde{\bq}_i$ are evaluated from the intermediate moment solution $\tilde{\bbf}$.
The slopes vanish for the first-order
discretization, while for the second-order discretization they are given by
\begin{align}
	\label{eq:slopes_sigma_q}
	\begin{aligned}
		\tilde{\bs_i} = \frac{\tilde{\bsigma}_{i+1} - \tilde{\bsigma}_{i-1}}
		{2\Delta x}, \quad \bar{\bs_i} = \frac{\tilde{\bq}_{i+1} - \tilde{\bq}
			_{i-1}}{2\Delta x},
	\end{aligned}
\end{align}
where $\tilde{\bsigma}_{i\pm1}$ and $\tilde{\bq}_{i\pm1}$ are also evaluated from $\tilde{\bbf}$ on the neighboring cells. 
The spatial discretization of the hydrodynamic equations is required to be consistent with that of the moment system, including the reconstruction procedure, numerical fluxes, and characteristic wave speeds used in the HLL flux.

In \cite{Li2026}, the boundary conditions for the hydrodynamic equations are obtained indirectly from the distribution-function boundary treatment of the moment system. In the present method, the ghost-cell values for the hydrodynamic equations are instead obtained directly from the explicit expressions given in \ref{sec:boundary-conditions}, thereby simplifying the boundary treatment.

The discrete system \cref{eq:res-HEs} can be solved using the Richardson
iteration given by
\begin{align}
	\label{eq:Richardson-HEs}
	U_i^{m+1} = U_i^{m} + \omega \mR_{i}^{\text{hydro}} \bra{\bU^m},
\end{align}
where the relaxation parameter $\omega$ is again determined by the CFL condition
given in \cref{eq:dt} and \cref{eq:dt-local}. Here, the superscript $m$
corresponds to the iterations for solving the hydrodynamic equations.
To improve convergence, the SGS method can be incorporated into
\cref{eq:Richardson-HEs}, yielding
\begin{align}
	\label{eq:SGS-HEs}
	U_i^{m+1} = U_i^{m} + \omega_i \mR_{i}^{\text{hydro}} \bra{\bU^*},
\end{align}
where $\bU^*$ is a vector 
initialized with $\bU^m$ and updated cell by cell as the new approximation $U_{i}
^{m+1}$ becomes available during the sweep. For convenience, this SGS iteration is referred to as the Hydro-SGS iteration hereafter.

Numerical experiments in \cite{Li2026} indicate that $\gamma_1$ is often set to $1$ or $2$ in practice, while $\gamma_2$ needs to be chosen appropriately to balance the per-iteration cost and the overall convergence rate. 
When $\Kn$ becomes small, $\gamma_2$ usually needs to be increased to tens or even hundreds to ensure that the hydrodynamic information is sufficiently captured.
Replacing the Richardson iteration \cref{eq:Richardson-HEs} with Hydro-SGS can improve the convergence of the hydrodynamic solver, but a relatively large $\gamma_2$ is still required.

\section{Convergence analysis of the FIM method}
\label{sec:fim3-convergence-analysis}

One hydrodynamic iteration \eqref{eq:Richardson-HEs} or \eqref{eq:SGS-HEs} is usually faster than one moment iteration due to the smaller number of unknowns. However, a large $\gamma_2$, which is often required at small $\Kn$ \cite{Li2026}, increases the computational cost of the hydrodynamic stage in each alternating iteration and may reduce the overall efficiency of the FIM method. In order to better understand this behavior and motivate further improvements, we analyze the convergence of the FIM method in the near-continuum regime based on a linearized one-dimensional BGK model. This simplified model retains the essential coupling between the primary macroscopic quantities and nonequilibrium moments while allowing an explicit local Fourier analysis. This section summarizes the main results, while the detailed derivations are provided in \ref{app:fim3-fourier-analysis}.

\subsection{Linearized model and the fast iterative method}
\label{subsec:fim3analysis-model-scheme}

Following \cite{Dong2024}, the moment equations for the linearized one-dimensional BGK model take the form 
\begin{equation}
 \mathbf{A}\frac{d\boldsymbol{u}}{dx}
 =\frac{1}{\varepsilon}\mathbf{L}\boldsymbol{u},
 \qquad
 \boldsymbol{u}=(u_0,u_1,\ldots,u_M)^{\mathsf{T}},
 \qquad
 \varepsilon=\mathrm{Kn},
 \label{eq:fim3analysis-model}
\end{equation}
where $M\geq3$ and
\begin{equation}
 \mathbf{A}=
 \begin{pmatrix}
 0&1&&&\\
 1&0&2&&\\
 &1&\ddots&\ddots&\\
 &&\ddots&\ddots&M\\
 &&&1&0
 \end{pmatrix},
 \qquad
 \mathbf{L}=\mathrm{diag}(0,0,0,-1,\ldots,-1).
 \label{eq:fim3analysis-matrices}
\end{equation}
The first three rows of \eqref{eq:fim3analysis-model}, describing the evolution of the primary macroscopic quantities, are
\begin{equation}
 \frac{d}{dx}\left(\mathbf{A}_h\boldsymbol{U}+\boldsymbol{b}\,u_3\right)
 =\boldsymbol{0},
 \label{eq:fim3analysis-hydroequation}
\end{equation}
where
\begin{equation}
 \boldsymbol{U}=(u_0,u_1,u_2)^{\mathsf{T}},
 \qquad
 \mathbf{A}_h=
 \begin{pmatrix}
 0&1&0\\
 1&0&2\\
 0&1&0
 \end{pmatrix},
 \qquad
 \boldsymbol{b}=\begin{pmatrix}0\\0\\3\end{pmatrix}.
 \label{eq:fim3analysis-hydroblock}
\end{equation}
With first-order spatial discretization using the HLL flux, the moment system \eqref{eq:fim3analysis-model} becomes
\begin{equation}
  \frac{\mathbf{A}}{2\Delta x}
  \left(\boldsymbol{u}_{j+1}-\boldsymbol{u}_{j-1}\right)
  -\frac{C_{M+1}}{2\Delta x}
  \left(\boldsymbol{u}_{j+1}-2\boldsymbol{u}_j
  +\boldsymbol{u}_{j-1}\right)
  =\frac{1}{\varepsilon}\mathbf{L}\boldsymbol{u}_j.
  \label{eq:fim3analysis-discrete-equation}
\end{equation}
The corresponding discretization of \eqref{eq:fim3analysis-hydroequation} is given by the first three rows of this system, with $u_3$ held fixed during the hydrodynamic stage.

We focus on the FIM method for \eqref{eq:fim3analysis-discrete-equation}, with SGS iterations applied to both the moment and hydrodynamic stages, and assume that the relaxation parameter $\omega$ is independent of $\varepsilon$. Define $\nu = \omega C_{M+1} / \Delta x$ with $0<\nu<1$, and let $\boldsymbol{u}_j^{(n)}$ be the approximate solution in the $j$th grid cell after $n$ FIM iterations. The next FIM iteration is defined by the following steps:
\begin{enumerate}
  \item Moment SGS: $\gamma_1$ implicit-explicit SGS (IMEX-SGS) iterations in the moment stage, defined by
  \begin{itemize}
    \item Initialize $\boldsymbol{u}_j^{(n,0)}=\boldsymbol{u}_j^{(n)}$, then perform the following two sweeps for each $s=0,1,\ldots,\gamma_1-1$:
    \item Forward sweep: in increasing order of $j$, solve for $\boldsymbol{u}_j^{(n,s+1/2)}$ from
      \begin{equation}
        \left(\mathbf{I}-\frac{\omega}{\varepsilon}\mathbf{L}\right)
        \boldsymbol{u}_j^{(n,s+1/2)}
        =(1-\nu)\boldsymbol{u}_j^{(n,s)}
        +\frac{\nu}{2}\left(\mathbf{I}+\frac{\mathbf{A}}{C_{M+1}}\right)
        \boldsymbol{u}_{j-1}^{(n,s+1/2)}
        +\frac{\nu}{2}\left(\mathbf{I}-\frac{\mathbf{A}}{C_{M+1}}\right)
        \boldsymbol{u}_{j+1}^{(n,s)};
        \label{eq:fim3analysis-kin-forward-state}
      \end{equation}
    \item Backward sweep: in decreasing order of $j$, solve for $\boldsymbol{u}_j^{(n,s+1)}$ from
      \begin{align}
        \left(\mathbf{I}-\frac{\omega}{\varepsilon}\mathbf{L}\right)
        \boldsymbol{u}_j^{(n,s+1)}
        =(1-\nu)\boldsymbol{u}_j^{(n,s+1/2)}
        +\frac{\nu}{2}\left(\mathbf{I}+\frac{\mathbf{A}}{C_{M+1}}\right)
        \boldsymbol{u}_{j-1}^{(n,s+1/2)}
        +\frac{\nu}{2}\left(\mathbf{I}-\frac{\mathbf{A}}{C_{M+1}}\right)
        \boldsymbol{u}_{j+1}^{(n,s+1)}.
        \label{eq:fim3analysis-kin-backward-state}
      \end{align}
  \end{itemize}
  \item Hydrodynamic SGS: $\gamma_2$ Hydro-SGS iterations in the hydrodynamic stage, defined by
  \begin{itemize}
    \item Set $\boldsymbol{U}_j^{(n,0)}=(u_{j,0}^{(n,\gamma_1)},u_{j,1}^{(n,\gamma_1)},u_{j,2}^{(n,\gamma_1)})^{\mathsf{T}}$, then perform the following two sweeps for each $\ell=0,1,\ldots,\gamma_{2}-1$:
    \item Forward sweep: in increasing order of $j$, solve for $\boldsymbol{U}_j^{(n,\ell+1/2)}$ from
      \begin{equation}
        \begin{aligned}
          \boldsymbol{U}_j^{(n,\ell+1/2)}=
          (1-\nu)\boldsymbol{U}_j^{(n,\ell)}
          +\frac{\nu}{2}\left(\mathbf{I}+\frac{\mathbf{A}_h}{C_{M+1}}\right)
          \boldsymbol{U}_{j-1}^{(n,\ell+1/2)}
          +\frac{\nu}{2}\left(\mathbf{I}-\frac{\mathbf{A}_h}{C_{M+1}}\right)
          \boldsymbol{U}_{j+1}^{(n,\ell)} 
          -\frac{\nu}{2C_{M+1}}\boldsymbol{b}
          \left(u_{j+1,3}^{(n,\gamma_1)}-u_{j-1,3}^{(n,\gamma_1)}\right);
        \end{aligned}
        \label{eq:fim3analysis-hydro-forward-state}
      \end{equation}
    \item Backward sweep: in decreasing order of $j$, solve for $\boldsymbol{U}_j^{(n,\ell+1)}$ from
      \begin{equation}
        \begin{aligned}
          \boldsymbol{U}_j^{(n,\ell+1)}=
          (1-\nu)\boldsymbol{U}_j^{(n,\ell+1/2)}
          +\frac{\nu}{2}\left(\mathbf{I}+\frac{\mathbf{A}_h}{C_{M+1}}\right)
          \boldsymbol{U}_{j-1}^{(n,\ell+1/2)}
          +\frac{\nu}{2}\left(\mathbf{I}-\frac{\mathbf{A}_h}{C_{M+1}}\right)
          \boldsymbol{U}_{j+1}^{(n,\ell+1)} 
          -\frac{\nu}{2C_{M+1}}\boldsymbol{b}
          \left(u_{j+1,3}^{(n,\gamma_1)}-u_{j-1,3}^{(n,\gamma_1)}\right).
        \end{aligned}
        \label{eq:fim3analysis-hydro-backward-state}
      \end{equation}
  \end{itemize}
  \item The solution after the $(n+1)$th FIM iteration is then given by
  \begin{equation}
    \boldsymbol{u}_j^{(n+1)}
    =\left(
    (\boldsymbol{U}_j^{(n,\gamma_2)})^{\mathsf{T}},
    u_{j,3}^{(n,\gamma_1)},u_{j,4}^{(n,\gamma_1)},\ldots,u_{j,M}^{(n,\gamma_1)}
    \right)^{\mathsf{T}}.
    \label{eq:fim3analysis-reconstruction}
  \end{equation}
\end{enumerate}
The results of the Fourier analysis of this numerical method are presented in the next subsection.

\subsection{Results of the Fourier analysis and low-frequency behavior}
\label{subsec:fim3analysis-amplification}

Define the error after $n$ FIM iterations by
\begin{equation}
 \delta\boldsymbol u_j^{(n)}
 =\boldsymbol u_j^{(n)}-\boldsymbol u_j,
 \label{eq:fim3analysis-outer-error}
\end{equation}
where $\boldsymbol{u}_j$ is the exact solution of
the discrete system \eqref{eq:fim3analysis-discrete-equation}. We neglect boundary effects associated with the sweeps and assume that the solution is periodic in space with period $2\pi$. Consider a Fourier mode of the form
\begin{equation}
 \delta\boldsymbol u_j^{(n)}
 =\widehat{\boldsymbol e}^{(n)}(k)e^{\mathrm i j k\Delta x},
 \qquad
 \widehat{\boldsymbol e}^{(n)}(k)\in\mathbb C^{M+1}.
 \label{eq:fim3analysis-full-fourier-error}
\end{equation}
Since the numerical scheme is linear, we define its Fourier error-propagation matrix $\mathbf T(\varepsilon,k)$ by
\begin{equation}
 \widehat{\boldsymbol e}^{(n+1)}(k)
 =\mathbf T(\varepsilon,k)\widehat{\boldsymbol e}^{(n)}(k),
 \qquad
 \mathbf T(\varepsilon,k)\in\mathbb C^{(M+1)\times(M+1)}.
 \label{eq:fim3analysis-fourier-iteration-matrix}
\end{equation}
The maximum magnitude of the amplification factors is given by the spectral radius of $\mathbf T(\varepsilon,k)$.

As shown in \ref{subsec:fim3app-leading-factor}, at most three eigenvalues of $\mathbf T(\varepsilon,k)$ have $O(1)$ leading terms in the near-continuum limit $\varepsilon\to0$. They are indexed by $\mu \in \{-\sqrt3,0,\sqrt3\}$, the eigenvalues of $\mathbf{A}_h$, and can be expanded as
\begin{equation}
 \Lambda_\mu(\varepsilon,k)
 =\Lambda_\mu^{(0)}(k)
 +\varepsilon\Lambda_\mu^{(1)}(k)
 +O(\varepsilon^2). 
 \label{eq:fim3analysis-eps-expansion}
\end{equation}
The remaining $M-2$ eigenvalues have magnitudes of $O(\varepsilon)$ and are therefore negligible in the near-continuum limit. For the three branches in \eqref{eq:fim3analysis-eps-expansion}, define
\begin{equation}
 \mathcal G_\mu^{(0)}(k)
 =\frac{
 \left[1-\nu+\dfrac{\nu}{2}\left(1-\dfrac{\mu}{C_{M+1}}\right)e^{\mathrm{i}k\Delta x}\right]
 \left[1-\nu+\dfrac{\nu}{2}\left(1+\dfrac{\mu}{C_{M+1}}\right)e^{-\mathrm{i}k\Delta x}\right]
 }{
 \left[1-\dfrac{\nu}{2}\left(1+\dfrac{\mu}{C_{M+1}}\right)e^{-\mathrm{i}k\Delta x}\right]
 \left[1-\dfrac{\nu}{2}\left(1-\dfrac{\mu}{C_{M+1}}\right)e^{\mathrm{i}k\Delta x}\right]
 }.
 \label{eq:fim3analysis-G}
\end{equation}
Their leading amplification factors are
\begin{equation}
 \Lambda_\mu^{(0)}(k)
 =\left[\mathcal G_\mu^{(0)}(k)\right]^{\gamma_1+\gamma_2},
 \qquad
 \mu\in\{-\sqrt3,0,\sqrt3\}.
 \label{eq:fim3analysis-lambda0}
\end{equation}
When $\gamma_1 = 1$, the first-order term $\Lambda_{\mu}^{(1)}(k)$ has also been derived. Due to its complexity, we refer readers to \ref{subsec:fim3app-first-order} (equation \eqref{eq:fim3app-lambda1-general}) for its expression.

It can be shown that $|\mathcal G_\mu^{(0)}(k)|$ attains its maximum value of $1$ at $k = 0$ for all $\mu\in\{-\sqrt3,0,\sqrt3\}$ (see Proposition \ref{prop:fim3app-G-modulus}). Meanwhile, when $|k\Delta x| < 2\pi/5$, we have $|\mathcal{G}_{0}^{(0)}(k)| \geq |\mathcal{G}_{\pm\sqrt3}^{(0)}(k)|$ (see Proposition \ref{prop:fim3app-G0-largest}). Therefore, we will mainly focus on the case $\mu = 0$ and small nonzero values of $|k\Delta x|$. In this case, we have
\begin{displaymath}
  \Lambda_0^{(0)}(k) = \exp\left( -\frac{2\nu}{2-\nu}(\gamma_1+\gamma_2)k^2\Delta x^2 \right) + O\left((k\Delta x)^4 \right).
\end{displaymath}
This result is consistent with the well-known smoothing property of the SGS method, indicating that the number of sweeps required to reduce a fixed low-frequency error mode by a constant factor should scale as $\Delta x^{-2}$ for fixed $\nu$. It also shows the benefit of the FIM method: when $\varepsilon$ is small, one can increase $\gamma_2$ while reducing $\gamma_1$, keeping their sum fixed, without affecting the leading-order convergence factor. This reduces the computational cost because the hydrodynamic stage is much less expensive than the moment stage.

However, this does not explain the observation in \cite{Li2026} that a larger $\gamma_2$ is required when the Knudsen number is smaller. To understand this, we consider the first-order term $\Lambda_0^{(1)}(k)$ for $\gamma_{1}=1$. Since $\Lambda_0^{(0)}(k) \geq 0$, we have
\begin{displaymath}
  |\Lambda_0(\varepsilon,k)|^2 = |\Lambda_0^{(0)}(k)|^2 + 2\varepsilon \Lambda_0^{(0)}(k) \operatorname{Re} \Lambda_0^{(1)}(k) + O(\varepsilon^2),
\end{displaymath}
and therefore, it suffices to study the real part of $\Lambda_0^{(1)}(k)$. Expanding \eqref{eq:fim3app-lambda1-general} for small nonzero $|k\Delta x|$ with $\gamma_2(k\Delta x)^2\ll1$, we obtain
\begin{equation}
  \operatorname{Re}\Lambda_0^{(1)}(k)
  =-\frac{2\nu}{C_{M+1}(2-\nu)^2}
  \left[(4-\nu)+2\gamma_2(2-\nu)\right]k^2\Delta x+O((k\Delta x)^4).
 \label{eq:fim3analysis-lambda1-low}
\end{equation}
Since the leading-order term of $\operatorname{Re}\Lambda_0^{(1)}(k)$ is negative for $0<\nu<1$ and $\Lambda_0^{(0)}(k)$ is positive, the amplification factor $|\Lambda_0(\varepsilon,k)|$ increases towards $|\Lambda_0^{(0)}(k)|$ as $\varepsilon$ decreases.
As a result, for a smaller Knudsen number, one needs to increase $\gamma_2$ to reduce $|\Lambda_0^{(0)}(k)|$ and make $\operatorname{Re}\Lambda_0^{(1)}(k)$ more negative in order to maintain the same convergence factor. This explains the observation in \cite{Li2026}.

\subsection{Motivation for an improved hydrodynamic solver}
\label{subsec:fim3analysis-multilevel-motivation}

The preceding analysis identifies the remaining bottleneck of the FIM method: the hydrodynamic stage aims to solve the system \eqref{eq:fim3analysis-hydroequation} with SGS iterations, but the number of iterations still depends strongly on both the Knudsen number and the grid resolution. To overcome these limitations, we seek a more efficient solver for the hydrodynamic stage that can achieve the effect of a large $\gamma_2$ in only a few iterations. Since the actual hydrodynamic system \eqref{eq:HEs} is nonlinear, we will apply Newton's method to achieve fast convergence. The linear system arising in each Newton iteration is then solved using a multigrid method to efficiently reduce errors over a broad range of spatial frequencies and mitigate the deterioration of convergence under grid refinement.

\section{Improved FIM method}
\label{sec:numerical-scheme}
Based on the considerations above, we now develop the improved FIM method for the full nonlinear high-order moment system \cref{eq:res-HMS}.
Before introducing the Newton-multigrid iteration for the hydrodynamic equations, we first present an IMEX scheme to address the stiffness of the Shakhov collision term in the moment stage.

\subsection{IMEX scheme for the Shakhov model} 
\label{sec:IMEX}
Since the collision term is not taken into account in determining the relaxation parameter $\omega$ by \cref{eq:dt}, the Richardson iteration \cref{eq:Richardson-HMS} may become unstable as the collision term becomes increasingly stiff with decreasing $\Kn$.
To alleviate this stiffness, we follow the IMEX framework in \cite{Li2026} to treat the Shakhov collision term implicitly and the convection term explicitly, which yields
\begin{align}
	\label{eq:SIS}
	f^{n+1}_{i} - \omega \mQ \bra{f_{i}^{n+1}} = f^{n}_{i} - \omega \mA_i
	\bra{\bbf^n},
\end{align}
where $\mA_i \bra{\bbf^n}$ denotes the remaining part of $-\mR_i(\bbf^n)$, namely,
the contribution from the spatial discretization of the convection term. Using the truncated expansions \cref{eq:truncated-dis-i} and \cref{eq:dis-expansion-F-Q}, the above equation can be written as
\begin{align}
	\label{eq:expansion-SIS}
	\sum_{|\alpha| \leq M} \bra{ f^{n+1}_{i, \alpha} - \omega \mQ^{n+1}
		_{i, \alpha}} \mH_{\alpha} ^{[\bu^{n+1}_{i}, \theta^{n+1}
		_{i}]} (\bv) = \sum_{|\alpha| \leq M} g_{i,\alpha} \mH_{\alpha}
	^{[\bu^{n}_{i}, \theta^{n}_{i}]} (\bv).
\end{align}
Here, $g_{i,\alpha}$ denotes the coefficient obtained from the right-hand side of \cref{eq:SIS}, while the Hermite expansion coefficients $\mQ_{i,\alpha}^{n+1}$ of the Shakhov collision term are given according to \ref{sec:expression_HME} by
\begin{align}
	\label{eq:Q-Shakhov}
	\mQ_{i,\alpha}^{n+1} = \left \{
	\begin{aligned}
			& 0 , &  & \text{if} \quad |\alpha| = 0,1, \\
			& \bar{\nu}_i^{n+1} \bra{\frac{1-\Pr}{5}q_{i,j}^{n+1} - f_{i,\alpha}^{n+1}},
			& & \text{if} \quad \alpha = 2e_k + e_j, \quad k,j = 1,2,3, \\
			& -\bar{\nu}_i^{n+1} f_{i,\alpha}^{n+1},  & & \text{otherwise}.
		\end{aligned}\right.
\end{align}
Multiplying both sides of \cref{eq:expansion-SIS} by $\left(1,\bv,\vert \bv \vert^{2}\right) ^\mathsf{T}$ and integrating over the velocity space, we obtain
\begin{align}
\label{eq:semi-implicit-u-theta-q}
\begin{aligned}
	& f_{i,0}^{n+1} = g_{i,0}, \quad \bu_{i}^{n+1} = \bu_{i}^{n} +
	\frac{1}{g_{i,0}}\left(g_{i,e_{1}}, g_{i,e_{2}},g_{i,e_{3}}\right)^\mathsf{T}, \quad \theta_{i}^{n+1} = \theta_{i}^{n} + \frac{2}{3}\sum\limits_{d=1}^{3}
	\frac{g_{i,2e_{d}}}{g_{i,0}} - \frac{\left\vert \bu_{i}^{n+1} - \bu_{i}
          ^{n}\right\vert^{2}}{3}.
\end{aligned}
\end{align}
The right-hand side of \cref{eq:expansion-SIS} is then projected onto the same function space $\mF_{M}^{[\bu_i^{n+1},\theta_i^{n+1}]}$ as the left-hand side, yielding
\begin{align}
	\label{eq:expansion-SIS-(n+1)}
	 \sum_{|\alpha| \leq M} \bra{ f^{n+1}_{i, \alpha} - \omega \mQ^{n+1}
	 _{i, \alpha}} \mH_{\alpha} ^{[\bu^{n+1}_{i}, \theta^{n+1}
	 _{i}]} (\bv) = \sum_{|\alpha| \leq M} \tilde{g}_{i,\alpha} \mH_{\alpha}
	 ^{[\bu^{n+1}_{i}, \theta^{n+1}_{i}]} (\bv),
\end{align}
with $\tilde{g}_{i,\alpha}$ denoting the coefficients after projection.
Substituting \cref{eq:Q-Shakhov} into \cref{eq:expansion-SIS-(n+1)} and matching the coefficients gives
\begin{align}
	 \label{eq:update-Shakhov} 
	 f_{i,\alpha}^{n+1} = \left \{
	 \begin{aligned}
		 & \tilde{g}_{i, \alpha}, & & \text{if} \quad |\alpha| = 0,1, \\
		 & \frac{1}{1 + \omega \bar{\nu}_{i}^{n+1}} \bra{ \tilde{g}_{i,\alpha} +
		 \omega \bar{\nu}_{i}^{n+1}\frac{1 - \Pr}{5} q_{i,j}^{n+1}},  & &
		 \text{if} \quad \alpha = 2e_k + e_j, \quad k,j = 1,2,3, \\
		 & \frac{1}{1 + \omega \bar{\nu}_{i}^{n+1}} \tilde{g}_{i,\alpha} ,  & &
		 \text{otherwise},
	 \end{aligned}\right.
\end{align}
where the updated heat flux is obtained from \cref{eq:moments-relation} as
\begin{equation*}
  q_{i,j}^{n+1} =\frac{2\tilde{g}_{i,3e_j}+\sum_{d=1}^3\tilde{g}_{i,2e_d+e_j}}
  {1+\omega\bar{\nu}_{i}^{n+1}\Pr}.
\end{equation*}

One step of the above IMEX scheme is summarized in
\textbf{\Cref{alg:semi-imp}}.
The implicit treatment of the Shakhov collision term improves the numerical stability of the moment iteration, whereas the computational cost per iteration remains essentially the same as that of the Richardson iteration
\cref{eq:Richardson-HMS}.

\begin{algorithm}[!htb]
  \renewcommand{\algorithmicrequire}{\textbf{Input:}}
  \renewcommand{\algorithmicensure}{\textbf{Output:}}
  \renewcommand{\algorithmicloop}{\textbf{begin}}
  \renewcommand{\algorithmicendloop}{\algorithmicend}
	\caption{One step of the IMEX iteration for \cref{eq:res-HMS} with the Shakhov model}
  \label{alg:semi-imp}
  \begin{algorithmic}[1]
	  \REQUIRE Grid, $\bbf^{n}$
	  \ENSURE The new approximation $\bbf^{n+1}$ denoted by $\bbf^{n+1}
	  =\text{IMEX}\bra{\bbf^{n}}$
    \vspace*{0.7em}
	\STATE Determine the relaxation parameter $\omega$ from ${\bbf}^{n}$ by \cref{eq:dt};
	\STATE On each grid cell, calculate the right-hand side of \cref{eq:SIS} and denote it by $g_{i}$; 
        \STATE Compute $\bu_{i}^{n+1}$ and $\theta_{i}^{n+1}$ from $g_{i}$ by \cref{eq:semi-implicit-u-theta-q};
	\STATE Project $g_{i}$ onto $\mF_{M}^{[\bu_{i}^{n+1}, \theta_{i}^{n+1}]}$ to obtain $\tilde{g}_{i}$;
	\STATE Obtain the new approximation ${f}_{i}^{n+1}$ from $\tilde{g}_{i}$ by \cref{eq:update-Shakhov};
	\end{algorithmic}
\end{algorithm}

The IMEX scheme \cref{eq:SIS}
  can be naturally reformulated as an IMEX-SGS iteration, with
  cellwise update given by
\begin{align}
	\label{eq:SISGS}
	f^{n+1}_{i} - \omega_i \mQ \bra{f_{i}^{n+1}} = f^{n}_{i} - \omega_i \mA_i
	\bra{\bbf^*},
\end{align}
where the local relaxation parameter $\omega_i$ replaces the global one, and $\bbf^*$ is initialized as $\bbf^n$. The $i$th component of $\bbf^{*}$ is immediately replaced by $f_i^{n+1}$ once the updated value becomes available during the sweeping process.

\subsection{Newton-multigrid iteration for hydrodynamic equations}
\label{eq:newton}
To further enhance the efficiency and robustness of the FIM method, the time-integration-based schemes for the hydrodynamic equations, namely \cref{eq:Richardson-HEs} and its SGS version, Hydro-SGS \cref{eq:SGS-HEs}, are replaced by a Newton-multigrid method.
Following \cite{Li2007,Hu2010,Wu2016multigrid}, the Newton iteration leads to the linearization
\begin{align}
	\label{eq:Newton-ite}
	\sum_{j \in \mI(i)} \pd{\mR_{i}^{\text{hydro}} \bra{\bU^{m}}}{U_{j}} \delta
	U_{j}^{m} = - \mR_{i}^{\text{hydro}} \bra{\bU^{m}},
\end{align}
where $\mI(i)$ is the set of the indices for the cells related to the
computation of the local residual $\mR_{i}^{\text{hydro}} \bra{\bU^{m}}$. 
Once the Newton correction $\delta U_i^m$ is obtained, the solution is updated by
\begin{align}
	\label{eq:Newton-update}
	U_{i}^{m+1} = U_{i}^{m} + \tau \delta U_{i}^{m},
\end{align}
where $\tau$ is a relaxation parameter often set as $1$.
The closed-form expression of the Jacobian 
$\partial{\mR_{i}^{\text{hydro}}} / \partial{U_j}$ 
can be obtained by the chain rule, but the result is usually complicated and depends on the choice of the numerical flux and the implementation of boundary conditions. Therefore, instead of deriving the Jacobian analytically, a numerical differentiation method is applied to evaluate the entries of the Jacobian matrix as
\begin{align}
	\label{eq:numerical-diff}
	& \pd{\mR_{i,l}^{\text{hydro}} \bra{\bU^{m}}}{U_{j,k}} \approx \frac{1}
	{\epsilon} \bra{\mR_{i,l}^{\text{hydro}} \bra{\bU^{m} + \epsilon \be} - 
	\mR_{i,l}^{\text{hydro}} \bra{\bU^{m}}},
\end{align}
where
$\partial{\mR_{i,l}^{\text{hydro}}} / \partial{U_{j,k}}$ 
denotes the derivative of the
$l$th component of the vector
$\mR_{i}^{\text{hydro}}$ 
with respect to the $k$th component of the vector
$U_{j}$. The unit vector $\be$ is of the same length as $\bU$, with its only
non-zero entry positioned where $U_{j,k}$ appears in $\bU$.
The parameter $\epsilon$ is given by $\epsilon =  \epsilon_1 U_{j,k}^{m} +
\epsilon_2$, where the scale factor $\epsilon_1=10^{-6}$ and the small perturbation
$\epsilon_2 = 10^{-8}$.

Since the Jacobian matrix may be singular, a suitable regularization is necessary to improve the robustness of the linear solver.
Unlike the problems considered in \cite{Li2007,Hu2010,Wu2016multigrid}, the hydrodynamic system \cref{eq:HEs} is not closed by itself but relies on the constitutive variables $\bsigma$ and $\bq$ provided by the high-order moment solution. These variables depend implicitly on the Knudsen number $\Kn$, and the resulting hydrodynamic system therefore inherits this dependence. Accordingly, in addition to the regularization based on the $\ell^2$ norm of the local residual, $\Kn$ is also incorporated into the regularization, leading to the following modified Newton iteration
\begin{align}
	\label{eq:regularize}
	\eta \bra{\Kn + \norm{\mR_{i}^{\text{hydro}} \bra{\bU^{m}}}_{l^2}} \delta U_{i}
	^{m} + \sum_{j \in \mI(i)} \pd{\mR_{i}^{\text{hydro}} \bra{\bU^{m}}}{U_{j}}
	\delta U_{j}^{m} = - \mR_{i}^{\text{hydro}} \bra{\bU^{m}},
\end{align}
where $\eta$ is a positive constant parameter and is set to $4$ in the numerical tests unless otherwise specified.

Combining \cref{eq:regularize} on all cells, a linear system for all the
unknowns over the computational domain can be rewritten as
\begin{align}
	\label{eq:linear-system}
	\bJ \delta \bU^{m} = - \bR^{m},
\end{align}
where $\bJ$ is the regularized Jacobian matrix, and $\delta \bU^{m}$ and $\bR^{m}$ are
vectors of $\delta U_{i}^{m}$ and $\mR_{i}^{\text{hydro}} \bra{\bU^{m}}$, respectively. 
To solve the linear system \cref{eq:linear-system} efficiently, the geometric multigrid method \cite{Li2007,Hu2011} is adopted, with the block LU-SGS method \cite{Chen2000} employed as the smoother on each grid level to damp high-frequency errors. A $V$-cycle iteration with two pre- and post-smoothing steps is used in the implementation.
The total number of grid levels is chosen such that the coarsest grid consists of $8$ cells in the one-dimensional case and $8\times8$ cells in the two-dimensional case \cite{Hu2023}.
With such a coarsest grid, the substantially reduced number of degrees of freedom (DOF) allows the problem to be solved accurately at a relatively low computational cost.

One step of the Newton-multigrid method is
summarized in \textbf{\Cref{alg:Newton-ite}}. 
The geometric multigrid method is employed as the inner linear solver in Step~2, where only a few multigrid iterations are performed. In the present work, two multigrid iterations per Newton iteration are used to balance computational cost and convergence efficiency.

\begin{algorithm}[!htb]
  \renewcommand{\algorithmicrequire}{\textbf{Input:}}
  \renewcommand{\algorithmicensure}{\textbf{Output:}}
  \renewcommand{\algorithmicloop}{\textbf{begin}}
  \renewcommand{\algorithmicendloop}{\algorithmicend}
	\caption{One step of the Newton-multigrid iteration for the hydrodynamic equations}
	\label{alg:Newton-ite}
	\begin{algorithmic}[1]
		\REQUIRE Grid, $\bU^{m}$
		\ENSURE The new approximation $\bU^{m+1}$ denoted by $\bU^{m+1}$ =
		Newton $\bra{\bU^{m}}$
		\vspace*{0.7em}
		\STATE Construct the regularized Jacobian matrix $\bJ$ using
		\cref{eq:numerical-diff,eq:regularize};
		\STATE Solve the linear system \cref{eq:linear-system} using the geometric multigrid method
		to obtain $\delta \bU^{m}$;
		\STATE Update $U_{i}^{m+1}$ using $U_{i}^{m+1} = U_{i}^{m} +
		\tau \delta U_{i}^{m}$ on every cell;
	\end{algorithmic}
\end{algorithm}

\begin{remark}
    As shown in the numerical experiments, one step of the Newton-multigrid iteration is generally sufficient to provide an effective correction to the primary macroscopic quantities within the FIM framework.
    Accordingly, $\gamma_2=1$ is adopted as the default choice in \textbf{\Cref{alg:FIM}}, which avoids case-dependent tuning of $\gamma_{2}$
    and contributes to reducing the overall computational time while maintaining fast convergence across all considered Knudsen numbers.
\end{remark}

\begin{remark}
  The Newton-multigrid method presented in \textbf{\Cref{alg:Newton-ite}} provides efficient spatial acceleration for the hydrodynamic equations. Incorporated into the alternating iteration framework, it yields an integrated multilevel solver that combines both velocity and spatial coarse-level corrections through the outer FIM iteration and the inner Newton-multigrid iteration, respectively, providing an alternative to the integrated multilevel construction in \cite{Li2026}.
\end{remark}

\subsection{FIM solvers}
\label{sec:FIM3}

Incorporating the improvement strategies introduced above, we
summarize a family of FIM solvers in \cref{tab:fim-solvers}, in which different
numerical schemes are adopted for the moment system and the hydrodynamic
equations.
When the Prandtl number is set to $\Pr=1$ and the first-order discretization is
adopted,
the FIM-1, FIM-2 and FIM-3 in \cref{tab:fim-solvers} 
actually correspond to the FIM solvers investigated in \cite{Li2026}, among which FIM-3 was shown to achieve the highest efficiency. 
Therefore, in the present study for the Shakhov model, we mainly focus on FIM-3 and FIM-4 with both first- and second-order discretization, and evaluate their performance through a series of numerical experiments.

\begin{table}[!htb]
  \tabcolsep=0.2cm
  \centering
  \caption{FIM solver setup.}
  \label{tab:fim-solvers}
  \begin{tabular}{ccc}
    \toprule
    \multirow{2}*{\quad FIM solver\quad} & \multicolumn{2}{c}{Schemes} \\
	\cline{2-3}
     & High-order moment system & Hydrodynamic equations \\
    \midrule
    FIM-1 & Richardson iteration \cref{eq:Richardson-HMS} & Richardson iteration
	  \cref{eq:Richardson-HEs} \\
	 FIM-2 & IMEX \cref{eq:SIS} & Richardson iteration \cref{eq:Richardson-HEs} \\
    FIM-3 & IMEX-SGS \cref{eq:SISGS} & Hydro-SGS \cref{eq:SGS-HEs} \\
	 FIM-4 & IMEX-SGS \cref{eq:SISGS} & Newton-multigrid method \cref{eq:linear-system,eq:Newton-update} \\ 
    \bottomrule
  \end{tabular}
\end{table}

\section{Numerical experiments}
\label{sec:experiments}
	
In this section, three numerical examples, i.e., the planar Couette flow, the lid-driven cavity flow and the heat transfer in a bottom-heated cavity, are carried out to assess the accuracy, efficiency, and robustness of the proposed first- and second-order FIM solvers. Throughout the experiments, argon gas is considered 
and the dimensionless collision frequency $\bar{\nu}$ is given by
\begin{align*}
	 \bar{\nu} = \sqrt{\frac{\pi}{2}} \frac{1}{\Kn} \rho \theta^{1 - w},
\end{align*}
where the viscosity index is $w=0.81$. 
Three Knudsen numbers in the near-continuum flow regime, namely, $\Kn=0.1$, $0.01$, and $0.001$, are considered. Following \cite{Li2026}, moment systems of orders $M=10$ and $9$ are employed for $\Kn=0.1$, $M=8$ and $7$ for $\Kn=0.01$, and $M=6$ and $5$ for $\Kn=0.001$, which have been shown to provide sufficiently accurate approximations for the corresponding Knudsen numbers.
The steady-state tolerance is set to $10^{-8}$ using the stopping criterion in \cite{Li2026}, and $\gamma_{1}=1$ is adopted for the FIM solvers. The CFL number for choosing $\omega_i$ in \cref{eq:dt-local} is $0.8$, and the Prandtl number is $\Pr=2/3$. 
Since the Maxwell boundary condition used here does not determine a unique steady state, the mass-conservation correction in \cite{Hu2014nmg,Hu2023} is also applied.

All numerical experiments are conducted using a C++ program on Intel Xeon Gold 6248 CPUs at 2.5 GHz. The two-dimensional computations are performed in parallel using four CPU threads.

\subsection{The planar Couette flow}
\label{sec:couette}

The first example is the planar Couette flow, which is a frequently used one-dimensional benchmark test \cite{Hu2014nmg,Zhu2016,Zhu2021}. For convenience, the same dimensionless settings as in \cite{Cai2013microflow,Hu2014nmg} are adopted in
our tests. The gas lies between two infinite parallel plates with a distance $L = 1$. These two plates are maintained at the dimensionless temperature $\theta^{W} = 1$ and move tangentially with a relative speed $u^W = 1.2577$. The computation starts at the global Maxwellian with $\rho^{0} = 1$, $\bu^{0} = 0$, and $\theta^{0} = 1$.
The gas is driven by the motion of the plates and eventually reaches a steady state.

\subsubsection{Solution validation and spatial accuracy verification}
\label{sec:couette-sol}

The profiles of the dimensionless density $\rho$, temperature $\theta$, vertical velocity $u_2$ and heat flux $q_1$, which are computed by the second-order FIM solver on a uniform grid with $N_x = 128$, are plotted in \cref{fig:couette_solutions}. The DUGKS solutions obtained by the dugksFoam solver \cite{Zhu2017foam} are provided as the references.
In the dugksFoam solver, the spatial domain of the Boltzmann-Shakhov equation is discretized using a second-order finite volume scheme, while the velocity space is represented by dimension-reduced Gauss-Hermite discrete velocity points. For the reference DUGKS computations, 100 physical-space cells are used for all three Knudsen numbers, together with $81\times81$ discrete velocity points for $\Kn=0.1$, and $28\times28$ discrete velocity points for $\Kn=0.01$ and $0.001$.
It is shown that our results agree well with the references at all three Knudsen numbers, confirming the validity of the high-order moment system.

In order to verify the spatial order of accuracy, the FIM solver is run on a sequence of uniform grids with $N_x$ ranging from $64$ to $4096$. 
Since the exact solution cannot be obtained analytically, the numerical solution obtained on the finest grid with $N_x=4096$ for each spatial discretization is taken as the benchmark solution. 
On a computational grid consisting of $N_x$ uniform cells, the $L^2$ error is defined by
\begin{align*}
	\text{Error} = \sqrt{\Delta x \sum_{i=0}^{N_x-1} \bra{\bar{m}_i - \bar{m}
	_i^{\text{finest}}}^2},
\end{align*}
where $\Delta x=1/N_x$, $\bar m_i$ denotes the cell average of either $\rho$ or $\theta$ over the $i$th computational cell, and $\bar m_i^{\mathrm{finest}}$ denotes the corresponding cell average from the benchmark solution.
The errors of the first-order and second-order schemes are displayed in \cref{fig:couette_error}, demonstrating that the expected orders of convergence are achieved for all considered Knudsen numbers.

\begin{figure}[!htb]
	 \centering
	 \subfloat[Density, $\rho$]{
		  \includegraphics[width=0.42\textwidth]{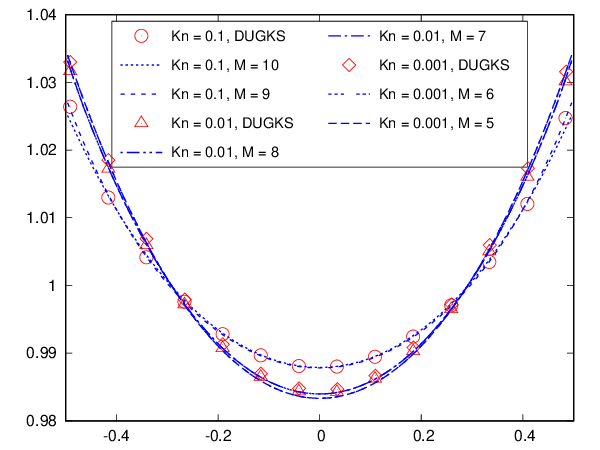}
	 }\hfill
	 \subfloat[Temperature, $\theta$]{
	      \includegraphics[width=0.42\textwidth]{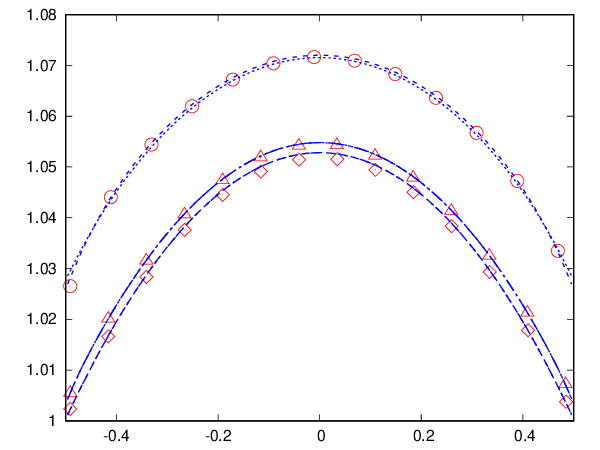}
	 }\\ 
	 \subfloat[Velocity, $u_2$]{
		  \includegraphics[width=0.42\textwidth]{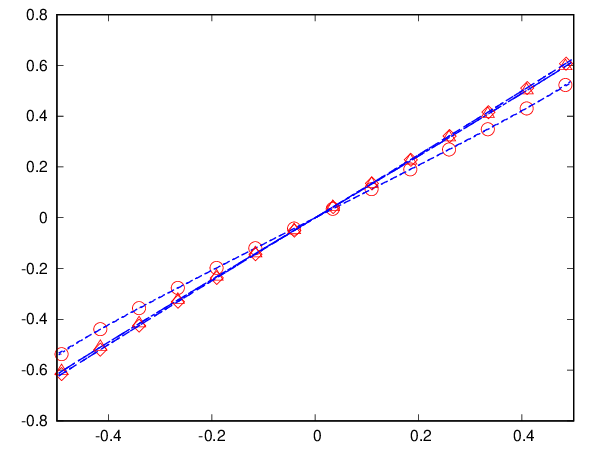}
	 }\hfill
	 \subfloat[Heat flux, $q_1$]{
	      \includegraphics[width=0.42\textwidth]{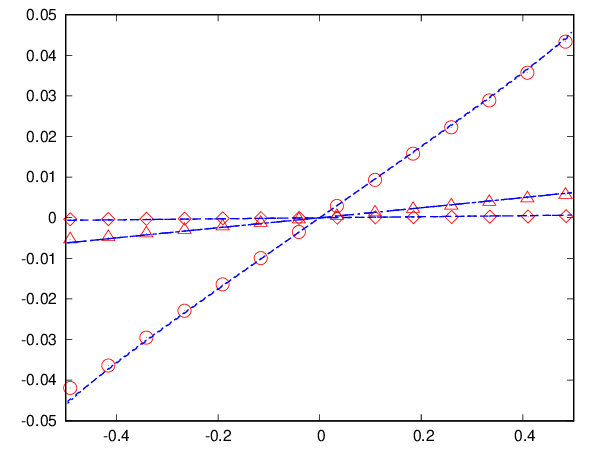}
	 }
	 \caption{Numerical solutions of the planar Couette flow with $N_x = 128$.}
	 \label{fig:couette_solutions}
\end{figure}

\begin{figure}[!htb]
	 \centering
	 \subfloat[First-order scheme]{
		 \includegraphics[width=0.42\textwidth]{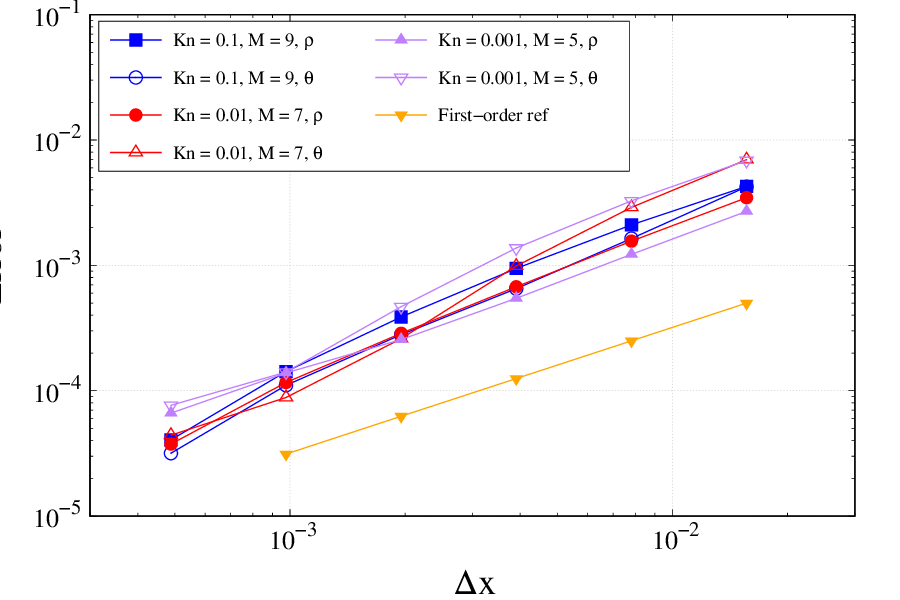}
	 }\hfill
	 \subfloat[Second-order scheme]{
		 \includegraphics[width=0.42\textwidth]{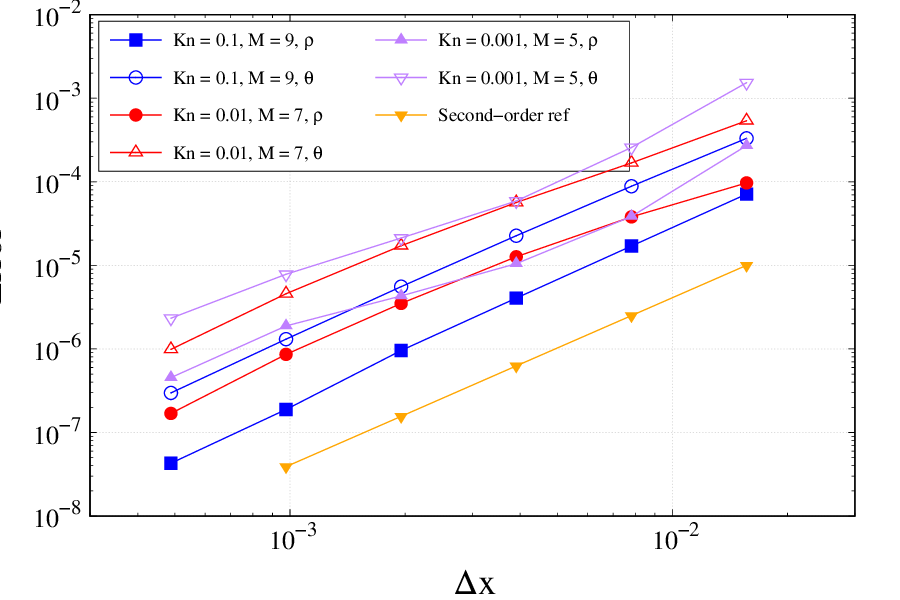}
	 }
	 \caption{The errors of the FIM method with various $\Kn$ for planar Couette
	 flow.}
	 \label{fig:couette_error}
\end{figure}

\subsubsection{Performance of first-order FIM solvers}
\label{sec:couette-eff-1st}

We begin by examining the first-order FIM solvers, aiming to compare their performance against that reported in \cite{Li2026}.
The total number of iterations and the associated wall-clock time spent by FIM-3 and FIM-4, as well as the IMEX-SGS iteration, to reach steady state on the grid with $N_x=256$ across three Knudsen numbers are listed in \cref{tab:couette-1st}.
It is observed that the convergence of the IMEX-SGS iteration deteriorates significantly as $\Kn$ decreases toward the continuum regime, leading to a substantial increase in the total number of iterations. However, the variation in wall-clock time is much less pronounced due to the use of lower-order moment systems for smaller $\Kn$.

As expected, both FIM-3 and FIM-4 improve the computational efficiency over the IMEX-SGS iteration, with FIM-4 consistently delivering the best performance. The values of $\gamma_2$ listed in \cref{tab:couette-1st} have been shown in \cite{Li2026} to provide good performance for the corresponding Knudsen numbers. The increase in the preferred value of $\gamma_2$ as the Knudsen number decreases is consistent with the Fourier analysis in \cref{sec:fim3-convergence-analysis}. The convergence histories in \cref{fig:couette-his-1st} further show that increasing $\gamma_2$ can effectively accelerate the convergence of FIM-3, although the additional improvement becomes limited for larger values.
Meanwhile, inappropriate choices of $\gamma_2$ and the CFL number may introduce minor oscillations in the residual decay, as also reported in \cite{Li2026}. Therefore, selecting $\gamma_2$ requires a compromise between convergence efficiency and robustness, which relies on empirical tuning.
In contrast, FIM-4 replaces repeated hydrodynamic iterations with one Newton-multigrid step in each alternating iteration, thereby avoiding the tuning of $\gamma_2$. Even when compared with the carefully tuned FIM-3, FIM-4 further reduces the wall-clock time by factors ranging from approximately $1.4$ to $4.0$. The convergence histories in \cref{fig:couette-his-1st} demonstrate that FIM-4 achieves a faster and more stable residual decay with the lowest computational cost. These results indicate that one Newton-multigrid step is sufficient to provide an effective hydrodynamic correction within the FIM framework.

\begin{table}[!htb]
\centering
\caption{Performance results of the IMEX-SGS, FIM-3 and FIM-4 for the Couette flow
	at various $\Kn$ with $N_x = 256$ and first-order discretization.}
\label{tab:couette-1st}
\begin{tabular}{ccccccccccc}
\toprule
\multirow{2}{*}{$\Kn$} & \multirow{2}{*}{$M$} & \multicolumn{2}{c}{IMEX-SGS} &
\multicolumn{2}{c}{FIM-3} & \multicolumn{2}{c}{FIM-3\textsuperscript{$\dagger$}} & \multicolumn{2}{c}
{FIM-4} & \multirow{2}{*}{Speedup \textsuperscript{$\ddagger$}} \\
\cmidrule(lr){3-4} \cmidrule(lr){5-6} \cmidrule(lr){7-8} \cmidrule(lr){9-10}
	&  & Steps & Time (s) & Steps & Time (s) & Steps &
	Time (s) & Steps & Time (s) \\
\midrule
0.1	& 10 & 7327 & 1165.12 & 3633 & 537.74 & 664 & 162.38 & 632 & 107.46 & 5.0 / 1.5 \\
		& 9  & 6923	& 839.29	& 3451 &	392.83 &	643 & 118.45 &	629 &	83.91 & 4.7 / 1.4 \\
0.01 	& 8 & 20038 & 1875.70 & 10161 & 983.36 & 95 & 28.27 & 87 & 9.13 & 107 / 3.1 \\		
		& 7 & 18796 & 1266.87 & 9557 & 666.95 & 88 & 16.11 & 81 & 6.42 &  103 / 2.5 \\
0.001 & 6 & 36192 & 1667.87 & 18364 & 875.84 & 52 & 15.63 & 77 & 4.51 & 194 / 3.4 \\		
		& 5 & 34729 & 1081.79 & 17679 & 579.94 & 51 & 13.72 & 78 & 3.42 & 169 / 4.0\\
\bottomrule
\end{tabular}
\smallskip
\begin{tabular}{@{}p{\linewidth}@{}}
  \textsuperscript{$\dagger$} FIM-3 uses $\gamma_2=1$,
  whereas FIM-3\textsuperscript{$\dagger$} uses $\gamma_2=40$, $400$,
  and $800$ for $\Kn=0.1$, $0.01$, and $0.001$, respectively. \\
  \textsuperscript{$\ddagger$} The speedups are the acceleration ratios of
  FIM-4 relative to FIM-3 with $\gamma_2=1$ and with increased $\gamma_2$, respectively.
\end{tabular}
\end{table}

\begin{figure}[!htb]
	 \centering
	 \subfloat[$\Kn=0.1$]{
		 \label{fig:couette-his-1st-kn-pt1}
		 \includegraphics[width=0.3\textwidth]{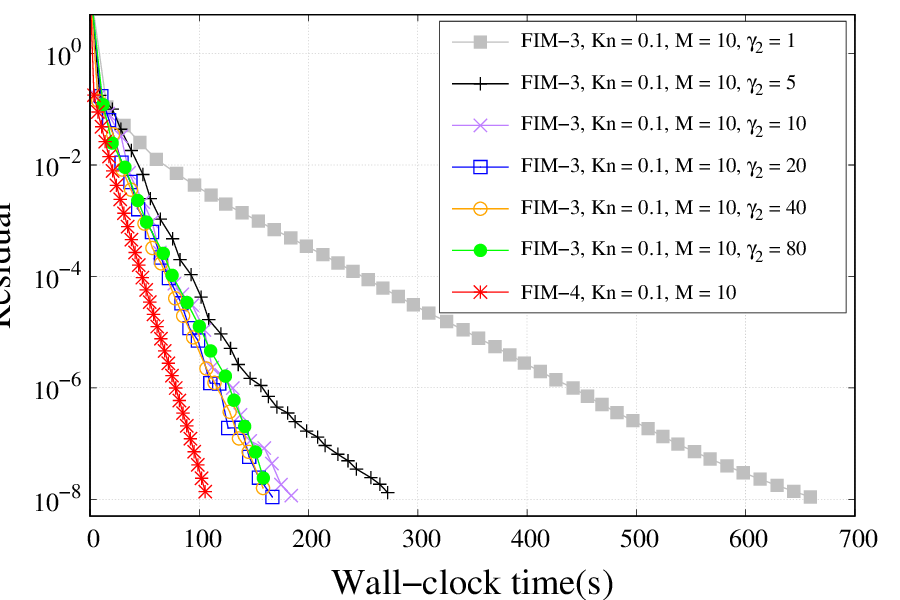}
	 }\hfill
	 \subfloat[$\Kn=0.01$]{
	    \label{fig:couette-his-1st-kn-pt01}
	    \includegraphics[width=0.3\textwidth]{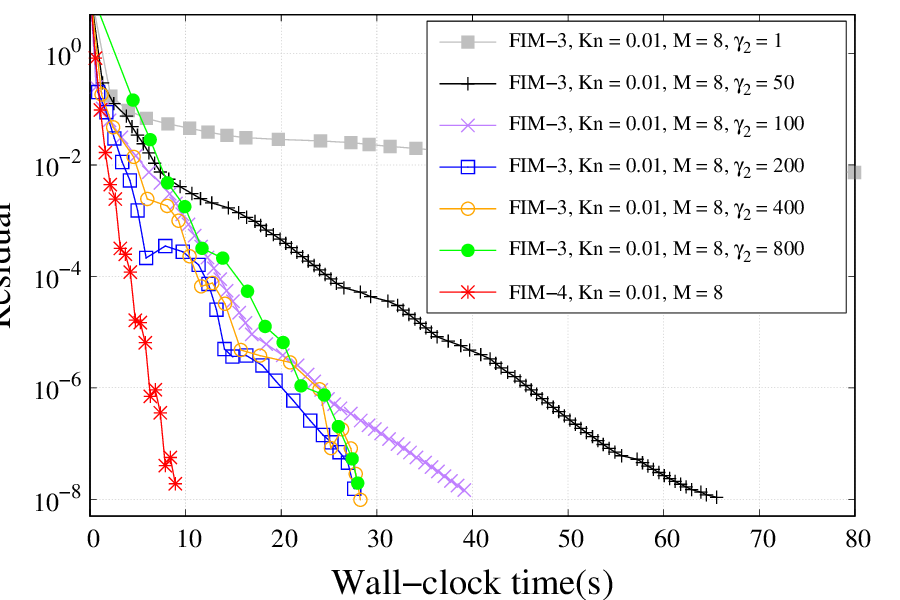}
	 }\hfill
	 \subfloat[$\Kn=0.001$]{
	    \label{fig:couette-his-1st-kn-pt001}
	    \includegraphics[width=0.3\textwidth]{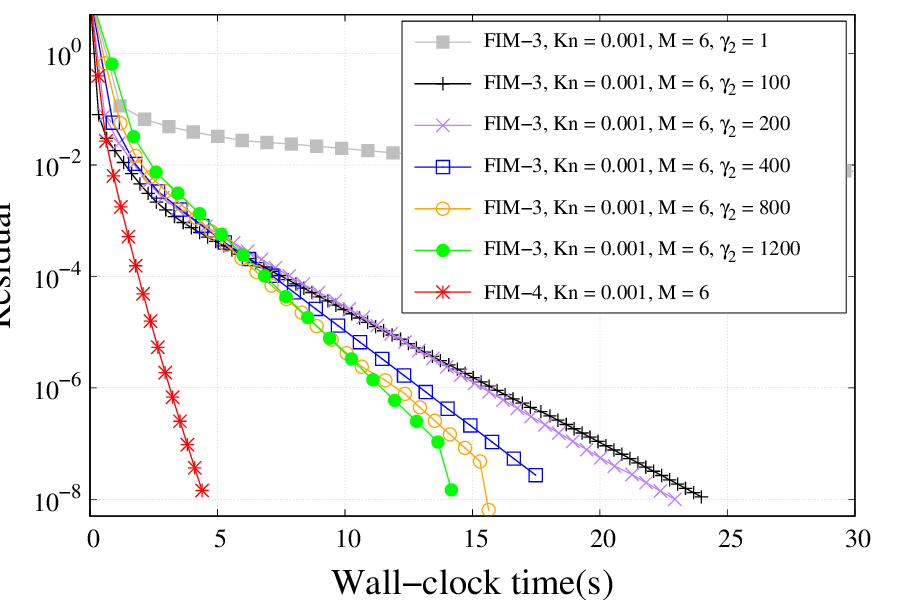}
	 }
	 \caption{Convergence histories for the Couette flow with $N_x=256$ and
	 first-order spatial discretization.}
	 \label{fig:couette-his-1st}
\end{figure}

\subsubsection{Performance of second-order FIM solvers}
\label{sec:couette-eff-2nd}

Now let us turn to investigating the performance of the second-order FIM solvers. 
To achieve a more stable residual decay, the parameters $\tau$ in \cref{eq:Newton-update} and $\eta$ in \cref{eq:regularize} are adjusted to $0.5$ and $10$, respectively, for FIM-4 at $\Kn=0.1$, while the default settings are retained for smaller Knudsen numbers.
The total number of iterations and the wall-clock time of the IMEX-SGS iteration, FIM-3 and FIM-4 on the grid with $N_x = 256$ are presented in \cref{tab:couette-2nd}. 
Compared with the first-order results in \cref{tab:couette-1st}, all three solvers converge more slowly with second-order discretization. This effect is particularly pronounced for the IMEX-SGS iteration at small $\Kn$, resulting in a substantial increase in both iteration numbers and wall-clock time.
Nevertheless, even with $\gamma_2=1$, FIM-3 still provides noticeable acceleration, with both the iteration numbers and wall-clock time reduced by roughly half compared with the IMEX-SGS iteration.

\begin{table}[!htb]
	\centering
	\caption{Performance results of the IMEX-SGS, FIM-3 and FIM-4 for the Couette flow
		at various $\Kn$ with $N_x = 256$ and second-order discretization.}
	\label{tab:couette-2nd}
	\begin{tabular}{ccccccccccc}
		\toprule
		\multirow{2}{*}{$\Kn$} & \multirow{2}{*}{$M$} & \multicolumn{2}{c}{IMEX-SGS} &
		\multicolumn{2}{c}{FIM-3} & \multicolumn{2}{c}{FIM-3\textsuperscript{$\dagger$}} & \multicolumn{2}{c}
		{FIM-4} & \multirow{2}{*}{Speedup \textsuperscript{$\ddagger$}} \\
		\cmidrule(lr){3-4} \cmidrule(lr){5-6} \cmidrule(lr){7-8} \cmidrule(lr){9-10}
		&  & Steps & Time (s) & Steps & Time (s) & Steps &
		Time (s) & Steps & Time (s) \\
		\midrule
		0.1	& 10 & 9347 & 1591.86 & 4662 & 794.52 & 1668 & 307.02 & 1148 & 218.26 & 3.6 / 1.4 \\
		& 9  & 8773 & 1121.80 & 4445 & 576.34 & 1399 & 198.28 & 1038 & 153.13 & 3.7 / 1.3 \\ 
		0.01 	& 8 & 38003 & 3799.43 & 19318 & 1970.32 & 1367 & 168.86 & 179 & 21.91 & 90 / 7.7 \\ 		
		& 7 & 34630 & 2415.83 & 17663 & 1258.12 & 1282 & 119.51 & 161 & 14.81 & 85 / 8.0 \\ 
		0.001 & 6 & 2.4e5 & 11759.7 & 1.2e5 & 6174.50 & 5037 & 428.34 & 364 & 25.08 & 246 / 17.1 \\ 		
		& 5 & 2.1e5 & 6830.1 & 1.0e5 & 3573.08 & 4455 & 300.04 & 328 & 17.14 & 208 / 17.5 \\
		\bottomrule
	\end{tabular}
	\smallskip
	\begin{tabular}{@{}p{\linewidth}@{}}
          \textsuperscript{$\dagger$} FIM-3 uses $\gamma_2=1$,
          whereas FIM-3\textsuperscript{$\dagger$} uses $\gamma_2=15$, $30$,
          and $50$ for $\Kn=0.1$, $0.01$, and $0.001$, respectively. \\
          \textsuperscript{$\ddagger$} The speedups are the acceleration ratios of
          FIM-4 relative to FIM-3 with $\gamma_2=1$ and with increased $\gamma_2$, respectively.
	\end{tabular}
\end{table}

As in the first-order case, the performance of FIM-3 can be further improved by increasing $\gamma_2$. However, the appropriate choice of $\gamma_2$ varies considerably with the Knudsen number, and the values reported in \cref{tab:couette-2nd} are selected empirically for the corresponding cases. An excessively large $\gamma_2$ may even lead to loss of convergence, as shown in \cref{fig:couette-his-2nd-kn-pt001} for $\gamma_2=80$ at $\Kn=0.001$ and $M=6$.
Although convergence can sometimes be recovered by reducing the CFL number and retuning $\gamma_2$, such case-dependent adjustments limit the robustness and practical applicability of FIM-3. In contrast, FIM-4 employs one Newton-multigrid step in each alternating iteration ($\gamma_2=1$) and avoids the parameter tuning required by FIM-3. Even when compared with the carefully tuned FIM-3, FIM-4 achieves substantial speedups, ranging from approximately $1.4$ at $\Kn=0.1$ to more than $17$ at $\Kn=0.001$. The advantage of the Newton-multigrid strategy becomes more pronounced for the second-order discretization than in the first-order case, demonstrating its effectiveness for high-resolution computations.

\begin{figure}[!htb]
	\centering
	\subfloat[$\Kn=0.1$]{
		\label{fig:couette-his-2nd-kn-pt1}
		\includegraphics[width=0.3\textwidth]{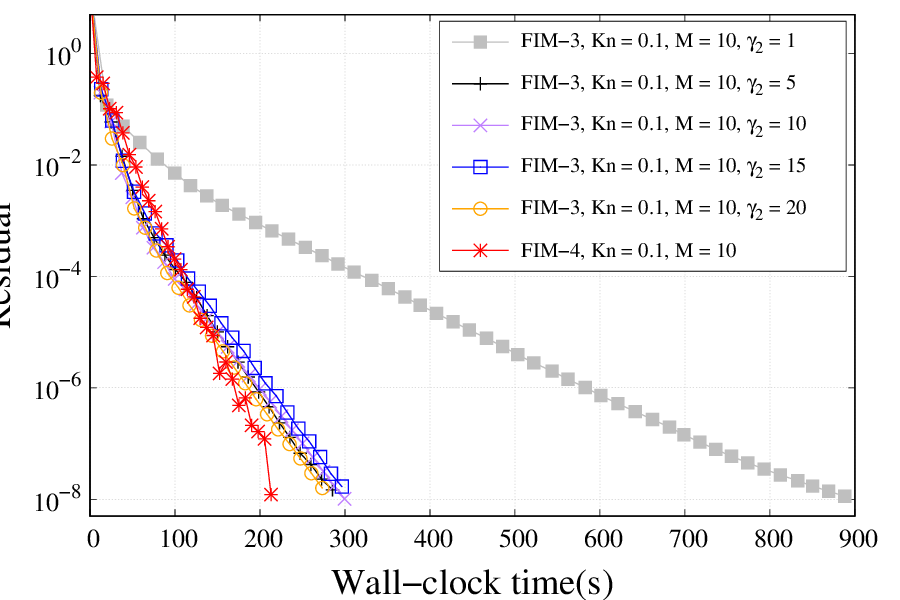}
	}\hfill
	\subfloat[$\Kn=0.01$]{
		\label{fig:couette-his-2nd-kn-pt01}
		\includegraphics[width=0.3\textwidth]{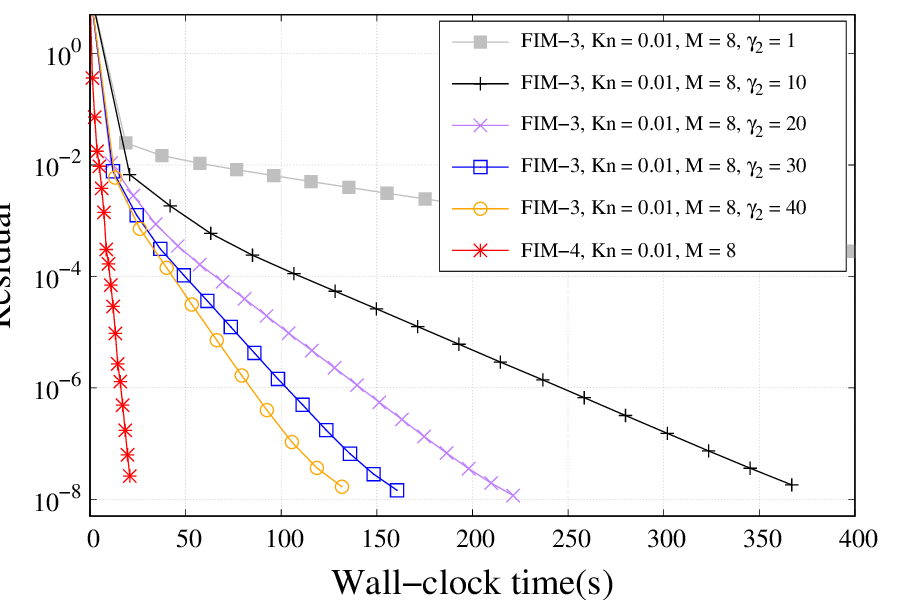}
	}\hfill
	\subfloat[$\Kn=0.001$]{
		\label{fig:couette-his-2nd-kn-pt001}
		\includegraphics[width=0.3\textwidth]{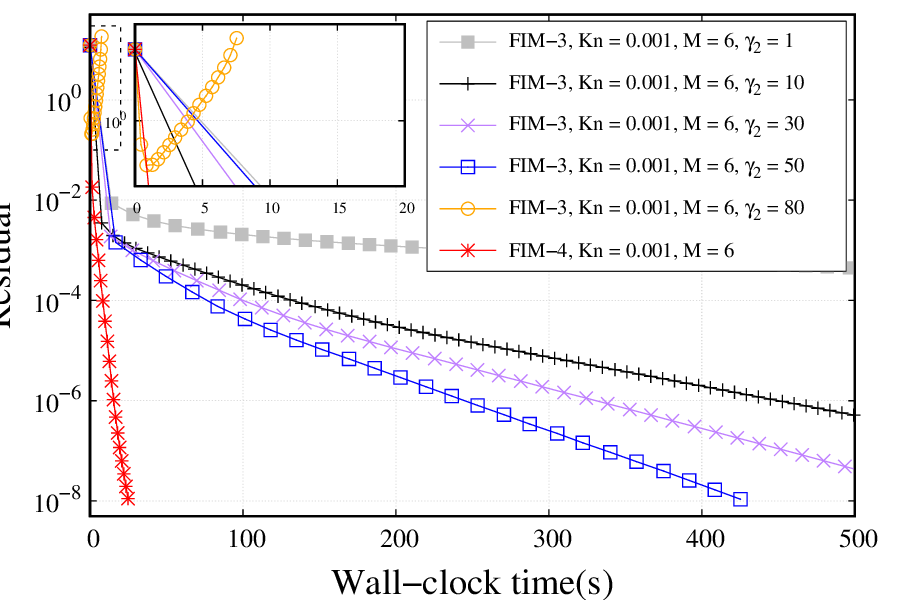}
	}
	\caption{Convergence histories for the Couette flow with $N_x=256$ and
		second-order spatial discretization.}
	\label{fig:couette-his-2nd}
\end{figure}

In addition, the total number of iterations of FIM-4 exhibits a slight rebound at $\Kn=0.001$ in \cref{tab:couette-2nd}. A supplemental test at $\Kn=0.001$ and $M=5$ shows that increasing $\gamma_2$ from $1$ to $2$ reduces the iteration number from $328$ to $175$, with the wall-clock time decreasing to about \qty{11}{\second}. Nevertheless, fixing $\gamma_2=1$ generally provides good convergence without requiring case-dependent parameter tuning.

The convergence histories of FIM-3 and FIM-4 in terms of wall-clock time are shown in \cref{fig:couette-his-2nd}. FIM-3 generally exhibits stable residual decay, except for the case with $\gamma_2=80$ at $\Kn=0.001$ and $M=6$, whereas FIM-4 maintains stable convergence in all cases shown. Similar to the first-order results, FIM-4 achieves the best overall performance, indicating that one Newton-multigrid step is sufficient to provide an effective hydrodynamic correction even with the second-order spatial discretization.

\subsubsection{Effect of the new regularization strategy}
	\label{sec:new-regularization-strategy}
	To investigate the effectiveness of the proposed regularization strategy defined in \cref{eq:regularize}, FIM-4 using the proposed $\Kn$-based regularization is compared with that using the original residual-based regularization for the Couette flow with $N_x=256$. 
	Since the two regularization strategies exhibit nearly identical performance at $\Kn=0.01$ and $0.001$, only the results at $\Kn=0.1$ are reported in \cref{tab:regular-couette}.

        As shown in \cref{tab:regular-couette}, the proposed $\Kn$-based regularization has little influence on the first-order discretization, indicating that the original residual-based regularization is already sufficient. For the second-order discretization at $\Kn=0.1$, it reduces the iteration number from $2949$ to $1148$ for $M=10$ and from $2601$ to $1038$ for $M=9$, corresponding to a reduction of more than $60\%$. The convergence histories in \cref{fig:rescompare} further show that the proposed $\Kn$-based regularization effectively suppresses the oscillations observed with the residual-based regularization strategy and achieves faster convergence. 
	These results suggest that the $\Kn$-dependent term provides additional stabilization for the Newton system under second-order discretization. Therefore, the proposed $\Kn$-based regularization is adopted in all  FIM-4 computations unless otherwise stated.

        \begin{table}[!htb]
		\centering
		\caption{Comparison of iteration numbers required by FIM-4 using different regularizations for the Couette flow at $\Kn=0.1$ with $N_x=256$.}
		\label{tab:regular-couette}
		\begin{tabular}{ccccc}
			\toprule
			\multirow{2}{*}{$M$} &
			\multicolumn{2}{c}{First-order} &
			\multicolumn{2}{c}{Second-order} \\
			\cmidrule(lr){2-3}\cmidrule(lr){4-5}
			& Residual-based & $\Kn$-based
			& Residual-based & $\Kn$-based \\
			\midrule
			10 & 653 & 632 & 2949 & 1148 \\
			9  & 647 & 629 & 2601 & 1038 \\
			\bottomrule
		\end{tabular}
	\end{table}
		
	\begin{figure}[!htb]
		\centering
		\includegraphics[width=0.45\linewidth]{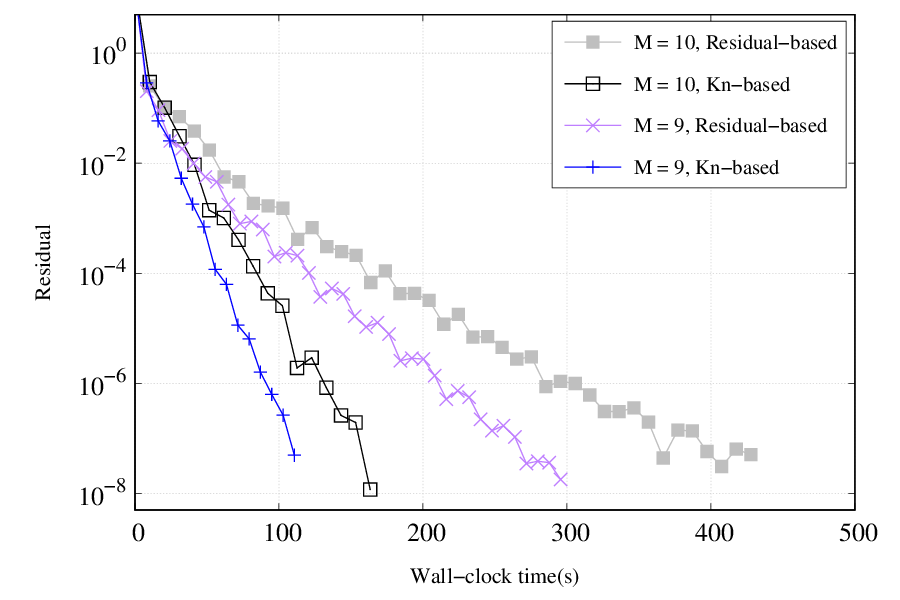}
		\caption{Convergence histories of FIM-4 using different regularizations for the Couette flow at $\Kn=0.1$ with $N_x=256$ and second-order discretization.}
		\label{fig:rescompare}
	\end{figure}

\subsection{Lid-driven cavity flow}
\label{sec:cavity}

The lid-driven cavity flow is one of the most frequently used benchmark tests in the multi-dimensional case. Our setting is the same as in \cite{Cai2014reduced, Cai2018linear, Hu2023}. The flow domain is a square cavity with side length $L = \qty{1}{\meter}$. The top boundary (the lid) of the cavity moves to the right at a constant speed $u^{W} = \qty{50}{\meter / \second}$, while the other walls are fixed. All the solid walls are maintained at a uniform reference temperature of $T_{\text{ref}} = \qty{273}{\kelvin}$.
Initially, the argon gas with a molecular mass of $m_{*}= \qty{6.63e-26}{\kg}$ is uniformly distributed and assumed to be in a Maxwellian equilibrium state with constant density $\rho^0$, zero mean velocity $\bu^0=\mathbf{0}$, and temperature $T^0=T_{\mathrm{ref}}$. 
Three initial densities, namely $\rho^{0} = \qty{8.58 e-7} {\kg / \metre^3}$, $\qty{8.58e-6}{\kg / \metre^3}$ and $\qty{8.58e-5}{\kg / \metre^3}$, corresponding to Knudsen numbers $\Kn = 0.1$, $0.01$ and $0.001$, respectively, are considered below.
Driven by the motion of the top lid, the gas would finally reach a steady state. 

\subsubsection{Solution validation and spatial DOF comparison}
\label{sec:cavity-sol}
Numerical solutions computed by the second-order FIM solver on the $64 \times 64$ grid are presented in \cref{fig:cavity_solutions}, including temperature contours and heat-flux streamlines at $\Kn=0.1$, as well as velocity streamlines and normalized velocity profiles along the centerlines at both $\Kn=0.01$ and $0.001$. 
Again, the DUGKS solutions obtained by the dugksFoam solver \cite{Zhu2017foam} are used as references, with a $60\times60$ physical grid employed for all Knudsen numbers and $81\times81$ discrete velocity points for $\Kn=0.1$ and $28\times28$ discrete velocity points for $\Kn=0.01$ and $0.001$.
It can be seen that our results agree well with DUGKS results, validating the applicability of the moment system with moderate orders in the near-continuum flow regime.

\begin{figure}[!htb]
	\centering
	\subfloat[Temperature $(\si{\kelvin})$, $M=10$]{
		\label{fig:cavity_pt1_T}
		\includegraphics[width=0.3\textwidth]{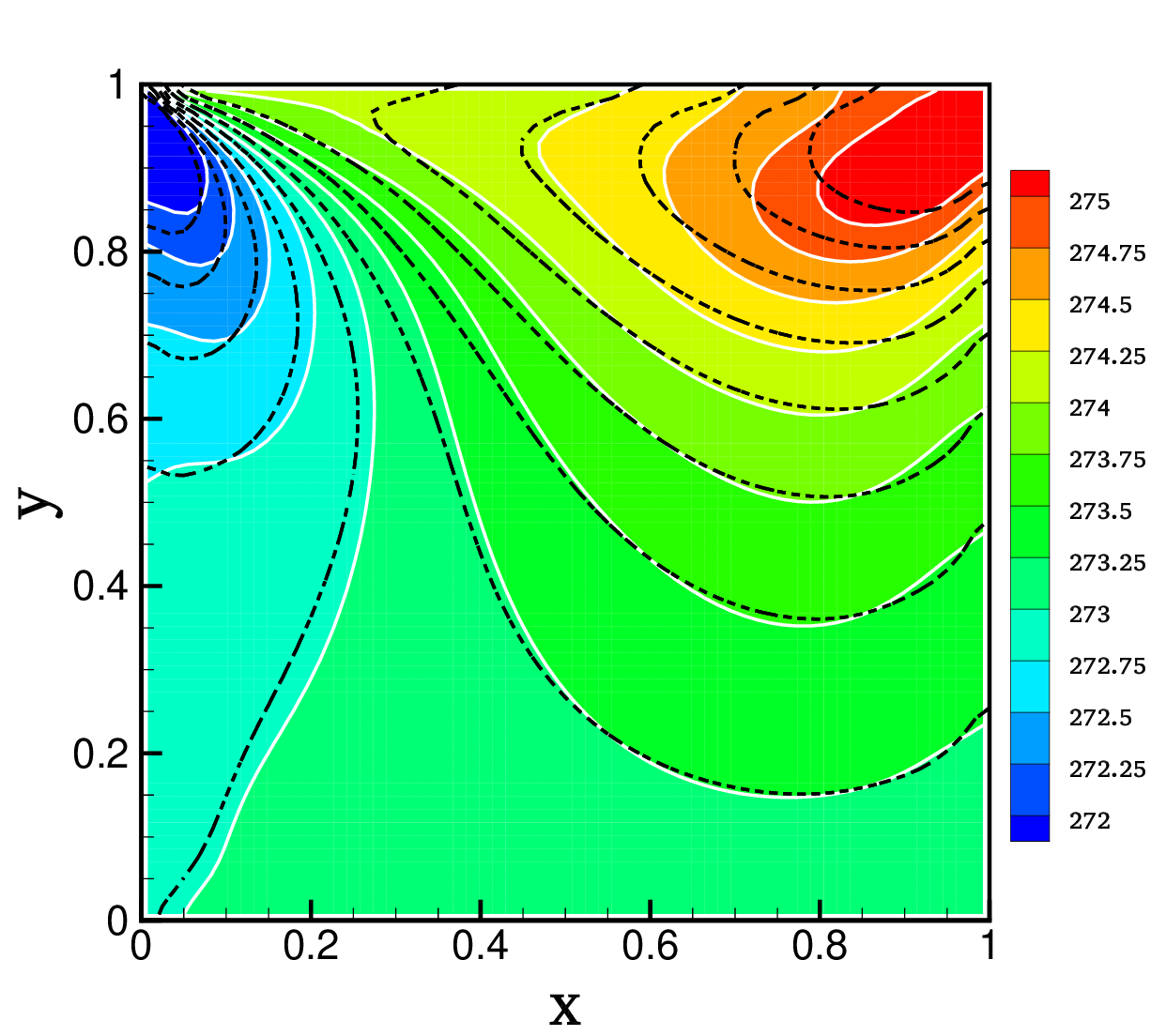}
	}\hfill
	\subfloat[Velocity streamlines, $M = 8$]{
		\label{fig:velocity_streamlines_pt01}
		\includegraphics[width=0.3\textwidth]{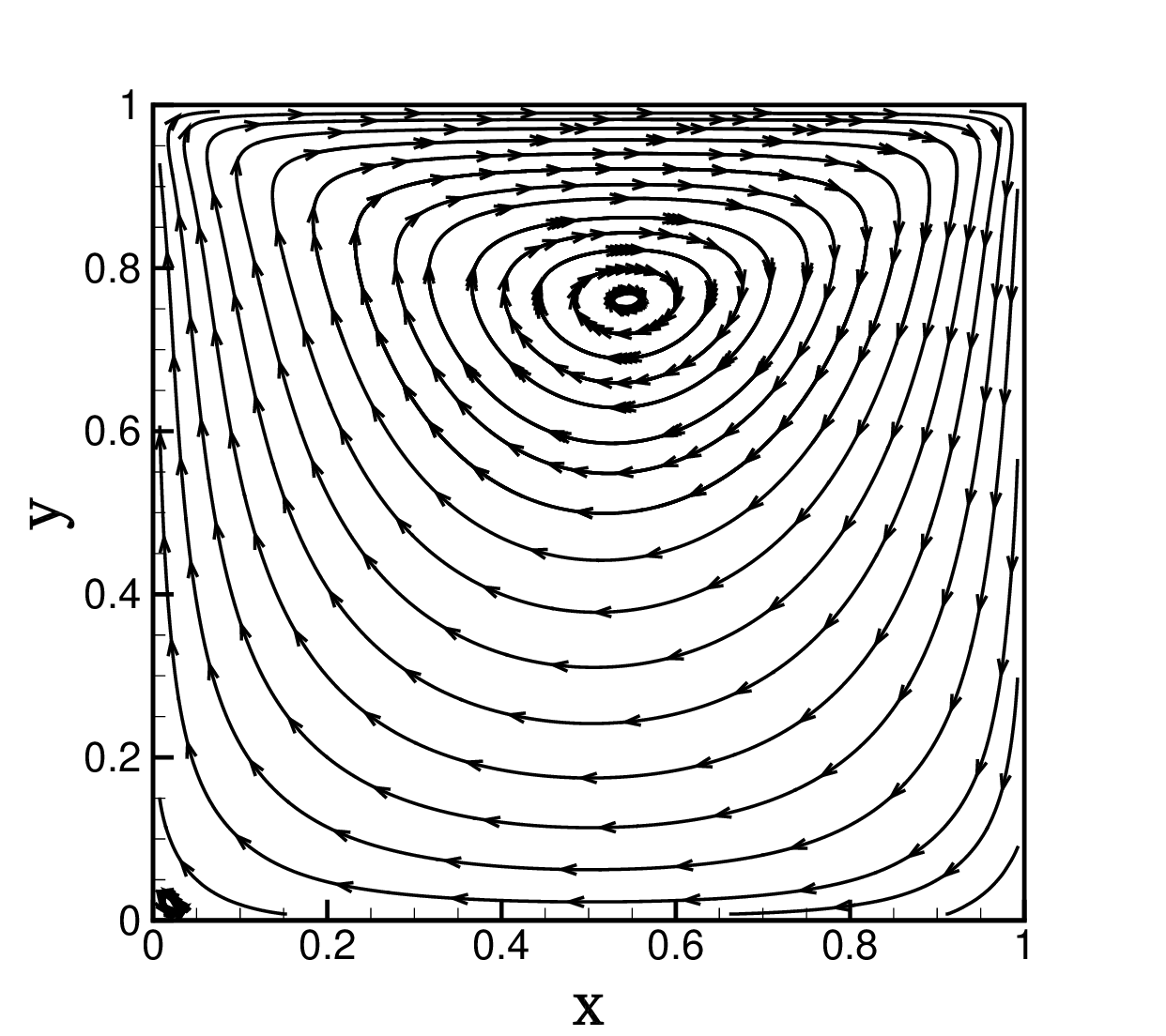}
	}\hfill
	\subfloat[Velocity streamlines, $M = 6$]{
		\label{fig:velocity_streamlines_pt001}
		\includegraphics[width=0.3\textwidth]{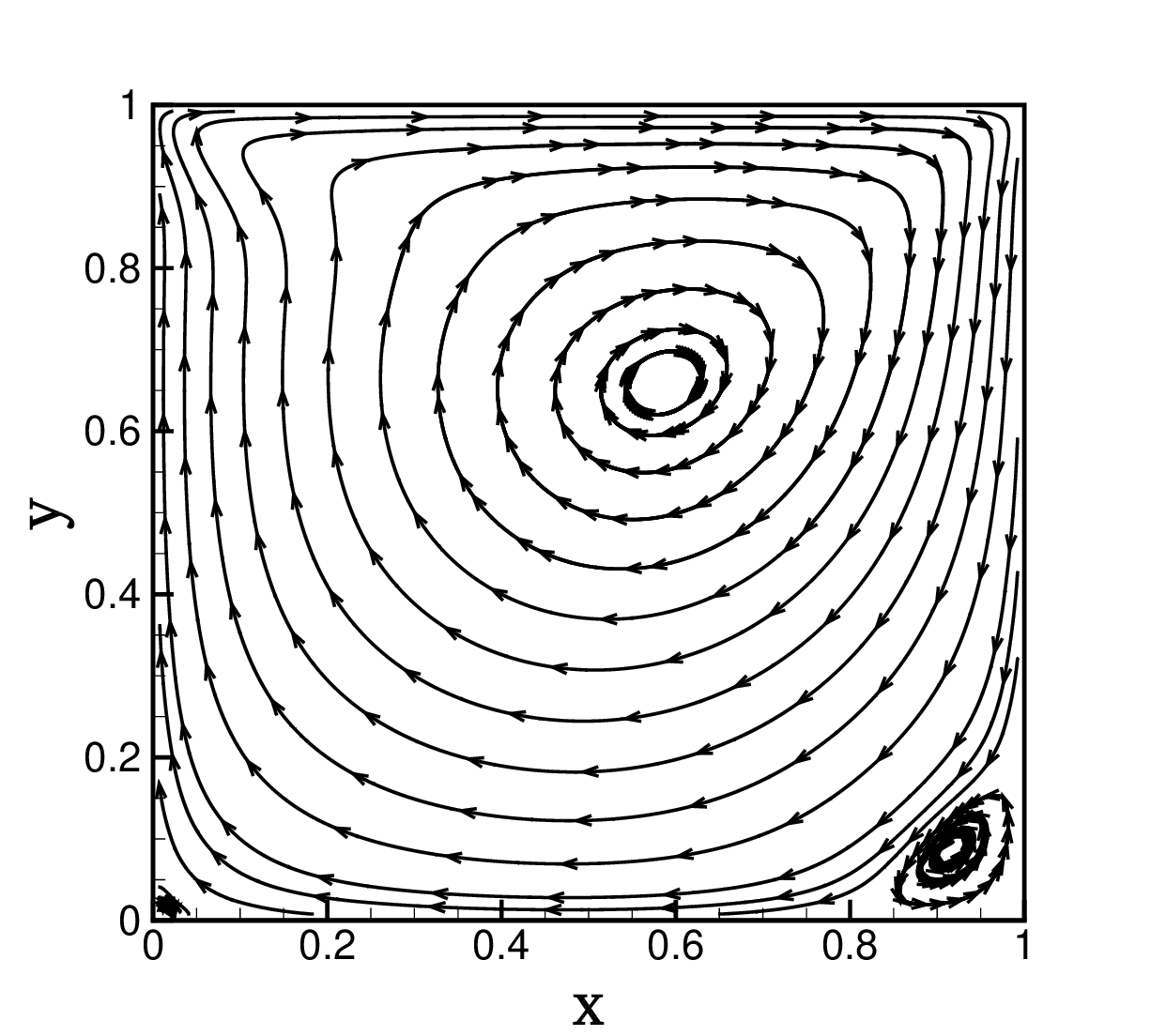}
	}\\
	\subfloat[Heat flux streamlines, $M=10$]{
		\label{fig:cavity_pt1_q}
		\includegraphics[width=0.3\textwidth]{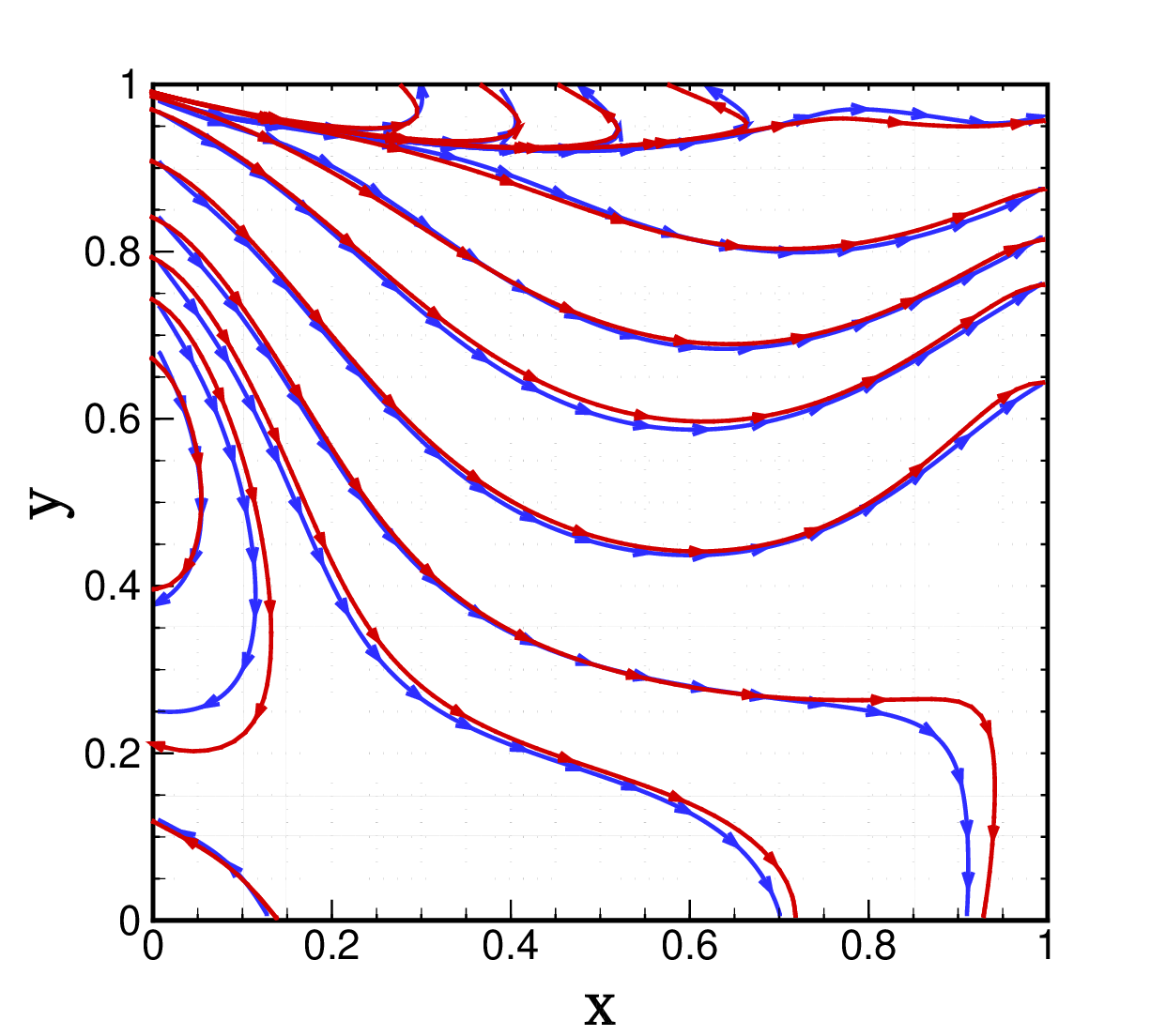}
	}\hfill
	\subfloat[$u_1$, $u_2$ along centerlines]{
		\label{fig:slice_u_v_pt01}
		\includegraphics[width=0.3\textwidth]{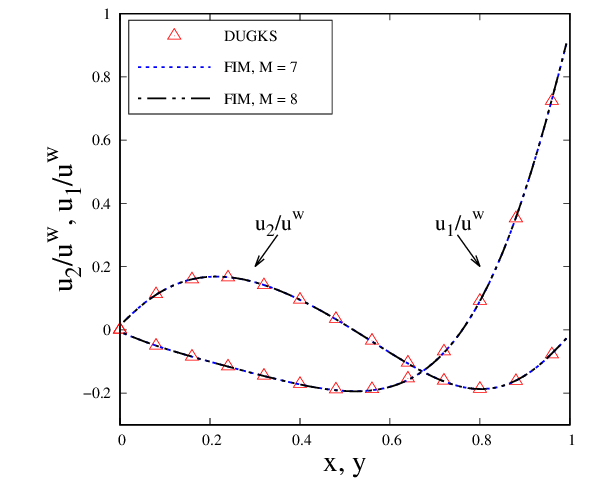}
	}\hfill
	\subfloat[$u_1$, $u_2$ along centerlines]{
		\label{fig:slice_u_v_pt001}
		\includegraphics[width=0.3\textwidth]{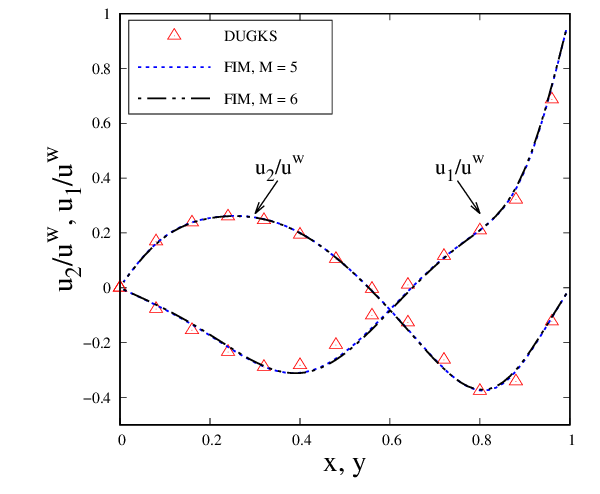}
	}
	\caption{Numerical solutions of the lid-driven cavity flow with $N_x \times
		N_y = 64 \times 64$. \textit{Left column}: results at $\Kn=0.1$, with dashed and
		red lines representing DUGKS solutions; \textit{Middle column}: velocity
		streamlines and normalized velocity profiles ($u_1$ along the vertical
		centerline, $u_2$ along the horizontal centerline) at $\Kn=0.01$;
		\textit{Right column}: corresponding velocity results at $\Kn=0.001$.}
	\label{fig:cavity_solutions}
\end{figure}

The same flow configuration was also investigated in \cite{Li2026}, where it was shown that the first-order FIM solver generally requires much finer grids to obtain solutions comparable with the reference results, especially in the near-continuum flow regime. 
At $\Kn=0.001$, 
even with a grid of $512\times512$ cells, the first-order discretization still shows noticeable discrepancies in the size, location, and shape of the primary and secondary vortices compared with the second-order results, as shown in \cref{fig:velocity_streamlines_pt001,fig:cavity_pt001_vel_1st}. To further examine the spatial convergence behavior, the velocity profiles along four representative slices are compared in \cref{fig:cavity_pt001_vel_sli_pt25,fig:cavity_pt001_vel_sli_pt75}. The second-order FIM solver on the $64\times64$ grid already agrees well with the DUGKS reference solution, whereas the first-order solver still exhibits visible deviations even on the $2048\times2048$ grid. These results demonstrate the substantially improved spatial resolution of the second-order discretization, which significantly reduces the required spatial DOF without sacrificing solution accuracy and thereby further improves the efficiency of the FIM method in the near-continuum flow regime.

\begin{figure}[!htb]
	\centering
	\subfloat[]{
		\label{fig:cavity_pt001_vel_1st}
	   \includegraphics[width=0.3\textwidth]{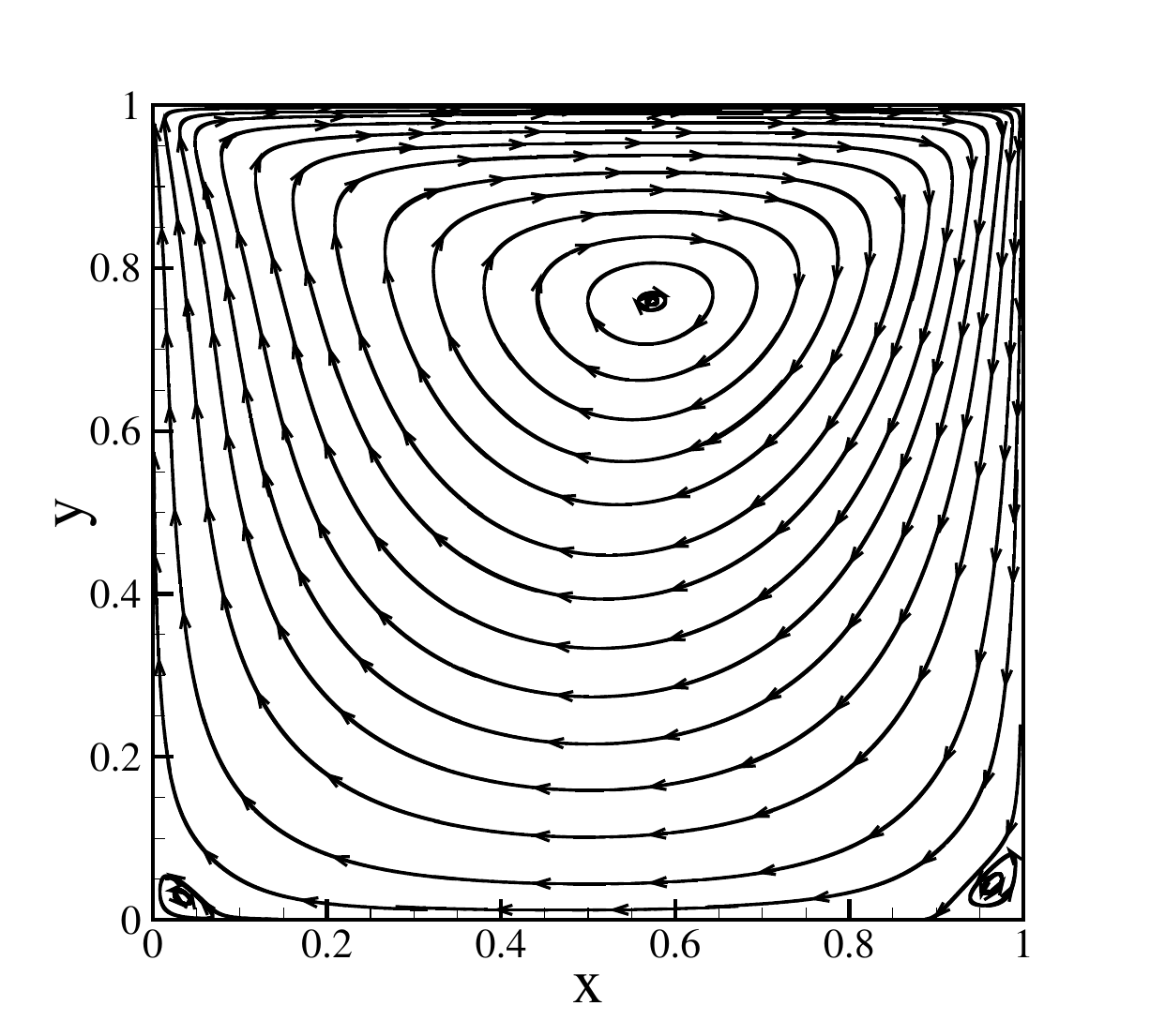}
	}\hfill
	\subfloat[]{
	    \label{fig:cavity_pt001_vel_sli_pt25}
	    \includegraphics[width=0.3\textwidth]{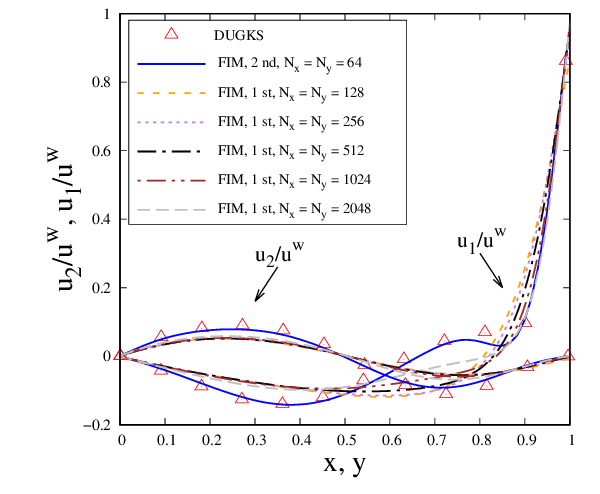}
	}\hfill
	\subfloat[]{
	    \label{fig:cavity_pt001_vel_sli_pt75}
	    \includegraphics[width=0.3\textwidth]{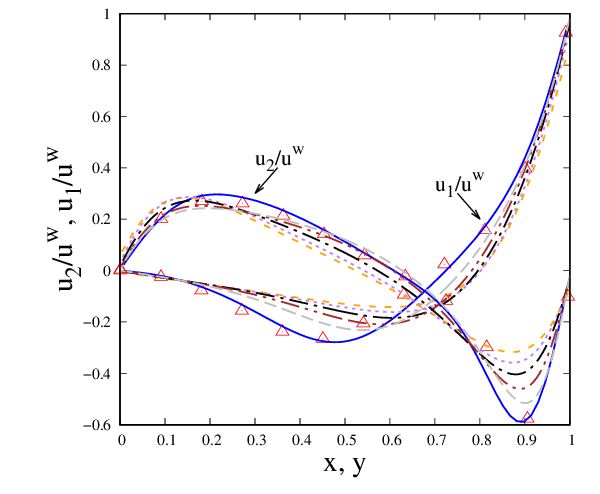}
	}
	\caption{Numerical solutions of the lid-driven cavity flow with $M=6$ at
	$\Kn=0.001$. \textit{Left}: velocity streamlines with $N_x = N_y = 512$ and
	first-order discretization; \textit{Middle}: $u_1$ ($u_2$) along $x=1/4$
	($y=1/4$); \textit{Right}: $u_1$ ($u_2$) along $x=3/4$ ($y=3/4$).}
	\label{fig:cavity-sol-1st}
\end{figure}

\subsubsection{Performance of FIM solvers}
\label{sec:cavity-eff}

The first-order IMEX-SGS iteration, FIM-3, and FIM-4 are tested on the $64\times64$ grid for all three Knudsen numbers, with the results summarized in \cref{tab:cavity-1st}. 
Similar to the Couette-flow results, FIM-3 can be accelerated by increasing $\gamma_2$, and the values of $\gamma_2$ listed in \cref{tab:cavity-1st} are selected empirically for the corresponding Knudsen numbers. 
However, the convergence histories in \cref{fig:cavity-his-1st} show that further increasing $\gamma_2$ provides only limited additional acceleration.
In contrast, FIM-4 achieves fast convergence with one Newton-multigrid step in each alternating iteration, i.e., $\gamma_2=1$, and thus requires no case-dependent parameter tuning. 
Compared with the tuned FIM-3, FIM-4 is slightly slower at $\Kn=0.1$, but becomes substantially more efficient as the Knudsen number decreases. At $\Kn=0.01$, the wall-clock time is reduced by factors of approximately $4$, while a speedup of about $2$ is maintained at $\Kn=0.001$. These results demonstrate that one Newton-multigrid step provides an effective hydrodynamic correction for the multi-dimensional FIM framework, achieving high efficiency and robust performance without empirical tuning of $\gamma_2$. 

\begin{table}[!htb]
	\centering
	\caption{Performance results of the IMEX-SGS, FIM-3 and FIM-4 for the lid-driven
		cavity flow at various $\Kn$ with $N_x = N_y = 64$ and first-order
		discretization.}
	\label{tab:cavity-1st}
	\begin{tabular}{ccccccccc}
		\toprule
		\multirow{2}{*}{$\Kn$} & \multirow{2}{*}{$M$} &
		\multicolumn{2}{c}{IMEX-SGS} &
		\multicolumn{2}{c}{FIM-3\textsuperscript{$\dagger$}} &
		\multicolumn{2}{c}{FIM-4} &
		\multirow{2}{*}{Speedup \textsuperscript{$\ddagger$}} \\
		\cmidrule(lr){3-4} \cmidrule(lr){5-6} \cmidrule(lr){7-8}
		& & Steps & Time (s) &
		Steps & Time (s) &
		Steps & Time (s) & \\
		\midrule
		0.1   & 10 & 2076 & 4173.95 & 274 & 795.75 & 276 & 909.58 & 0.87 \\
		& 9  & 2054 & 3267.26 & 263 & 574.18 & 262 & 670.63 & 0.85 \\
		0.01  & 8  & 2604 & 3114.57 & 74  & 321.13 & 42  & 81.08  & 3.96 \\
		& 7  & 2482 & 2287.67 & 70  & 251.10 & 40  & 56.18  & 4.47\\
		0.001 & 6  & 2731 & 1676.43 & 39  & 100.15 & 47  & 50.62  & 1.98 \\
		& 5  & 2610 & 1224.42 & 38  & 80.12  & 46  & 40.73  & 1.97 \\
		\bottomrule
	\end{tabular}
	\smallskip
	\begin{tabular}{@{}p{\linewidth}@{}}
	\textsuperscript{$\dagger$} The parameter $\gamma_2$ in FIM-3 is set to $10$, $40$, and $80$ for $\Kn = 0.1$, $0.01$, and $0.001$, respectively. \\
	\textsuperscript{$\ddagger$} The speedup is the acceleration ratio of FIM-4 relative to FIM-3.
	\end{tabular}
\end{table}

\begin{figure}[!htb]
	 \centering
	 \subfloat[$\Kn=0.1$]{
	    \includegraphics[width=0.3\textwidth]{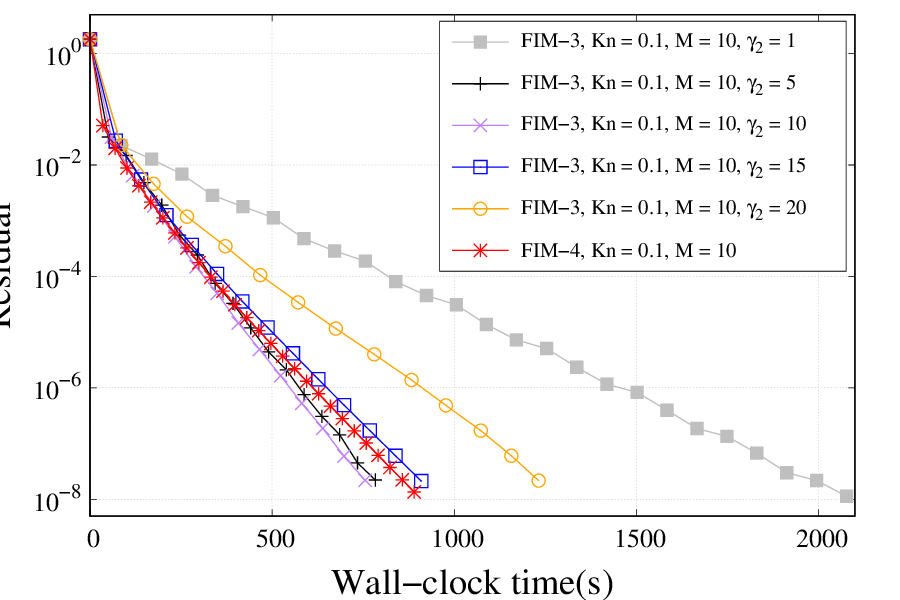}
	 }\hfill
	 \subfloat[$\Kn=0.01$]{
	    \includegraphics[width=0.3\textwidth]{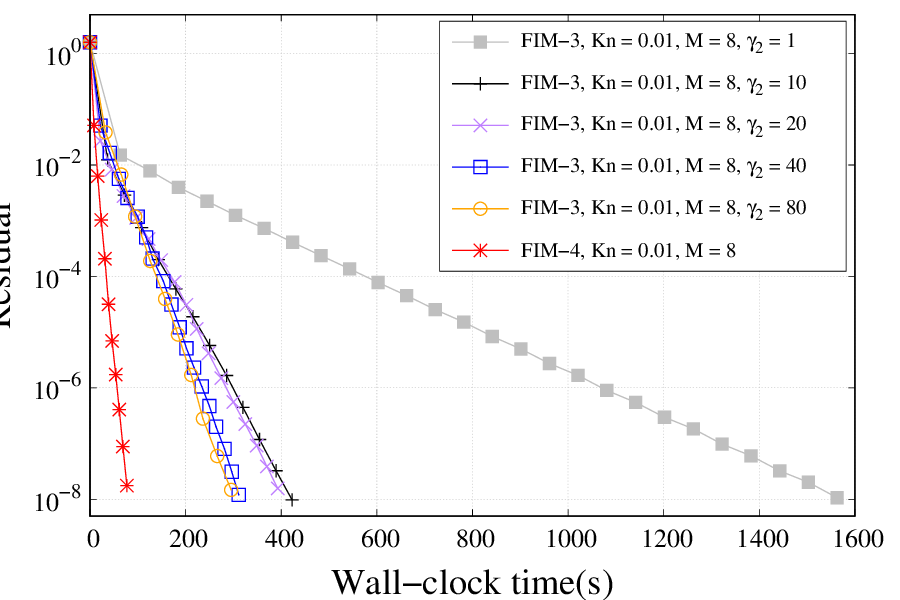}
	 }\hfill
	 \subfloat[$\Kn=0.001$]{
	    \includegraphics[width=0.3\textwidth]{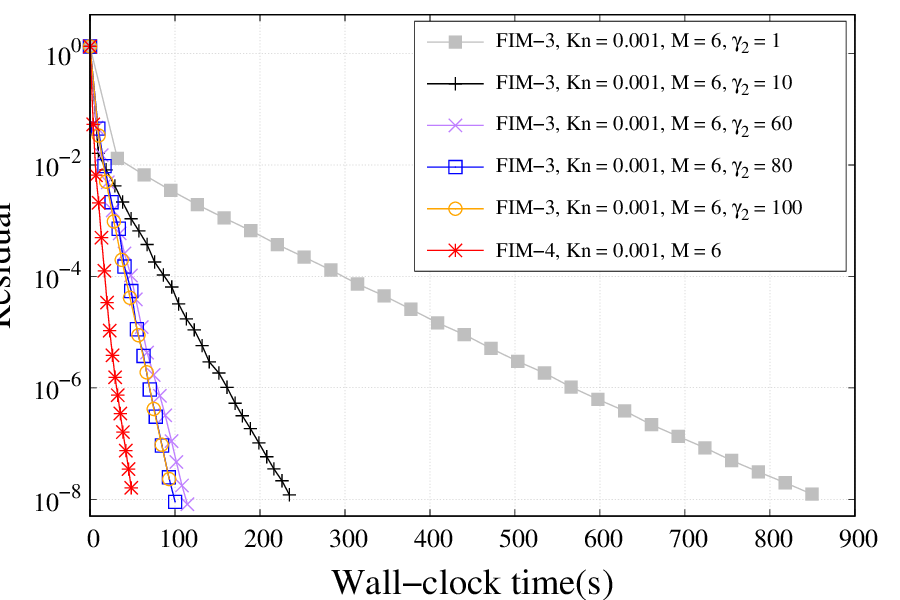}
	 }
	 \caption{Convergence histories for the lid-driven cavity flow with
	 $N_x=N_y=64$ and first-order spatial discretization.}
	 \label{fig:cavity-his-1st}
\end{figure}

The results of the second-order solvers for the lid-driven cavity flow are summarized in \cref{tab:cavity-2nd} and the convergence histories are shown in \cref{fig:cavity-his-2nd}. Possibly due to the wider stencil of the second-order discretization, all solvers generally converge more slowly than in the first-order case, while providing higher spatial accuracy. As in the first-order case, FIM-3 requires empirical tuning of $\gamma_2$ to achieve satisfactory performance, whereas FIM-4 employs one Newton-multigrid step in each alternating iteration. 
Compared with the tuned FIM-3, FIM-4 is substantially more efficient at small Knudsen numbers, achieving speedups exceeding $4$ at $\Kn=0.01$ and about $3$ at $\Kn=0.001$.
At $\Kn=0.1$, FIM-4 is moderately less efficient than the tuned FIM-3, which may be partly related to the higher computational cost associated with the full Jacobian system. This cost could potentially be reduced by simplifying the Jacobian matrix. 
In the problems studied in \cite{Hu2026SWEs}, such a strategy reduced the Jacobian-related computational cost by more than $30\%$ without noticeably affecting the convergence behavior.

Additionally, a rebound in the number of iterations is again observed for FIM-4 with fixed $\gamma_{2}$ at $\Kn=0.001$, which can be effectively suppressed by increasing the value of $\gamma_2$. A supplemental test for $M=5$ shows that increasing $\gamma_2$ from $1$ to $5$ reduces the number of iterations from $416$ to $95$, while the wall-clock time decreases only slightly, from $\qty{455.75}{\second}$ to about $\qty{439}{\second}$.

\begin{table}[!htb]
	\centering
	\caption{Performance results of the IMEX-SGS, FIM-3 and FIM-4 for the lid-driven
		cavity flow at various $\Kn$ with $N_x = N_y = 64$ and second-order
		discretization.}
	\label{tab:cavity-2nd}
	\begin{tabular}{ccccccccc}
		\toprule
		\multirow{2}{*}{$\Kn$} & \multirow{2}{*}{$M$} &
		\multicolumn{2}{c}{IMEX-SGS} &
		\multicolumn{2}{c}{FIM-3\textsuperscript{$\dagger$}} &
		\multicolumn{2}{c}{FIM-4} &
		\multirow{2}{*}{Speedup \textsuperscript{$\ddagger$}} \\
		\cmidrule(lr){3-4} \cmidrule(lr){5-6} \cmidrule(lr){7-8}
		& & Steps & Time (s) &
		Steps & Time (s) &
		Steps & Time (s) & \\
		\midrule
		0.1   & 10 & 3336  & 4435.14 & 524 & 1293.78 & 504 & 1775.9  & 0.73 \\
		& 9  & 3129  & 3217.69 & 499 & 844.08  & 471 & 1250.1  & 0.67 \\
		0.01  & 8  & 8100  & 6186.40 & 258 & 731.21  & 82  & 173.15  & 4.22 \\
		& 7  & 7418  & 4062.74 & 237 & 564.20  & 77  & 122.43  & 4.61 \\
		0.001 & 6  & 41117 & 16244.3 & 734 & 1710.64 & 458 & 580.82  & 2.95 \\
		& 5  & 36324 & 9956.71 & 502 & 1478.64 & 416 & 455.75  & 3.24 \\
		\bottomrule
	\end{tabular}	
	\smallskip
	\begin{tabular}{@{}p{\linewidth}@{}}
		\textsuperscript{$\dagger$} The parameter $\gamma_2$ in FIM-3 is set to $10$,
		$40$, and $80$ for $\Kn = 0.1$, $0.01$, and $0.001$, respectively. \\
		\textsuperscript{$\ddagger$} The speedup is the acceleration ratio of FIM-4 relative to FIM-3.
	\end{tabular}
\end{table}

\begin{figure}[!htb]
	 \centering
	 \subfloat[$\Kn=0.1$]{
		 \includegraphics[width=0.3\textwidth]{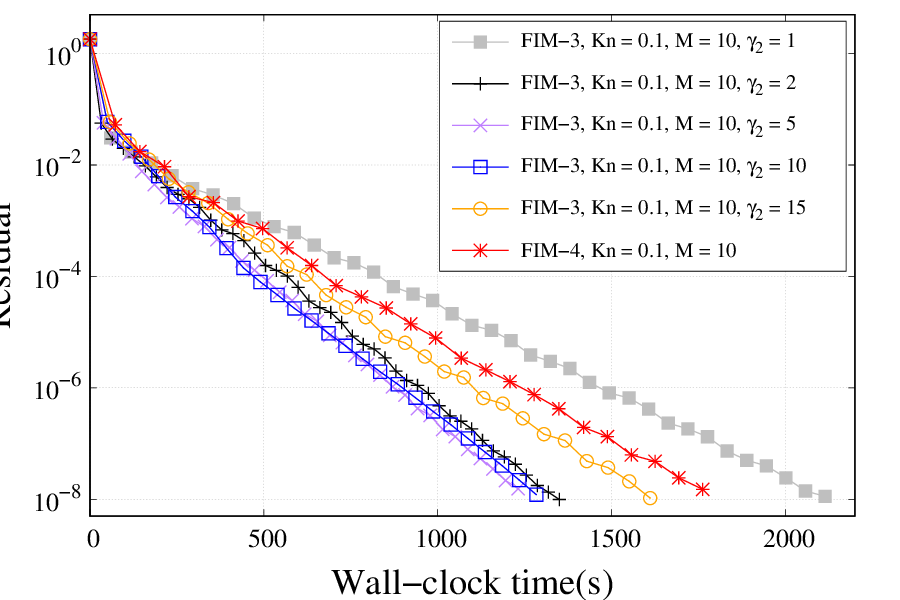}
	 }\hfill
	 \subfloat[$\Kn=0.01$]{
	    \includegraphics[width=0.3\textwidth]{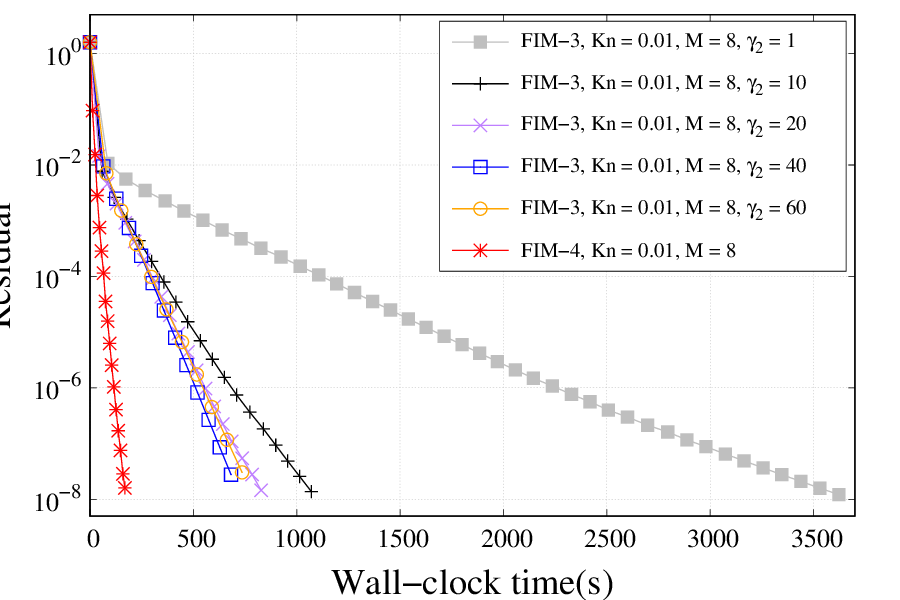}
	 }\hfill
	 \subfloat[$\Kn=0.001$]{
	    \includegraphics[width=0.3\textwidth]{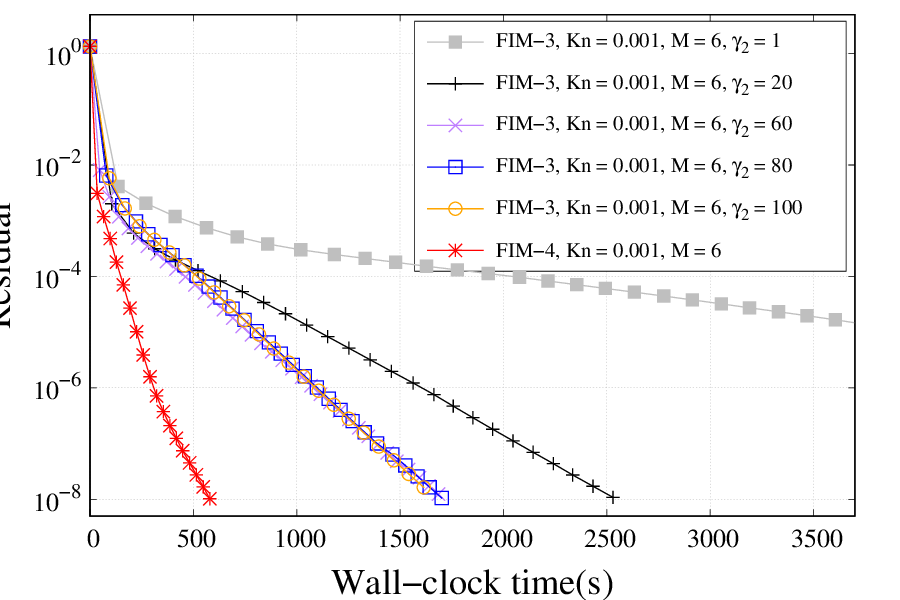}
	 }
	 \caption{Convergence histories for the lid-driven cavity flow with
	 $N_x=N_y=64$ and second-order spatial discretization.}
	 \label{fig:cavity-his-2nd}
\end{figure}

\subsection{Heat transfer in a bottom-heated cavity}
\label{sec:heat_trans}
The third example is a heat transfer problem for the gas confined in a bottom-heated cavity \cite{Cai2018linear, Rana2015, Hu2023}. The computational domain and gas properties are the same as those in the lid-driven cavity problem (\cref{sec:cavity}). The initial conditions are also the same, except that the temperature is $T^{0}=\qty{300}{\kelvin}$. All walls are stationary, with the bottom wall maintained at \qty{600}{\kelvin} and the other three walls at \qty{300}{\kelvin}.

The temperature and heat-flux solutions obtained by the second-order FIM solver on a $64\times64$ grid are shown in \cref{fig:heat_transfer_solutions} and agree well with the DUGKS results generated by the dugksFoam solver. As $\Kn$ increases, the temperature jumps near the bottom-left and bottom-right corners are more obvious due to the discontinuous boundary temperature \cite{Cai2025}.

\begin{figure}[!htb]
	\centering
	\subfloat[$M = 10$, $\Kn = 0.1$]{
	   \includegraphics[width=0.3\textwidth]{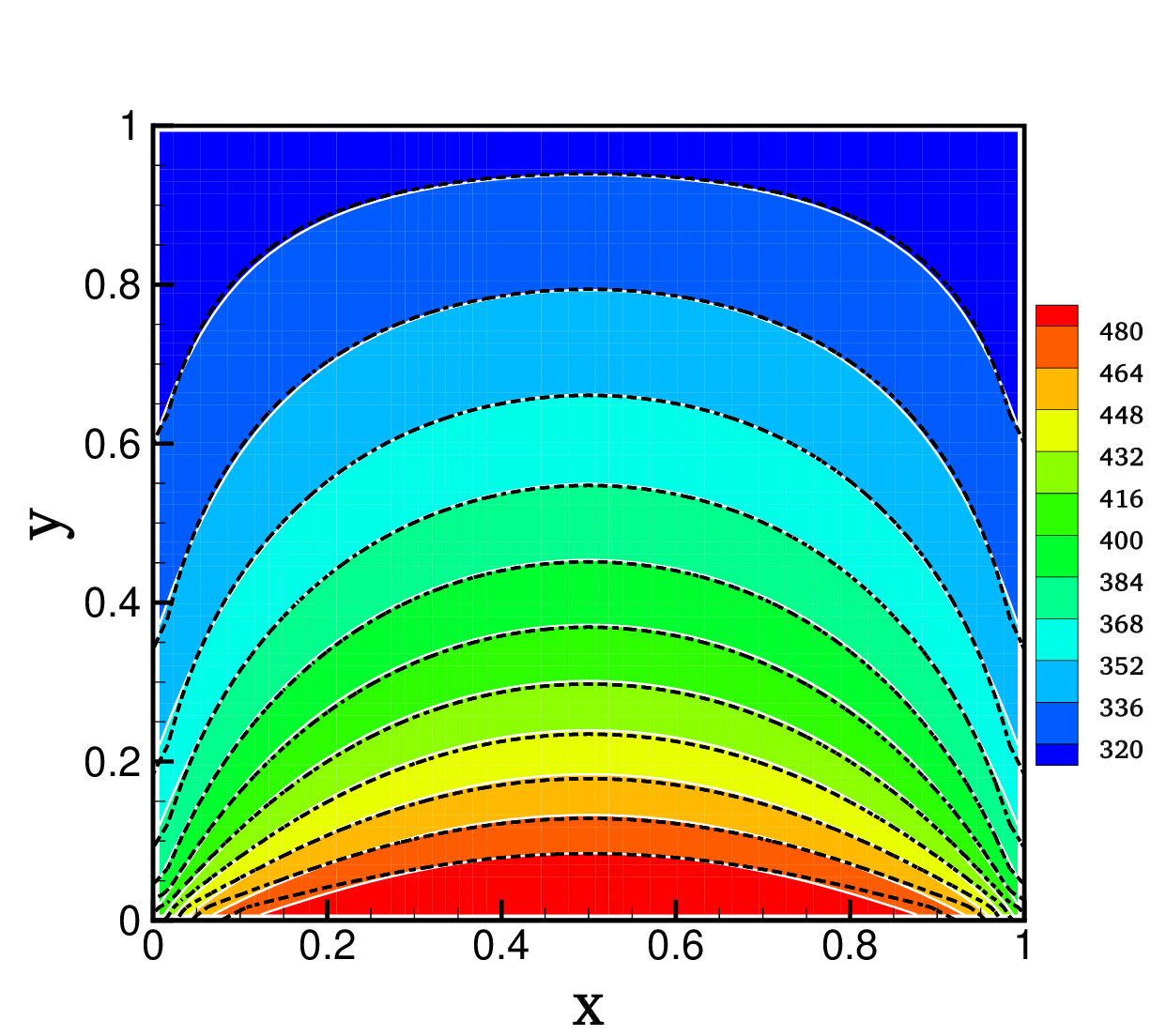}
	}\hfill
	\subfloat[$M = 8$, $\Kn = 0.01$]{
	    \includegraphics[width=0.3\textwidth]{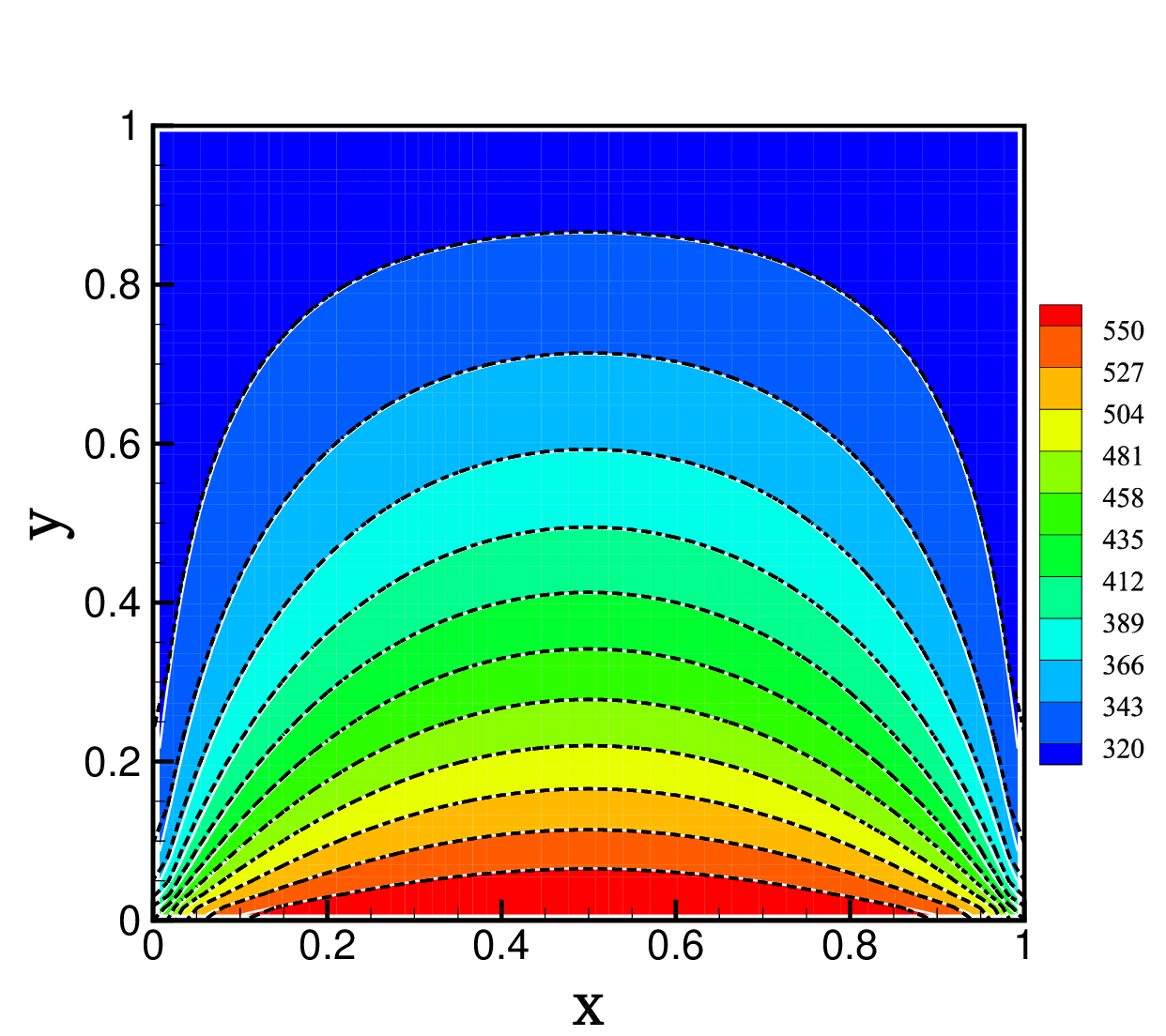}
	}\hfill
	\subfloat[$M = 6$, $\Kn = 0.001$]{
	    \includegraphics[width=0.3\textwidth]{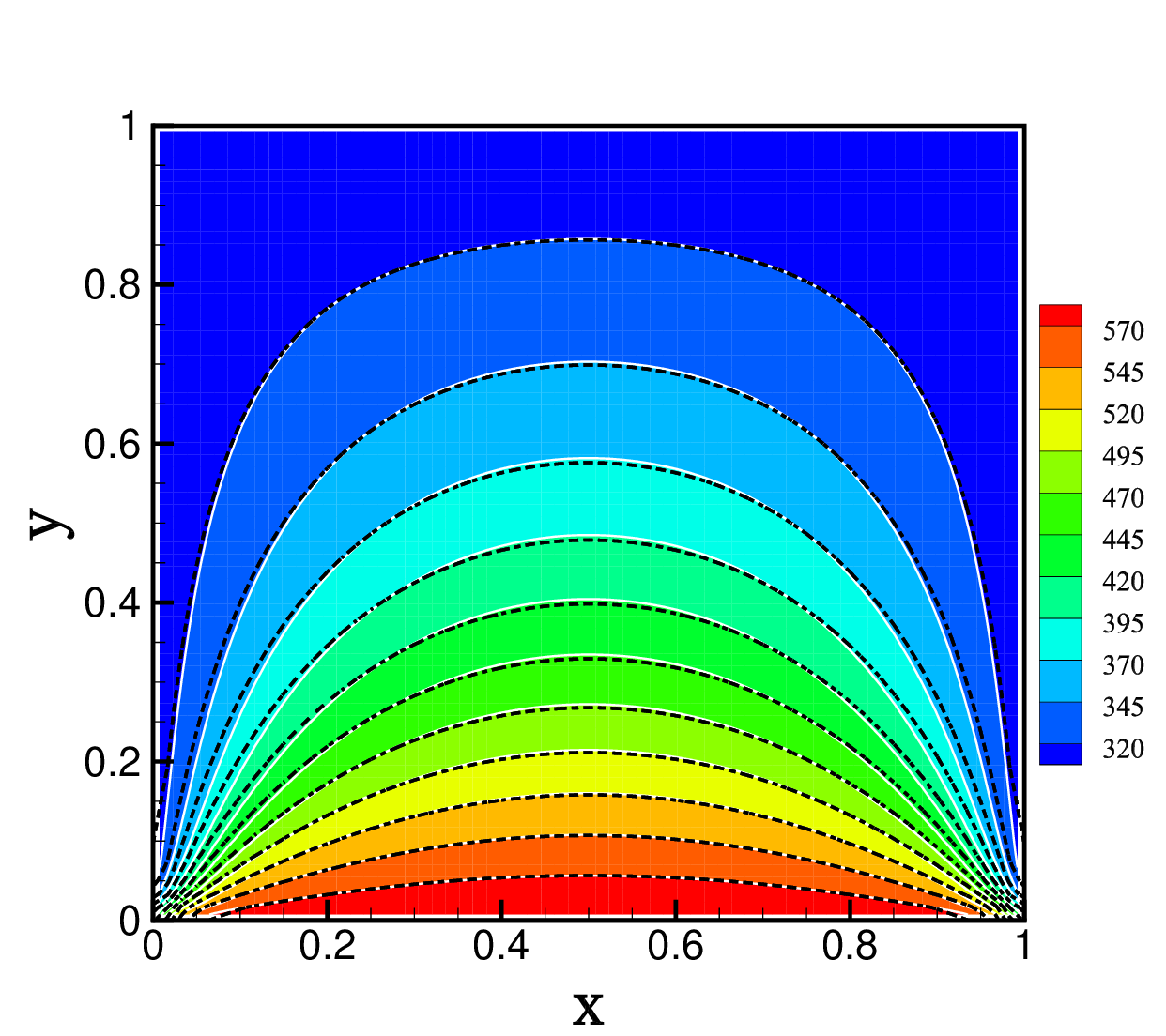}
	}\\
	\subfloat[$M = 10$, $\Kn = 0.1$]{
	    \includegraphics[width=0.3\textwidth]{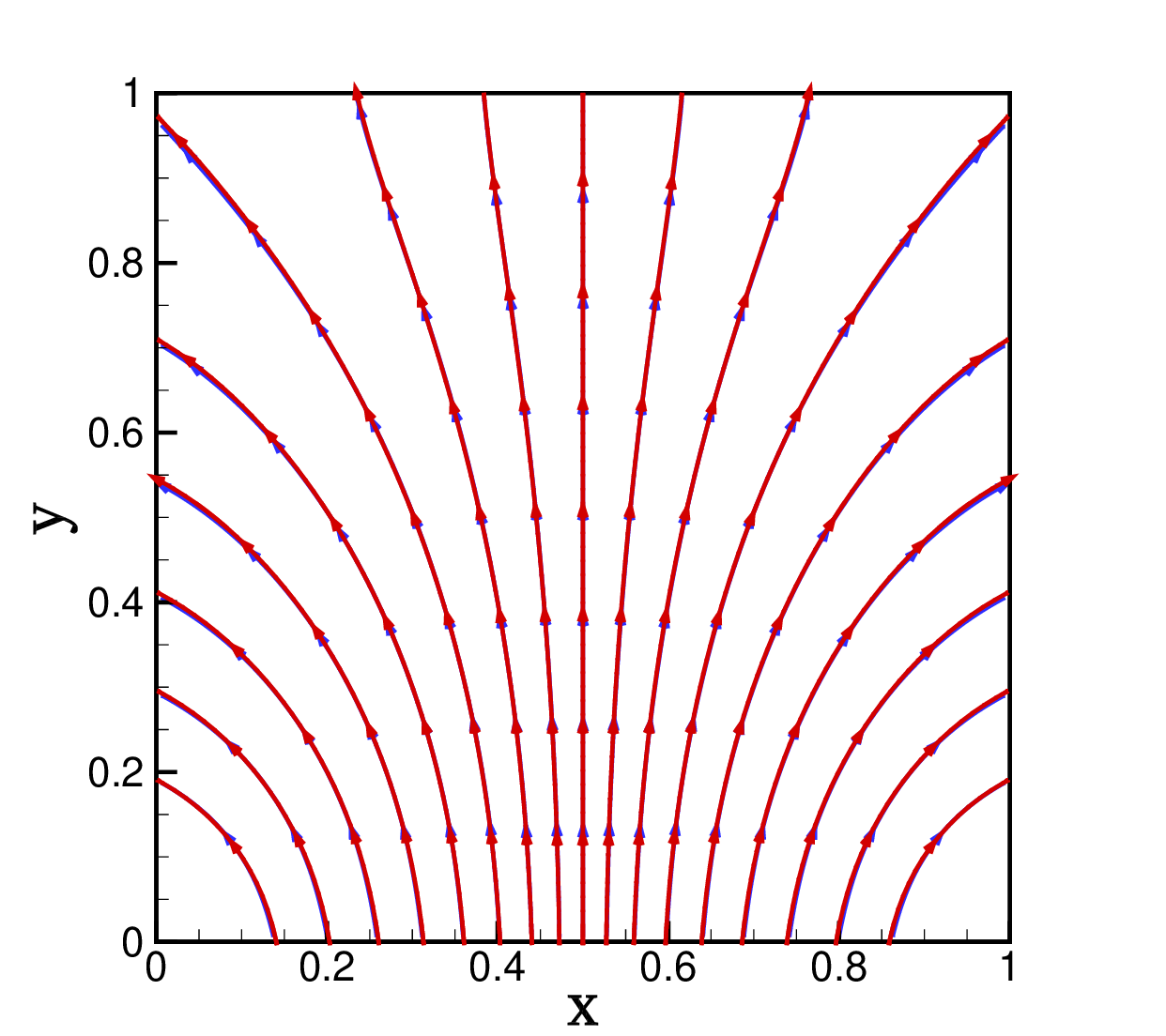}
	}\hfill
	\subfloat[$M = 8$, $\Kn = 0.01$]{
	    \includegraphics[width=0.3\textwidth]{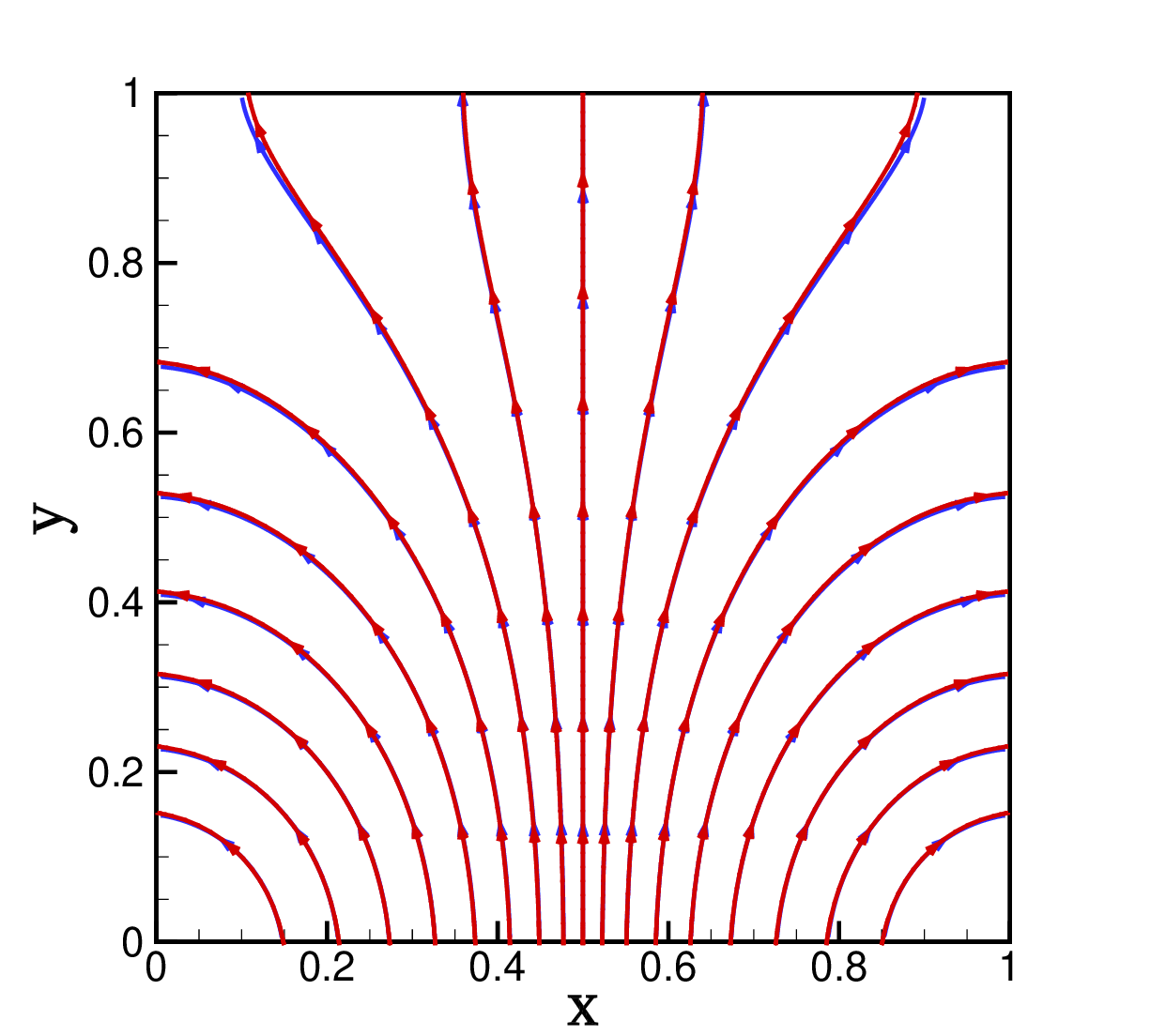}
	}\hfill
	\subfloat[$M = 6$, $\Kn = 0.001$]{
	    \includegraphics[width=0.3\textwidth]{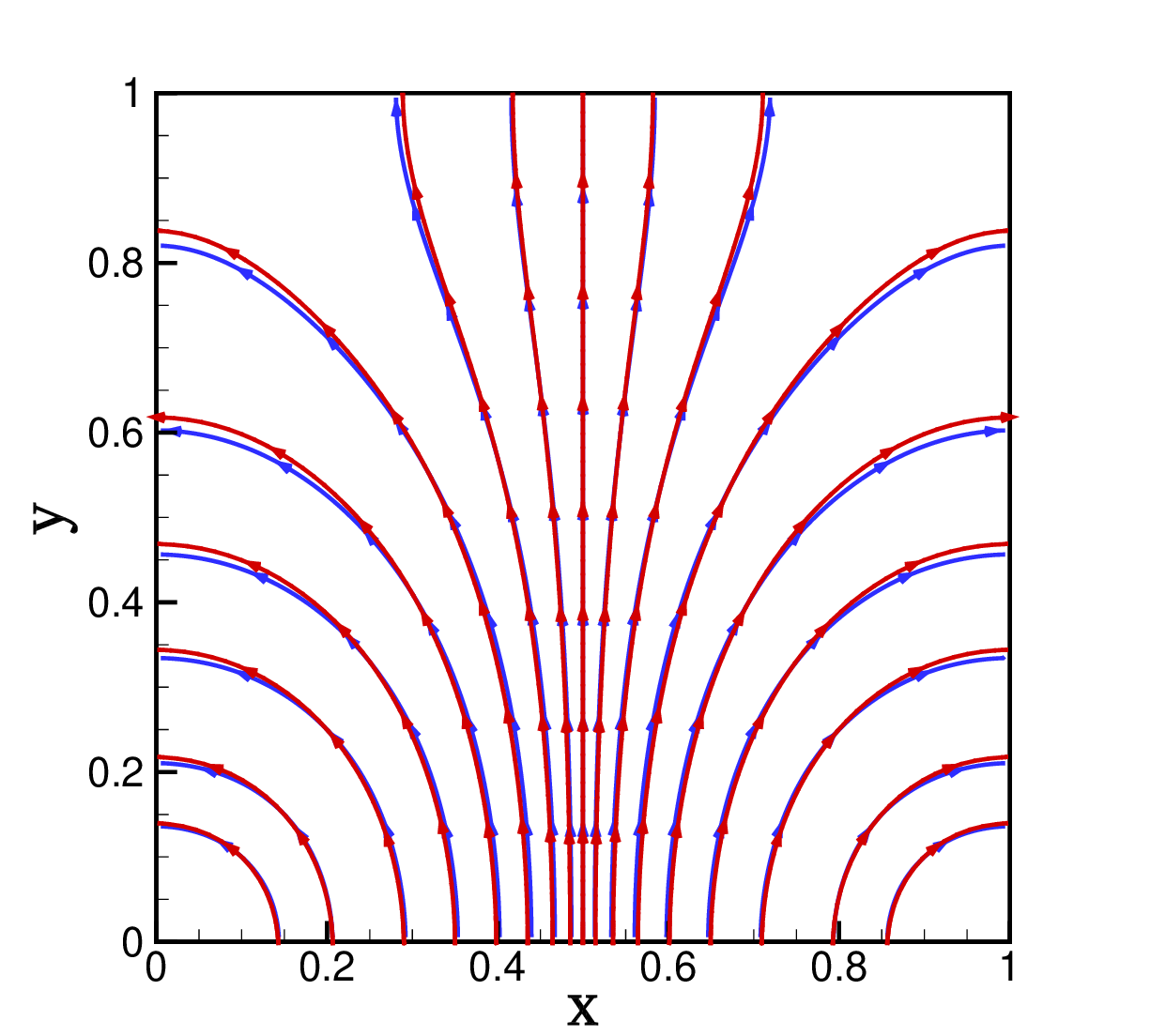}
	}
	\caption{Temperature $(\si{\kelvin})$ contours and heat flux streamlines of
	heat transfer on the uniform $N_x \times N_y = 64 \times 64$ grid. The dashed
	and red lines are the reference DUGKS solutions.}
	\label{fig:heat_transfer_solutions}
\end{figure}

The performance of the FIM solvers is again examined for this problem. The total numbers of iterations and wall-clock times on the $64\times64$ grid are listed in \cref{tab:heat-1st} and \cref{tab:heat-2nd} for the first- and second-order discretization, respectively, with the corresponding convergence histories shown in \cref{fig:heat-his-1st} and \cref{fig:heat-his-2nd}. As in the lid-driven cavity flow, FIM-3 relies on empirical tuning of $\gamma_2$, whereas FIM-4 employs one Newton-multigrid step in each alternating iteration without case-dependent tuning. 
Compared with the tuned FIM-3, FIM-4 is moderately less efficient at $\Kn=0.1$ but becomes significantly faster at smaller Knudsen numbers, particularly for the second-order discretization. Moreover, its efficiency could potentially be further improved by adopting a Jacobian simplification strategy similar to that proposed in \cite{Hu2026SWEs}.

\begin{table}[!htb]
	\centering
	\caption{Performance results of the IMEX-SGS, FIM-3 and FIM-4 for the heat transfer
		in a bottom-heated cavity at various $\Kn$ with $N_x = N_y = 64$ and first-order
		discretization.}
	\label{tab:heat-1st}
	\begin{tabular}{ccccccccc}
		\toprule
		\multirow{2}{*}{$\Kn$} & \multirow{2}{*}{$M$} &
		\multicolumn{2}{c}{IMEX-SGS} &
		\multicolumn{2}{c}{FIM-3\textsuperscript{$\dagger$}} &
		\multicolumn{2}{c}{FIM-4} &
		\multirow{2}{*}{Speedup \textsuperscript{$\ddagger$}} \\
		\cmidrule(lr){3-4} \cmidrule(lr){5-6} \cmidrule(lr){7-8}
		& & Steps & Time (s) &
		Steps & Time (s) &
		Steps & Time (s) & \\
		\midrule
		0.1   & 10 & 2264 & 4707.65 & 350 & 1037.54 & 363 & 1251.78 & 0.83 \\
		& 9  & 2208 & 3564.67 & 341 & 741.98  & 362 & 962.90  & 0.77 \\
		0.01  & 8  & 3170 & 3897.69 & 88  & 241.27  & 62  & 126.31  & 1.91 \\
		& 7  & 3008 & 2847.76 & 84  & 181.50  & 59  & 90.20   & 2.01 \\
		0.001 & 6  & 3672 & 2444.45 & 52  & 117.56  & 65  & 72.77   & 1.62 \\
		& 5  & 3549 & 1804.98 & 51  & 88.75   & 64  & 55.11   & 1.61 \\
		\bottomrule
	\end{tabular}
	
	\smallskip
	\begin{tabular}{@{}p{\linewidth}@{}}
		\textsuperscript{$\dagger$} The parameter $\gamma_2$ in FIM-3 is set to $10$,
		$40$, and $80$ for $\Kn = 0.1$, $0.01$, and $0.001$, respectively. \\
		\textsuperscript{$\ddagger$} The speedup is the acceleration ratio of FIM-4 relative to FIM-3.
	\end{tabular}
\end{table}

\begin{figure}[!htb]
	\centering
	\subfloat[$\Kn=0.1$]{
	   \includegraphics[width=0.31\textwidth]{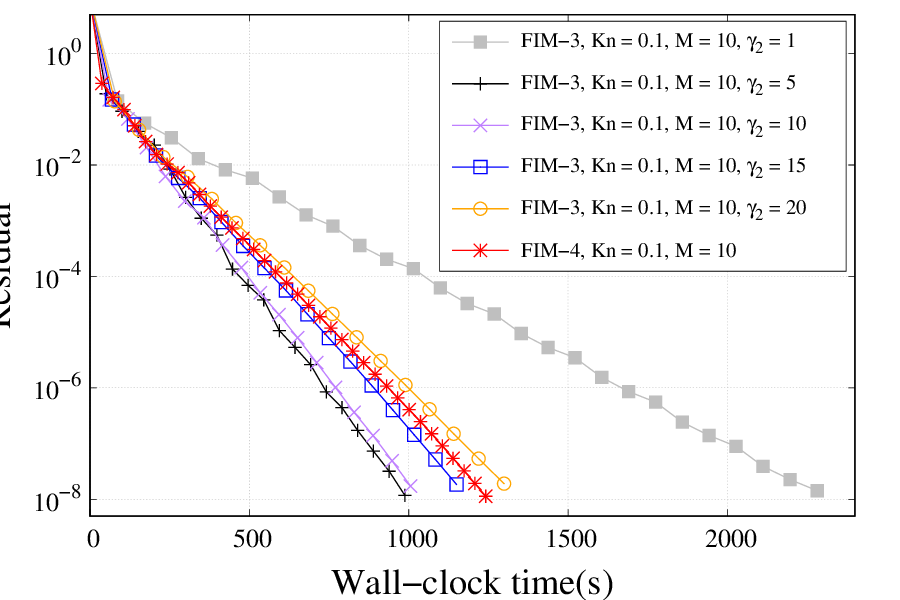}
	}\hfill
	\subfloat[$\Kn=0.01$]{
	   \includegraphics[width=0.31\textwidth]{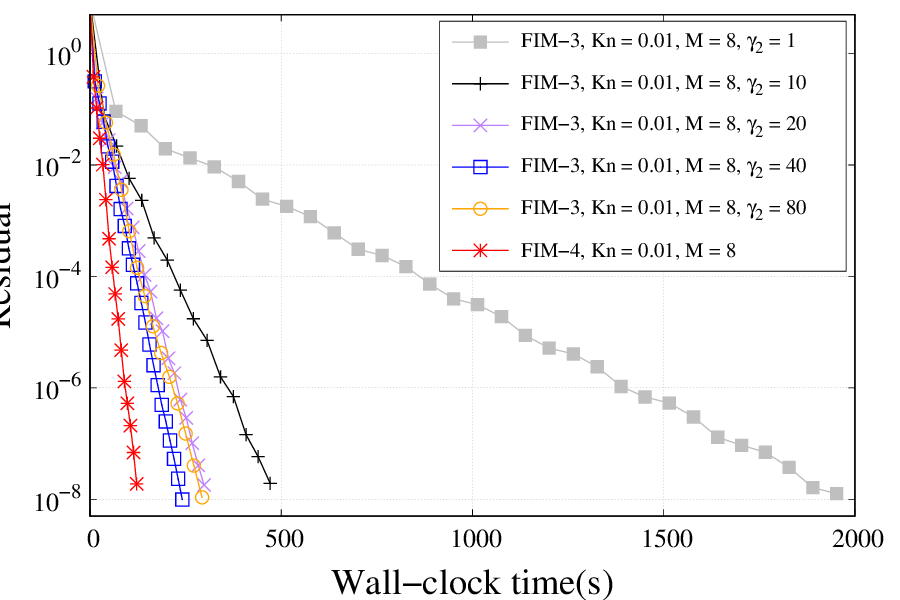}
	}\hfill
	\subfloat[$\Kn=0.001$]{
	   \includegraphics[width=0.31\textwidth]{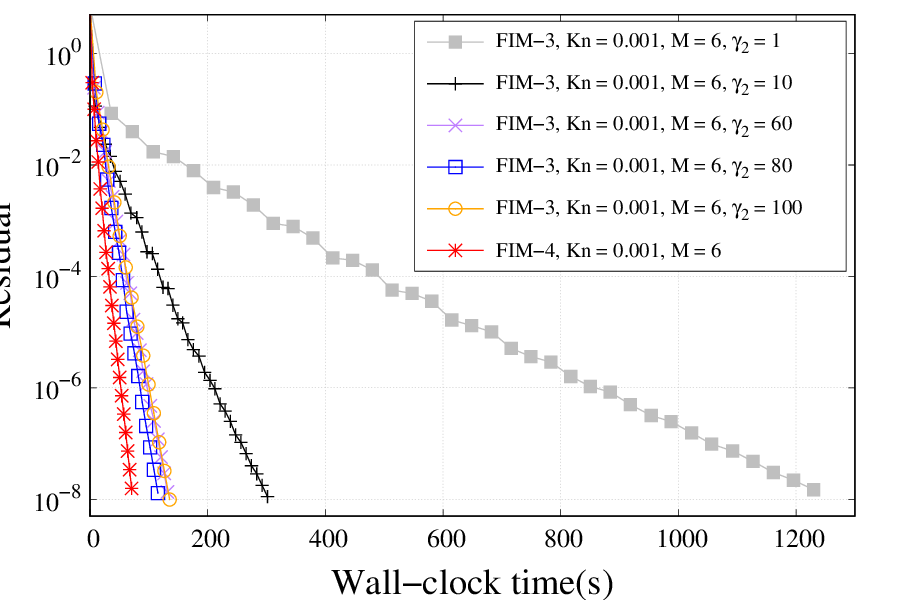}
	}
	\caption{Convergence histories for the heat transfer with $N_x = N_y = 64$
	and first-order spatial discretization.}
	\label{fig:heat-his-1st}
\end{figure}

\begin{table}[!htb]
	\centering
	\caption{Performance results of the IMEX-SGS, FIM-3 and FIM-4 for the heat transfer
		in a bottom-heated cavity at various $\Kn$ with $N_x = N_y = 64$ and second-order
		discretization.}
	\label{tab:heat-2nd}
	\begin{tabular}{ccccccccc}
		\toprule
		\multirow{2}{*}{$\Kn$} & \multirow{2}{*}{$M$} &
		\multicolumn{2}{c}{IMEX-SGS} &
		\multicolumn{2}{c}{FIM-3\textsuperscript{$\dagger$}} &
		\multicolumn{2}{c}{FIM-4} &
		\multirow{2}{*}{Speedup \textsuperscript{$\ddagger$}} \\
		\cmidrule(lr){3-4} \cmidrule(lr){5-6} \cmidrule(lr){7-8}
		& & Steps & Time (s) &
		Steps & Time (s) &
		Steps & Time (s) & \\
		\midrule
		0.1   & 10 & 3952  & 4948.82 & 669 & 1602.53 & 685 & 2518.59 & 0.64 \\
		& 9  & 3368  & 3463.51 & 631 & 1148.13 & 613 & 1641.27 & 0.70 \\
		0.01  & 8  & 10725 & 8313.79 & 299 & 842.87  & 126 & 280.90  & 3.00 \\
		& 7  & 9814  & 5490.72 & 274 & 599.04  & 116 & 191.66  & 3.13 \\
		0.001 & 6  & 69502 & 28705.4 & 966 & 3438.18 & 765 & 1052.61 & 3.27 \\
		& 5  & 61445 & 17354.4 & 861 & 2672.26 & 691 & 782.22  & 3.42 \\
		\bottomrule
	\end{tabular}
	
	\smallskip
	\begin{tabular}{@{}p{\linewidth}@{}}
		\textsuperscript{$\dagger$} The parameter $\gamma_2$ in FIM-3 is set to $10$,
		$40$, and $80$ for $\Kn = 0.1$, $0.01$, and $0.001$, respectively. \\
		\textsuperscript{$\ddagger$} The speedup is the acceleration ratio of FIM-4 relative to FIM-3.
	\end{tabular}
\end{table}

\begin{figure}[!htb]
	 \centering
	 \subfloat[$\Kn=0.1$]{
		 \includegraphics[width=0.31\textwidth]{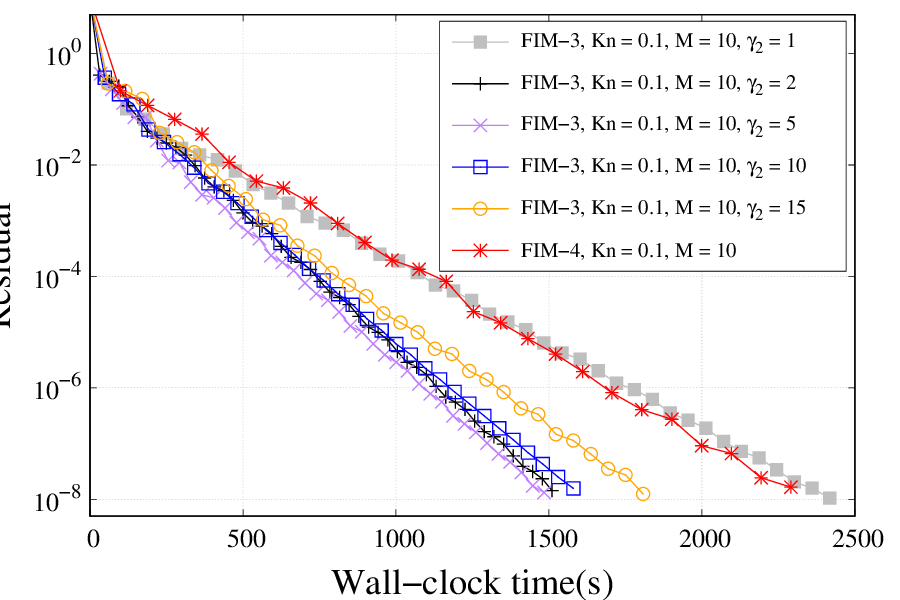}
	 }\hfill
	 \subfloat[$\Kn=0.01$]{
	    \includegraphics[width=0.31\textwidth]{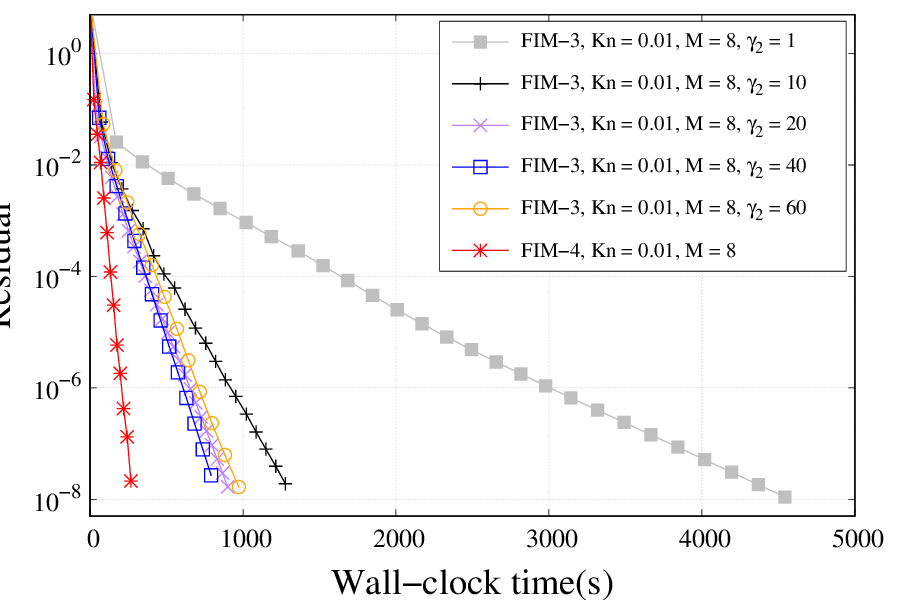}
	 }\hfill
	 \subfloat[$\Kn=0.001$]{
	    \includegraphics[width=0.31\textwidth]{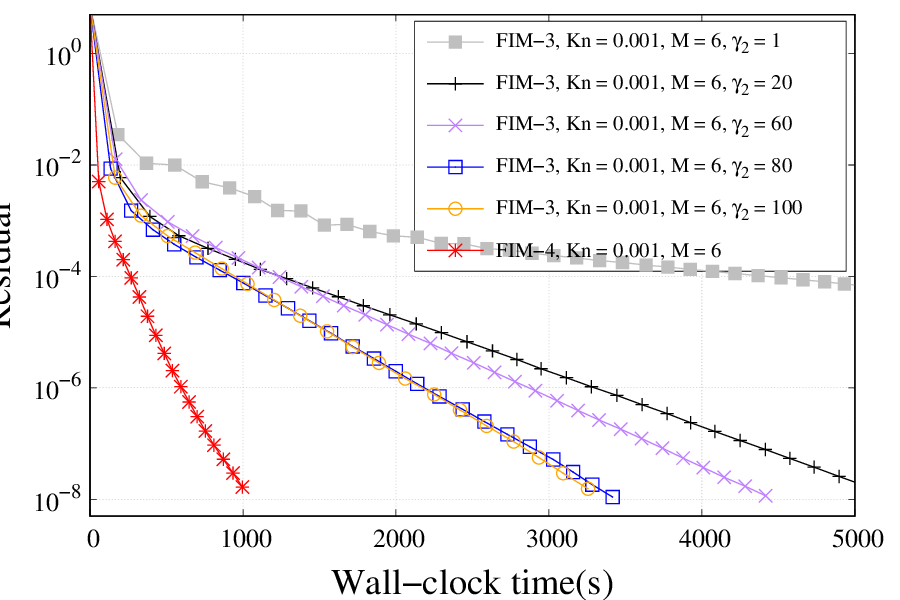}
	 }
	 \caption{Convergence histories for the heat transfer with $N_x = N_y = 64$
	 and second-order spatial discretization.}
	 \label{fig:heat-his-2nd}
\end{figure}

\section{Conclusion}
\label{sec:conclusion}
The fast iterative moment (FIM) method developed in \cite{Li2026} provides an efficient framework for accelerating steady-state solutions of near-continuum gas flows by alternating between the high-order moment system and the consistent hydrodynamic equations. In this work, the FIM solver is extended to the Shakhov model and further improved through second-order finite volume spatial discretization, an IMEX scheme with implicit treatment of the collision term, and a Newton-multigrid hydrodynamic solver equipped with a new regularization strategy. 
A local Fourier analysis of a linearized one-dimensional BGK model is performed to investigate the convergence behavior of the original FIM method. The analysis reveals the need for more hydrodynamic iterations at smaller Knudsen numbers and their slower convergence on refined grids, motivating the replacement of the time-integration-based hydrodynamic solver by a Newton-multigrid method.
The resulting solver also forms an integrated multilevel method, in which the FIM iteration and Newton-multigrid correction provide acceleration in the velocity and spatial spaces, respectively.

Numerical experiments for planar Couette flow, lid-driven cavity flow, and heat transfer in a bottom-heated cavity demonstrate that the improved FIM solver achieves significantly enhanced efficiency and robustness, particularly at small Knudsen numbers. Generally, one Newton-multigrid step in each alternating iteration is sufficient to maintain fast convergence, thereby substantially reducing the need for case-dependent parameter tuning. The numerical results also show that the proposed solver offers a greater efficiency advantage with second-order spatial discretization than with first-order discretization at small Knudsen numbers. Overall, the proposed method provides a practical and efficient framework for steady-state computation of high-order moment systems in near-continuum gas flows.

\appendix

\section{Expression of the moment system \cref{eq:moment-system}}
\label{sec:expression_HME}
The expression of the moment system \cref{eq:moment-system} is
\begin{align}
  \label{eq:mnt-eq}
  \begin{split}
    & \sum_{j=1}^D \Bigg[ \left( \theta \frac{\partial f_{\alpha -
          e_j}}{\partial x_j} + u_j \frac{\partial
        f_{\alpha}}{\partial x_j} + (1-\delta_{|\alpha|,M})(\alpha_j +
      1) \frac{\partial f_{\alpha+e_j}}{\partial x_j} \right) \\ &+
    \sum_{d=1}^3 \frac{\partial u_d}{\partial x_j} \left( \theta
      f_{\alpha-e_d-e_j} + u_j f_{\alpha-e_d} +
      (1-\delta_{|\alpha|,M}) (\alpha_j + 1) f_{\alpha-e_d+e_j}
    \right) \\ &+ \frac{1}{2} \frac{\partial \theta}{\partial x_j}
    \sum_{d=1}^3 \left( \theta f_{\alpha-2e_d-e_j} + u_j
      f_{\alpha-2e_d} + (1-\delta_{|\alpha|,M}) (\alpha_j + 1)
		f_{\alpha-2e_d+e_j} \right) \Bigg] = \bar{\nu} \bra{f_{\alpha}^{\text{E}} -
		f_{\alpha}} = \mQ_{\alpha}, \quad |\alpha| \leq M,
    \end{split}
\end{align}
where $\mQ_{\alpha}$ denotes the $\alpha$th Hermite expansion coefficient of the collision term $\mQ(f)$, and $f_{\alpha}$ is taken as zero if any component of $\alpha$ is negative.
For the Shakhov model,
\begin{align}
	\label{eq:f_E_alpha}
	f_{\alpha}^{\text{E}} = 
	\left \{
	\begin{aligned}
		& \rho, &  &  \text{if } \alpha = 0, \\
		& \frac{1-\Pr}{5} q_j, &  &  \text{if } \alpha = 2 e_i + e_j, \quad
		i,j = 1, 2, 3, \\
		& 0, &	& 	\text{otherwise}.
	\end{aligned} \right.
\end{align}

\section{The hydrodynamic boundary conditions}
\label{sec:boundary-conditions}

Assume the solid wall has temperature $\theta^W$ and velocity $\bu^W$.
For the case where the outward unit normal vector is
$\bn=e_1=(1,0,0)^{\mathsf{T}}$, let
$[\rho^+,\bu^+,\theta^+]$ and $\tilde f^+$ denote the updated interior
macroscopic quantities and the moment solution, respectively. The ghost-cell macroscopic quantities
$[\rho^g,\bu^g,\theta^g,\bsigma^g,\bq^g]$
are obtained from the formulae in \cite{Cai2013microflow} as
\begin{align}
	\begin{aligned}
		\rho^{g} & = \rho^{+}, \quad \bu^{g} = \bu^{+} + 2(u_1^{W} - u_{1}^{+}) e_1,
		\quad \theta^{g} = \theta^{+}, \quad \sigma_{kk}^{g} = \tilde{\sigma}_{kk}^{+}
		(k = 1,2,3), \quad \sigma_{23}^{g} = \tilde{\sigma}_{23}^{+}, \\
		\sigma_{12}^{g} & = \frac{4\chi}{2 - \chi} \bra{\sum_{k=0}^{\lfloor (M-1)/2
					\rfloor} S(1,2k) \bra{\theta^{+}}^{\frac{1}{2}-k} \tilde{f}_{(2k,1,0)}^{+} +
			\bra{u_2^{+} - u_{2}^{W}} \sum_{k=0}^{\lfloor M/2 \rfloor} S(1,2k)
			\bra{\theta^{+}}^{\frac{1}{2}-k} \tilde{f}_{(2k,0,0)}^{+}} - \tilde{\sigma}_{12}^{+}, \\
		\sigma_{13}^{g} & = \frac{4\chi}{2 - \chi} \bra{\sum_{k=0}^{\lfloor (M-1)/2
					\rfloor} S(1,2k) \bra{\theta^{+}}^{\frac{1}{2}-k} \tilde{f}_{(2k,0,1)}^{+} +
			\bra{u_3^{+} - u_{3}^{W}} \sum_{k=0}^{\lfloor M/2 \rfloor} S(1,2k)
			\bra{\theta^{+}}^{\frac{1}{2}-k} \tilde{f}_{(2k,0,0)}^{+}} - \tilde{\sigma}_{13}^{+}, \\
		q_{1}^{g} & = \frac{12 \chi}{2 - \chi} \bra{\sum_{k=0}^{\lfloor M/2
					\rfloor} S(3,2k) \bra{\theta^{+}}^{\frac{3}{2}-k}
			\tilde{f}_{(2k,0,0)}^{+} + \bra{\frac{\theta^{+}}{2} -
				\frac{\theta^W}{3}} \sum_{k=0}^{\lfloor M/2 \rfloor} S(1,2k)
			\bra{\theta^{+}}^{\frac{1}{2}-k} \tilde{f}_{(2k,0,0)}^{+}}
		- 3 \tilde{f}_{(3,0,0)}^{+} \\
		& + \frac{4 \chi}{2 - \chi} \bra{\sum_{k=0}^{\lfloor (M-2)/2 \rfloor} S(1,2k)
			\bra{\theta^{+}}^{\frac{1}{2}-k} \tilde{f}_{(2k,2,0)}^{+}
			- \frac{1}{2} \bra{(\theta^{W} - \theta^{+}) + (u_{2}^{W} -
				u_{2}^{+})^2} \sum_{k=0}^{\lfloor M/2 \rfloor} S(1,2k)
			\bra{\theta^{+}}^{\frac{1}{2}-k} \tilde{f}_{(2k,0,0)}^{+}} -
		\tilde{f}_{(1,2,0)}^{+} \\
		& + \frac{4 \chi}{2 - \chi} \bra{\sum_{k=0}^{\lfloor (M-2)/2 \rfloor} S(1,2k)
			\bra{\theta^{+}}^{\frac{1}{2}-k} \tilde{f}_{(2k,0,2)}^{+}
			- \frac{1}{2} \bra{(\theta^{W} - \theta^{+}) + (u_{3}^{W} -
				u_{3}^{+})^2} \sum_{k=0}^{\lfloor M/2 \rfloor} S(1,2k)
			\bra{\theta^{+}}^{\frac{1}{2}-k} \tilde{f}_{(2k,0,0)}^{+}} -
		\tilde{f}_{(1,0,2)}^{+}, \\
		q_{2}^{g} & = \tilde{q}_2^{+}, \quad q_{3}^{g} = \tilde{q}_3^{+},
	\end{aligned}
\end{align}
where $\chi =1$ is the accommodation coefficient.
The stress tensor $\tilde{\bsigma}^{+}$ and heat flux $\tilde{\bq}^{+}$ are
computed from the expansion coefficients $\tilde{f}_{\alpha}^{+}$ of
$\tilde{f}^{+}$.
The function $S(\cdot,\cdot)$ is defined recursively as
\begin{equation}
	\label{eq:S}
	S(r,s) =
	\left\{
	\begin{aligned}
		& 1/2, & r=s=0,\\
		& K(1,s-1), & r=0,\ s\neq0,\\
		& K(r,0), & r\neq0,\ s=0,\\
		& K(r,s)+S(r-1,s-1)s/r, & \mbox{otherwise},
	\end{aligned}
	\right.
\end{equation}
where the auxiliary function $K(\cdot,\cdot)$ is defined as
\begin{equation}
	\label{eq:K}
	K(r,s)=
	\left\{
	\begin{aligned}
		& \frac{(-1)^{(r+s-1)/2}(s-1)!!}
		{\sqrt{2\pi}\,r2^{(r-1)/2}((r-1)/2)!},
		& \mbox{$r$ is odd and $s$ is even},\\
		&0,
		&\mbox{otherwise}.
	\end{aligned}
	\right.
\end{equation}
\section{Derivation of the Fourier convergence estimates}
\label{app:fim3-fourier-analysis}

This appendix gives the details behind the Fourier-analysis results stated in
\cref{sec:fim3-convergence-analysis}. We use the moment system
\eqref{eq:fim3analysis-model}--\eqref{eq:fim3analysis-matrices}, the
hydrodynamic variables in \eqref{eq:fim3analysis-hydroblock}, and the SGS
iterations \eqref{eq:fim3analysis-kin-forward-state}--\eqref{eq:fim3analysis-hydro-backward-state}.
Throughout the derivation, the parameters satisfy
\begin{equation}
 0<\nu<1,
 \qquad C_{M+1}\geq\sqrt3.
 \label{eq:fim3app-parameter-range}
\end{equation}
The first condition is the CFL restriction for the model, while the second holds automatically for $M\geq 2$.

\subsection{Error evolution and the error-propagation matrix}
\label{subsec:fim3app-branch-separation}

We first clarify the relationship between $\delta \boldsymbol u_j^{(n+1)}$ and $\delta \boldsymbol u_j^{(n)}$, defined in \eqref{eq:fim3analysis-outer-error}. Recall that $\boldsymbol u_j$ is the exact discrete solution in \eqref{eq:fim3analysis-discrete-equation}. For one moment SGS iteration, define the errors before the forward sweep, after the forward sweep, and after the backward sweep by
\begin{equation}
 \boldsymbol e_j^{n,s}=\boldsymbol u_j^{(n,s)}-\boldsymbol u_j,
 \qquad
 \boldsymbol e_j^{n,s+1/2}=\boldsymbol u_j^{(n,s+1/2)}-\boldsymbol u_j,
 \qquad
 \boldsymbol e_j^{n,s+1}=\boldsymbol u_j^{(n,s+1)}-\boldsymbol u_j.
 \label{eq:fim3app-sweep-errors}
\end{equation}
Then $\delta \boldsymbol u_j^{(n)} = \boldsymbol e_j^{n,0}$. Subtracting \eqref{eq:fim3analysis-discrete-equation} from the two moment SGS updates \eqref{eq:fim3analysis-kin-forward-state} and \eqref{eq:fim3analysis-kin-backward-state} gives the moment error evolution
\begin{align}
 \left(\mathbf I-\frac{\omega}{\varepsilon}\mathbf L\right)
 \boldsymbol e_j^{n,s+1/2}
 &=(1-\nu)\boldsymbol e_j^{n,s}
 +\frac{\nu}{2}\left(\mathbf I+\frac{\mathbf A}{C_{M+1}}\right)
 \boldsymbol e_{j-1}^{n,s+1/2}
 +\frac{\nu}{2}\left(\mathbf I-\frac{\mathbf A}{C_{M+1}}\right)
 \boldsymbol e_{j+1}^{n,s},
 \label{eq:fim3app-kin-forward-error}\\
 \left(\mathbf I-\frac{\omega}{\varepsilon}\mathbf L\right)
 \boldsymbol e_j^{n,s+1}
 &=(1-\nu)\boldsymbol e_j^{n,s+1/2}
 +\frac{\nu}{2}\left(\mathbf I+\frac{\mathbf A}{C_{M+1}}\right)
 \boldsymbol e_{j-1}^{n,s+1/2}
 +\frac{\nu}{2}\left(\mathbf I-\frac{\mathbf A}{C_{M+1}}\right)
 \boldsymbol e_{j+1}^{n,s+1}.
 \label{eq:fim3app-kin-backward-error}
\end{align}
For later use, we will split each moment error into its first three and last $M-2$
components:
\begin{equation}
 \boldsymbol e_j^{n,r}
 =\begin{pmatrix}
   \boldsymbol H_j^{n,r}\\
   \boldsymbol K_j^{n,r}
  \end{pmatrix},
 \quad
 \boldsymbol H_j^{n,r}
 =\begin{pmatrix}e_{0,j}^{n,r}\\e_{1,j}^{n,r}\\e_{2,j}^{n,r}\end{pmatrix},
 \quad
 \boldsymbol K_j^{n,r}
 =\begin{pmatrix}e_{3,j}^{n,r}\\\vdots\\e_{M,j}^{n,r}\end{pmatrix},
 \qquad r=0,\frac12,1,\frac32,\ldots,\gamma_1.
 \label{eq:fim3app-moment-error-split}
\end{equation}

For the hydrodynamic SGS stage, similarly define
\begin{equation}
 \boldsymbol E_j^{n,\ell}
 :=\boldsymbol U_j^{(n,\ell)}-\boldsymbol U_j,
 \qquad
 \boldsymbol U_j=(u_{j,0},u_{j,1},u_{j,2})^{\mathsf{T}},
 \qquad
 \ell=0,\frac12,1,\frac32,\ldots,\gamma_2.
 \label{eq:fim3app-hydro-error}
\end{equation}
The initialization of the hydrodynamic stage $\boldsymbol{U}_j^{(n,0)} = (u_{j,0}^{(n,\gamma_1)}, u_{j,1}^{(n,\gamma_1)}, u_{j,2}^{(n,\gamma_1)})^{\mathsf{T}}$ gives $\boldsymbol E_j^{n,0}=\boldsymbol H_j^{n,\gamma_1}$, and the two Hydro-SGS updates \eqref{eq:fim3analysis-hydro-forward-state} and \eqref{eq:fim3analysis-hydro-backward-state} give
\begin{align}
 \boldsymbol E_j^{n,\ell+1/2}={}&
 (1-\nu)\boldsymbol E_j^{n,\ell}
 +\frac{\nu}{2}\left(\mathbf I+\frac{\mathbf A_h}{C_{M+1}}\right)
 \boldsymbol E_{j-1}^{n,\ell+1/2}
 +\frac{\nu}{2}\left(\mathbf I-\frac{\mathbf A_h}{C_{M+1}}\right)
 \boldsymbol E_{j+1}^{n,\ell}
-\frac{\nu}{2C_{M+1}}\boldsymbol b
 \left(e_{3,j+1}^{n,\gamma_1}-e_{3,j-1}^{n,\gamma_1}\right),
 \label{eq:fim3app-hydro-forward-error}\\
 \boldsymbol E_j^{n,\ell+1}={}&
 (1-\nu)\boldsymbol E_j^{n,\ell+1/2}
 +\frac{\nu}{2}\left(\mathbf I+\frac{\mathbf A_h}{C_{M+1}}\right)
 \boldsymbol E_{j-1}^{n,\ell+1/2}
 +\frac{\nu}{2}\left(\mathbf I-\frac{\mathbf A_h}{C_{M+1}}\right)
 \boldsymbol E_{j+1}^{n,\ell+1}
-\frac{\nu}{2C_{M+1}}\boldsymbol b
 \left(e_{3,j+1}^{n,\gamma_1}-e_{3,j-1}^{n,\gamma_1}\right).
 \label{eq:fim3app-hydro-backward-error}
\end{align}
Finally, the reconstruction step \eqref{eq:fim3analysis-reconstruction} gives the complete outer error after one FIM
iteration:
\begin{equation}
  \delta\boldsymbol u_j^{(n+1)}
  = \begin{pmatrix}
    \boldsymbol E_j^{n,\gamma_2}\\
    \boldsymbol K_j^{n,\gamma_1}
  \end{pmatrix}.
 \label{eq:fim3app-reconstructed-error}
\end{equation}

We now apply the specific form of $\delta \boldsymbol{u}_j^{(n)}$ in \eqref{eq:fim3analysis-full-fourier-error} to find the error-propagation matrix $\mathbf T(\varepsilon,k)$ in \eqref{eq:fim3analysis-fourier-iteration-matrix}. Plugging $\delta \boldsymbol u_j^{(n)} = \widehat{\boldsymbol e}^{(n)}(k) e^{\mathrm i j k\Delta x}$ into \eqref{eq:fim3app-kin-forward-error}--\eqref{eq:fim3app-hydro-backward-error} gives
\begin{equation}
  \widehat{\boldsymbol e}^{(n+1)}(k) = \underbrace{\Big([\mathbf{M}_{\mathrm{hb}}(k)]^{-1}\mathbf{N}_{\mathrm{hb}}(k) [\mathbf{M}_{\mathrm{hf}}(k)]^{-1}\mathbf{N}_{\mathrm{hf}}(k) \Big)^{\gamma_2} \Big([\mathbf{M}_{\mathrm{kb}}(\varepsilon,k)]^{-1}\mathbf{N}_{\mathrm{kb}}(k) [\mathbf{M}_{\mathrm{kf}}(\varepsilon,k)]^{-1}\mathbf{N}_{\mathrm{kf}}(k)\Big)^{\gamma_1}}_{\mathbf T(\varepsilon,k)}\widehat{\boldsymbol e}^{(n)}(k),
  \label{eq:fim3app-fourier-iteration-matrix}
\end{equation}
where $\mathbf{M}_{\mathrm{kf}}$, $\mathbf{N}_{\mathrm{kf}}$, $\mathbf{M}_{\mathrm{kb}}$, $\mathbf{N}_{\mathrm{kb}}$ arise from the moment SGS stage:
\begin{gather*}
  \mathbf M_{\mathrm{kf}}(\varepsilon,k)
  :=\mathbf I-\frac{\omega}{\varepsilon}\mathbf L
  -\frac{\nu}{2}\left(\mathbf I+\frac{\mathbf A}{C_{M+1}}\right)
  e^{-\mathrm ik\Delta x}, \qquad
  \mathbf N_{\mathrm{kf}}(k)
  :=(1-\nu)\mathbf I
  +\frac{\nu}{2}\left(\mathbf I-\frac{\mathbf A}{C_{M+1}}\right)
  e^{\mathrm i k\Delta x}, \\
  \mathbf M_{\mathrm{kb}}(\varepsilon,k)
  :=\mathbf I-\frac{\omega}{\varepsilon}\mathbf L
  -\frac{\nu}{2}\left(\mathbf I-\frac{\mathbf A}{C_{M+1}}\right)
  e^{\mathrm ik \Delta x}, \qquad
  \mathbf N_{\mathrm{kb}}(k)
  :=(1-\nu)\mathbf I
  +\frac{\nu}{2}\left(\mathbf I+\frac{\mathbf A}{C_{M+1}}\right)
  e^{-\mathrm ik \Delta x},
\end{gather*}
and $\mathbf{M}_{\mathrm{hf}}$, $\mathbf{N}_{\mathrm{hf}}$, $\mathbf{M}_{\mathrm{hb}}$, $\mathbf{N}_{\mathrm{hb}}$ arise from the hydrodynamic SGS stage. The detailed expressions of these four matrices will be omitted, and only their block structures are shown below:
\begin{equation}
  \mathbf M_{\mathrm{hf}}(k)
  :=\begin{pmatrix}
    * & \mathbf0\\
    \mathbf0 & \mathbf I_{M-2}
  \end{pmatrix}, \quad
  \mathbf N_{\mathrm{hf}}(k)
  :=\begin{pmatrix}
    * & * \\
    \mathbf0 & \mathbf I_{M-2}
  \end{pmatrix},\quad
  \mathbf M_{\mathrm{hb}}(k)
  :=\begin{pmatrix}
    * & \mathbf0\\
    \mathbf0 & \mathbf I_{M-2}
  \end{pmatrix},\quad
  \mathbf N_{\mathrm{hb}}(k)
  :=\begin{pmatrix}
    * & * \\
    \mathbf0 & \mathbf I_{M-2}
  \end{pmatrix}.
  \label{eq:MN-block-structure}
\end{equation}
Here $*$ denotes a potentially nonzero block.

\subsection{Order analysis of the error-propagation matrix}
We now write the matrices $\mathbf{T}(\varepsilon,k)$, $\mathbf{M}_{\mathrm{kf}}(\varepsilon,k)$ and $\mathbf{M}_{\mathrm{kb}}(\varepsilon,k)$ in block form, separating the hydrodynamic and nonequilibrium components:
\begin{gather*}
  \mathbf T(\varepsilon,k)
  =\begin{pmatrix}
    \mathbf T^{\mathrm{11}}(\varepsilon,k) & \mathbf T^{\mathrm{12}}(\varepsilon,k)\\
    \mathbf T^{\mathrm{21}}(\varepsilon,k) & \mathbf T^{\mathrm{22}}(\varepsilon,k)
  \end{pmatrix},\qquad
  \mathbf M_{\mathrm{kf}}(\varepsilon,k)
  =\begin{pmatrix}
    \mathbf M_{\mathrm{kf}}^{\mathrm{11}}(k) & \mathbf M_{\mathrm{kf}}^{\mathrm{12}}(k)\\[3pt]
    \mathbf M_{\mathrm{kf}}^{\mathrm{21}}(k) & \mathbf M_{\mathrm{kf}}^{\mathrm{22}}(\varepsilon,k)
  \end{pmatrix},\qquad
  \mathbf M_{\mathrm{kb}}(\varepsilon,k)
  =\begin{pmatrix}
    \mathbf M_{\mathrm{kb}}^{\mathrm{11}}(k) & \mathbf M_{\mathrm{kb}}^{\mathrm{12}}(k)\\[3pt]
    \mathbf M_{\mathrm{kb}}^{\mathrm{21}}(k) & \mathbf M_{\mathrm{kb}}^{\mathrm{22}}(\varepsilon,k)
  \end{pmatrix},
\end{gather*}
where we have incorporated the fact that some blocks of $\mathbf{M}_{\mathrm{kf}}$ and $\mathbf{M}_{\mathrm{kb}}$ are independent of $\varepsilon$. We would like to show that $\mathbf T^{\mathrm{21}}(\varepsilon,k)$ and $\mathbf T^{\mathrm{22}}(\varepsilon,k)$ are $O(\varepsilon)$, which implies that the nonequilibrium components of a bounded error become $O(\varepsilon)$ after one FIM iteration.

Notice that $\mathbf M_{\mathrm{kf}}^{\mathrm{22}}(\varepsilon,k)$ has the form
\begin{displaymath}
  \mathbf{M}_{\mathrm{kf}}^{\mathrm{22}}(\varepsilon,k) = \frac{\omega}{\varepsilon}\mathbf I_{M-2} + \widetilde{\mathbf{M}}_{\mathrm{kf}}^{\mathrm{22}}(k),
\end{displaymath}
and $\mathbf{M}_{\mathrm{kf}}^{\mathrm{11}}(k)$ is invertible for all $k$ under the parameter range \eqref{eq:fim3app-parameter-range}.
We can then define the Schur complement of the upper-left block by
\begin{displaymath}
  \mathbf S_{\mathrm{kf}}(\varepsilon,k)
  :=\frac{\omega}{\varepsilon} \mathbf{I}_{M-2}+\widetilde{\mathbf M}_{\mathrm{kf}}^{\mathrm{22}}(k)
  -\mathbf M_{\mathrm{kf}}^{\mathrm{21}}(k)
  [\mathbf M_{\mathrm{kf}}^{\mathrm{11}}(k)]^{-1}
  \mathbf M_{\mathrm{kf}}^{\mathrm{12}}(k).
\end{displaymath}
One sees that $[\mathbf{S}_{\mathrm{kf}}(\varepsilon,k)]^{-1} = O(\varepsilon)$ as $\varepsilon \to 0$. Thus, by suppressing the dependence on $\varepsilon$ and $k$, the block inversion formula gives
\begin{equation}
  \mathbf M_{\mathrm{kf}}^{-1}
  =\begin{pmatrix}
    (\mathbf M_{\mathrm{kf}}^{\mathrm{11}})^{-1}
    +(\mathbf M_{\mathrm{kf}}^{\mathrm{11}})^{-1}
    \mathbf M_{\mathrm{kf}}^{\mathrm{12}}\mathbf S_{\mathrm{kf}}^{-1}
    \mathbf M_{\mathrm{kf}}^{\mathrm{21}}
    (\mathbf M_{\mathrm{kf}}^{\mathrm{11}})^{-1}
    &-(\mathbf M_{\mathrm{kf}}^{\mathrm{11}})^{-1}
    \mathbf M_{\mathrm{kf}}^{\mathrm{12}}
    \mathbf S_{\mathrm{kf}}^{-1}\\[3pt]
    -\mathbf S_{\mathrm{kf}}^{-1}\mathbf M_{\mathrm{kf}}^{\mathrm{21}}
    (\mathbf M_{\mathrm{kf}}^{\mathrm{11}})^{-1}
    &\mathbf S_{\mathrm{kf}}^{-1}
  \end{pmatrix} = \begin{pmatrix}
    O(1) & O(\varepsilon)\\
    O(\varepsilon) & O(\varepsilon)
  \end{pmatrix}.
  \label{eq:fim3app-Mkf-inverse-order}
\end{equation}
Similarly, one can also show that blocks of $[\mathbf{M}_{\mathrm{kb}}(\varepsilon,k)]^{-1}$ are of the same order. Together with \eqref{eq:fim3app-fourier-iteration-matrix} and \eqref{eq:MN-block-structure}, one can conclude that
\begin{displaymath}
  \mathbf T^{\mathrm{11}}(\varepsilon,k) = O(1),\qquad
  \mathbf T^{\mathrm{12}}(\varepsilon,k) = O(1),\qquad
  \mathbf T^{\mathrm{21}}(\varepsilon,k) = O(\varepsilon),\qquad
  \mathbf T^{\mathrm{22}}(\varepsilon,k) = O(\varepsilon).
\end{displaymath}
This allows us to analyze the leading-order amplification factors by considering only the hydrodynamic components of the error, which is the focus of the next two subsections.

\subsection{Leading-order amplification factors}
\label{subsec:fim3app-leading-factor}

The preceding subsection shows that the lower block row of $\mathbf T(\varepsilon,k)$ is $O(\varepsilon)$. For any eigenvalue branch with a nonzero limit, the last $M-2$ components of a normalized eigenvector are therefore $O(\varepsilon)$. Thus its leading part can be nonzero only in the first three components, and we can perform formal asymptotic expansions:
\begin{equation}
 \boldsymbol H_j^{n,r}
 =\boldsymbol H_j^{n,r,(0)}+\varepsilon \boldsymbol H_j^{n,r,(1)}+\cdots,
 \qquad
 \boldsymbol K_j^{n,r}=\varepsilon \boldsymbol K_j^{n,r,(1)}+\cdots,
 \qquad
 \boldsymbol E_j^{n,\ell} = \boldsymbol E_j^{n,\ell,(0)}+\varepsilon \boldsymbol E_j^{n,\ell,(1)}+\cdots.
 \label{eq:fim3app-leading-HKE-expansion}
\end{equation}
The next step is to insert these expansions into \eqref{eq:fim3app-kin-forward-error}--\eqref{eq:fim3app-kin-backward-error} and \eqref{eq:fim3app-hydro-forward-error}--\eqref{eq:fim3app-hydro-backward-error} and collect the leading-order terms. To avoid writing the same pair of equations twice, let
$(\boldsymbol X_j^0,\boldsymbol X_j^{1/2},\boldsymbol X_j^1)$ denote either
\begin{equation}
  \left(\boldsymbol H_j^{n,s,(0)},
        \boldsymbol H_j^{n,s+1/2,(0)},
        \boldsymbol H_j^{n,s+1,(0)}\right) \quad \text{or} \quad
  \left(\boldsymbol E_j^{n,\ell,(0)},
        \boldsymbol E_j^{n,\ell+1/2,(0)},
        \boldsymbol E_j^{n,\ell+1,(0)}\right).
\end{equation}
In both cases,
\begin{align}
 \boldsymbol X_j^{1/2}={}&
 (1-\nu)\boldsymbol X_j^0
 +\frac{\nu}{2}\left(\mathbf I+\frac{\mathbf A_h}{C_{M+1}}\right)
 \boldsymbol X_{j-1}^{1/2}
 +\frac{\nu}{2}\left(\mathbf I-\frac{\mathbf A_h}{C_{M+1}}\right)
 \boldsymbol X_{j+1}^{0},
 \label{eq:fim3app-forward-leading}\\
 \boldsymbol X_j^{1}={}&
 (1-\nu)\boldsymbol X_j^{1/2}
 +\frac{\nu}{2}\left(\mathbf I+\frac{\mathbf A_h}{C_{M+1}}\right)
 \boldsymbol X_{j-1}^{1/2}
 +\frac{\nu}{2}\left(\mathbf I-\frac{\mathbf A_h}{C_{M+1}}\right)
 \boldsymbol X_{j+1}^{1}.
 \label{eq:fim3app-backward-leading}
\end{align}
It is easy to see that when $\boldsymbol X_j^0 = a \boldsymbol v_{\mu} e^{\mathrm{i} j k \Delta x}$ for any eigenvector $\boldsymbol v_{\mu}$ of $\mathbf A_h$ associated with eigenvalue $\mu$, it holds that
\begin{displaymath}
  \boldsymbol X_j^{1/2} = \phi_{\mu}^{(0)}(k) a \boldsymbol v_{\mu} e^{\mathrm{i} j k \Delta x}, \qquad \boldsymbol X_j^1 = \psi_{\mu}^{(0)}(k) \phi_{\mu}^{(0)}(k) a \boldsymbol v_{\mu} e^{\mathrm{i} j k \Delta x},
\end{displaymath}
where
\begin{displaymath}
  \phi_\mu^{(0)}(k)=\frac{1-\nu+\dfrac{\nu}{2}\left(1-\dfrac{\mu}{C_{M+1}}\right)e^{\mathrm ik\Delta x}}{1-\dfrac{\nu}{2}\left(1+\dfrac{\mu}{C_{M+1}}\right)e^{-\mathrm ik\Delta x}}, \qquad
  \psi_\mu^{(0)}(k)=\frac{1-\nu+\dfrac{\nu}{2}\left(1+\dfrac{\mu}{C_{M+1}}\right)e^{-\mathrm ik\Delta x}}{1-\dfrac{\nu}{2}\left(1-\dfrac{\mu}{C_{M+1}}\right)e^{\mathrm ik\Delta x}}.
\end{displaymath}
Let $\mathcal G_\mu^{(0)}(k) = \phi_\mu^{(0)}(k)\psi_\mu^{(0)}(k)$ as in \eqref{eq:fim3analysis-G}, which represents the leading-order amplification factor of one SGS iteration. Note that $\mathbf{A}_h$ is the $3\times 3$ upper-left block of $\mathbf{A}$ with eigenvalues $\mu = -\sqrt{3}, 0, \sqrt{3}$. Define
\begin{displaymath}
  \mathbf{Q}(k) = \mathbf{V} \begin{pmatrix}
    \mathcal G_{-\sqrt3}^{(0)}(k) & 0 & 0\\
    0 & \mathcal G_0^{(0)}(k) & 0\\
    0 & 0 & \mathcal G_{\sqrt3}^{(0)}(k)
  \end{pmatrix} \mathbf{V}^{-1}, \qquad
  \mathbf{V} = (\boldsymbol v_{-\sqrt3}, \boldsymbol v_0, \boldsymbol v_{\sqrt3}).
\end{displaymath}
The leading order term of the propagation matrix $\mathbf{T}(\varepsilon,k)$ can be expressed as
\begin{displaymath}
  \mathbf{T}^{(0)}(k) =
    \begin{pmatrix} \mathbf{Q}(k) & \mathbf{B}_1(k) \\ \mathbf{0} & \mathbf{I}_{M-2} \end{pmatrix}^{\gamma_2}
    \begin{pmatrix} \mathbf{Q}(k) & \mathbf{B}_2(k) \\ \mathbf{0} & \mathbf{0} \end{pmatrix}^{\gamma_1} = \begin{pmatrix} [\mathbf{Q}(k)]^{\gamma_1+\gamma_2} & [\mathbf{Q}(k)]^{\gamma_1+\gamma_2-1} \mathbf{B}_2(k) \\ \mathbf{0} & \mathbf{0} \end{pmatrix},
\end{displaymath}
where we have dropped $O(\varepsilon)$ terms in the expression of $\mathbf{T}(\varepsilon,k)$ given in \eqref{eq:fim3app-fourier-iteration-matrix} using the block structures \eqref{eq:MN-block-structure} and \eqref{eq:fim3app-Mkf-inverse-order}, and $\mathbf{B}_1$ and $\mathbf{B}_2$ are some matrices whose details are not important for our analysis. Below, we are going to show that at least $M-2$ eigenvalues of $\mathbf{T}(\varepsilon,k) = \mathbf{T}^{(0)}(k) + \varepsilon \widetilde{\mathbf{T}}(\varepsilon,k)$ are $O(\varepsilon)$, where $\widetilde{\mathbf{T}}(\varepsilon,k) = O(1)$.

The proof begins with the diagonalization of $\mathbf{T}^{(0)}(k)$. Let
\begin{displaymath}
  \mathbf P(k)=\begin{pmatrix}
    \mathbf V&-\mathbf R(k)\\
    \mathbf0&\mathbf I_{M-2}
  \end{pmatrix}, \qquad
  \mathbf R(k)
  =\mathbf V\begin{pmatrix}
    h_{-\sqrt3} & 0 & 0\\
    0 & h_0 & 0\\
    0 & 0 & h_{\sqrt3}
  \end{pmatrix}\mathbf V^{-1}\mathbf B_2(k),
  \qquad
  h_\mu=\begin{cases}
    0, & \text{if } \mathcal{G}_\mu^{(0)}(k)=0,\\
    [\mathcal{G}_\mu^{(0)}(k)]^{-1}, & \text{otherwise.}
  \end{cases}
\end{displaymath}
Under the condition $\gamma_1 + \gamma_2 \geq 2$, straightforward calculations show that
\begin{equation}
  [\mathbf P(k)]^{-1}\mathbf T^{(0)}(k)\mathbf P(k) = \mathbf{D}(k)
  :=\operatorname{diag}\left([\mathcal{G}_{-\sqrt3}^{(0)}(k)]^{\gamma_1+\gamma_2},[\mathcal{G}_0^{(0)}(k)]^{\gamma_1+\gamma_2},
    [\mathcal{G}_{\sqrt3}^{(0)}(k)]^{\gamma_1+\gamma_2},\underbrace{0,\ldots,0}_{M-2}\right).
  \label{eq:fim3app-T0-diagonalization}
\end{equation}
Since $\mathbf P(k)$ is independent of $\varepsilon$, it follows that
\begin{equation}
  [\mathbf P(k)]^{-1}\mathbf T(\varepsilon,k)\mathbf P(k)
  =\mathbf D(k)+\varepsilon [\mathbf P(k)]^{-1}\widetilde{\mathbf T}(\varepsilon,k)\mathbf P(k)
  =\mathbf D(k) + O(\varepsilon).
\end{equation}
Let $m$ be the number of zero diagonal entries of $\mathbf D(k)$. Then by the Gershgorin circle theorem, when $\varepsilon$ is sufficiently small,
$m$ eigenvalues of $\mathbf T(\varepsilon,k)$ are contained in the union of the $m$ disks with centers and radii of order $O(\varepsilon)$, which are separated from the other $M+1-m$ disks. Since $m \geq M-2$, we conclude that at least $M-2$ eigenvalues of $\mathbf T(\varepsilon,k)$ are $O(\varepsilon)$. 

To summarize, we have shown that the three hydrodynamic leading-order amplification factors of the FIM iteration are
\begin{equation}
  \Lambda_\mu^{(0)}(k)
  =\left[\mathcal G_\mu^{(0)}(k)\right]^{\gamma_1+\gamma_2},
  \qquad \mu\in\{-\sqrt3,0,\sqrt3\},
  \label{eq:fim3app-leading-Lambda}
\end{equation}
and the remaining $M-2$ eigenvalues are $O(\varepsilon)$. The comparison of these three leading factors is the focus of the next subsection.

\subsection{Ordering of the three leading factors}
\label{subsec:fim3app-ordering}
In this subsection, we make two claims required in the main text:

\begin{proposition} \label{prop:fim3app-G-modulus}
Under \eqref{eq:fim3app-parameter-range}, for each
$\mu\in\{-\sqrt3,0,\sqrt3\}$, the quantity
$|\mathcal G_\mu^{(0)}(k)|$ is an even function of
$\zeta=k\Delta x$ and is nonincreasing with respect to $|\zeta|$ on
$0\leq|\zeta|\leq\pi$.  In particular,
\begin{equation}
 |\mathcal G_\mu^{(0)}(k)|\leq1,
 \qquad
 \mu\in\{-\sqrt3,0,\sqrt3\},
 \label{eq:fim3app-G-max-zero}
\end{equation}
and equality holds at $k=0$.
\end{proposition}

\begin{proposition} \label{prop:fim3app-G0-largest}
Under \eqref{eq:fim3app-parameter-range}, when $|k\Delta x|\leq 2\pi/5$,
\begin{equation}
  |\mathcal G_0^{(0)}(k)|
  \geq|\mathcal G_{\pm\sqrt3}^{(0)}(k)|.
  \label{eq:fim3app-G0-largest-nonstrict}
\end{equation}
More precisely, the inequality is strict for $0<|k\Delta x|\leq2\pi/5$, whereas all three moduli are equal to one at $k=0$.
\end{proposition}

\medskip
\begin{proof}[Proof of Proposition 1.]
Set
\begin{equation}
 r=\frac{\mu}{C_{M+1}},
 \qquad
 x=\cos \zeta,
 \qquad
 F_\alpha(x)
 =\frac{
 (1-\nu)^2+\dfrac{\nu^2\alpha^2}{4}
 +\nu(1-\nu)\alpha x}
 {1+\dfrac{\nu^2\alpha^2}{4}-\nu\alpha x}.
 \label{eq:fim3app-Falpha}
\end{equation}
Since $|r|\leq1$ under \eqref{eq:fim3app-parameter-range}, both
$1-r$ and $1+r$ belong to $[0,2]$.  Taking the modulus of
\eqref{eq:fim3analysis-G} gives
\begin{equation}
 |\mathcal G_\mu^{(0)}(k)|^2
 =F_{1-r}(x)F_{1+r}(x).
 \label{eq:fim3app-G-modulus-F}
\end{equation}
For $\alpha>0$,
\begin{equation}
 \frac{dF_\alpha}{dx}
 =\frac{
 \nu\alpha(2-\nu)
 \left(1-\nu+\dfrac{\nu^2\alpha^2}{4}\right)}
 {\left(1+\dfrac{\nu^2\alpha^2}{4}-\nu\alpha x\right)^2}>0,
 \label{eq:fim3app-F-monotone}
\end{equation}
whereas $F_0(x)=(1-\nu)^2$ is independent of $x$.  Therefore the right-hand
side of \eqref{eq:fim3app-G-modulus-F} is nondecreasing in
$x=\cos\zeta$.  Since $\cos\zeta$ is even in $\zeta$ and decreases as
$|\zeta|$ increases from $0$ to $\pi$, it follows that
$|\mathcal G_\mu^{(0)}(k)|$ is even in $\zeta$ and nonincreasing in
$|\zeta|$.  Finally, direct substitution into \eqref{eq:fim3analysis-G}
gives
\begin{equation}
 \mathcal G_\mu^{(0)}(0)=1,
\end{equation}
which proves \eqref{eq:fim3app-G-max-zero}.
\end{proof}

\begin{proof}[Proof of Proposition 2.]
Replacing $\mu$ by $-\mu$ in \eqref{eq:fim3analysis-G} gives the complex
conjugate after exchanging the two factors.  Hence
\begin{equation}
 |\mathcal G_{-\sqrt3}^{(0)}(k)|
 =|\mathcal G_{\sqrt3}^{(0)}(k)|.
 \label{eq:fim3app-nonzero-equal}
\end{equation}
It remains to compare the branches $\mu=0$ and $\mu=\sqrt3$. Set $q=3/C_{M+1}^2$ and $t=1-\cos k\Delta x$. Then $0<q\leq1$ and $0\leq t\leq2$. For $0\leq\alpha\leq 2$, define
\begin{equation}
 D_\alpha(t)
 =1+\frac{\nu^2\alpha^2}{4}-\nu\alpha(1-t)>0.
 \label{eq:fim3app-Dalpha}
\end{equation}
A direct subtraction of the two squared moduli gives
\begin{equation}
 |\mathcal G_0^{(0)}(k)|^2
 -|\mathcal G_{\sqrt3}^{(0)}(k)|^2
 =\frac{\nu^3q(2-\nu)t}
 {32[D_1(t)]^2 D_{1-\sqrt{q}}(t) D_{1+\sqrt{q}}(t)}P_q(t),
 \label{eq:fim3app-G-difference}
\end{equation}
where
\begin{displaymath}
  P_q(t)=64(1-\nu)^2t^2
  +2\left[(2-\nu)^2(\nu^2+20\nu-20)-\nu^4q\right]t
  +(2-\nu)^2\left[5(2-\nu)^2-\nu^2q\right].
\end{displaymath}
Moreover,
\begin{equation}
 \frac{\partial P_q(t)}{\partial q}
 =-2\nu^4t-\nu^2(2-\nu)^2<0,
 \label{eq:fim3app-Pq-q}
\end{equation}
so $q\leq1$ implies $P_q(t)\geq P_1(t)$. Let
\begin{equation}
 t_\ast=1-\cos\frac{2\pi}{5}=\frac{5-\sqrt5}{4}.
 \label{eq:fim3app-tstar}
\end{equation}
For $0<\nu<1$, since
\begin{equation}
  P_1''(t)=128(1-\nu)^2\geq0, \qquad
  P_1'(t_\ast)
  =32(\nu-1)(\nu^2-\sqrt5\,\nu+\sqrt5)<0,
\end{equation}
we conclude that $P_1'(t)<0$ for $0\leq t\leq t_\ast$. Thus $P_1$ is decreasing on $[0,t_\ast]$ and
\begin{displaymath}
  P_q(t) \geq P_1(t)\geq P_1(t_\ast)=4\nu^2 \left[\nu^2+(1-2\sqrt5)\nu+2\sqrt5-1\right]>0, \qquad \forall t \in [0,t_\ast].
\end{displaymath}
Therefore, when $0<|k\Delta x|\leq2\pi/5$, every factor on the right-hand side of \eqref{eq:fim3app-G-difference} is positive, which proves \eqref{eq:fim3app-G0-largest-nonstrict}.
\end{proof}

\subsection{First-order correction for $\gamma_1=1$}
\label{subsec:fim3app-first-order}

The first-order correction to the amplification factor can be derived in a manner similar to the leading-order analysis. Here we first claim that along an eigenvalue branch with a nonzero limit, the error $e_{4,j}^{n,r}$ has order $O(\varepsilon^2)$ when $M \geq 4$. Taking the fifth component of the equation \eqref{eq:fim3app-kin-forward-error}, we obtain
\begin{displaymath}
  \left(1+\frac{\omega}{\varepsilon}\right)e_{4,j}^{n,s+1/2}
  = (1-\nu)e_{4,j}^{n,s} + \frac{\nu}{2}e_{4,j-1}^{n,s+1/2} + \frac{\nu}{2}e_{4,j+1}^{n,s}
  + \frac{\nu}{2C_{M+1}}(e_{3,j-1}^{n,s+1/2}-e_{3,j+1}^{n,s})
  + \frac{5\nu}{2C_{M+1}}(e_{5,j-1}^{n,s+1/2}-e_{5,j+1}^{n,s}),
\end{displaymath}
where the terms involving $e_{5,j\pm 1}$ are omitted when $M=4$ and the right-hand side is $O(\varepsilon)$ due to \eqref{eq:fim3app-leading-HKE-expansion}. Since $1 + \omega/\varepsilon = O(\varepsilon^{-1})$, it follows that $e_{4,j}^{n,s+1/2} = O(\varepsilon^2)$. Similarly, the fifth component of \eqref{eq:fim3app-kin-backward-error} gives $e_{4,j}^{n,s+1}=O(\varepsilon^2)$. Since $e_{4,j}^{n,\gamma_1} = e_{4,j}^{n+1,0}$, below we will also assume $e_{4,j}^{n,0}$ to be $O(\varepsilon^2)$. This conclusion enables us to derive the first-order correction to the amplification factor by considering only the first four components of the error vector $\delta \boldsymbol u_j^{(n)}$, as other components are not coupled to these four components up to the first order in $\varepsilon$.

For simplicity, we set $\gamma_1=1$ in this subsection. For a fixed wave number $k$, we write
\begin{equation}
 \boldsymbol H_j^{n,r}
 =\widehat{\boldsymbol H}^{n,r}(\varepsilon,k)e^{\mathrm i j k\Delta x},
 \qquad
 e_{3,j}^{n,r}=\widehat e_3^{n,r}(\varepsilon,k)e^{\mathrm i j k\Delta x},
 \qquad
 r\in\left\{0,\frac12,1\right\}.
 \label{eq:fim3app-first-order-fourier}
\end{equation}
Applying the results of the leading-order analysis, we know that along the leading-order eigenvalue branch $\mu$, the first three components of an eigenvector of $\mathbf{T}(\varepsilon,k)$, together with the intermediate amplitudes obtained during the moment sweeps, can be expanded as
\begin{align}
  \widehat{\boldsymbol H}^{n,0}(\varepsilon,k) &= a \boldsymbol v_\mu + \varepsilon \widehat{\boldsymbol H}^{n,0,(1)}(k) + O(\varepsilon^2), \label{eq:fim3app-first-order-expansion-1}\\
  \widehat{\boldsymbol H}^{n,1/2}(\varepsilon,k) &= a \phi_\mu^{(0)}(k) \boldsymbol v_\mu + \varepsilon \widehat{\boldsymbol H}^{n,1/2,(1)}(k) + O(\varepsilon^2), \label{eq:fim3app-first-order-expansion-2}\\
  \widehat{\boldsymbol H}^{n,1}(\varepsilon,k) &= a \mathcal G_\mu^{(0)}(k) \boldsymbol v_\mu + \varepsilon \widehat{\boldsymbol H}^{n,1,(1)}(k) + O(\varepsilon^2), \label{eq:fim3app-first-order-expansion-3}
\end{align}
and for the fourth component, we let $\widehat e_3^{n,r}(\varepsilon,k) = a \varepsilon \xi_r(k) + O(\varepsilon^2)$ for $r \in \{0, 1/2, 1\}$. Since the leading-order eigenvalue of $\mathbf{T}(\varepsilon,k)$ is $\Lambda_\mu^{(0)}(k) = [\mathcal G_\mu^{(0)}(k)]^{\gamma_2 + 1}$, and the hydrodynamic stage does not change $e_{3,j}^{n,r}$, we have
\begin{equation}
  \xi_1(k) = [\mathcal G_\mu^{(0)}(k)]^{\gamma_2+1} \xi_0(k).
  \label{eq:fim3app-xi1-xi0}
\end{equation}

Without loss of generality, we may assume the first-order term in \eqref{eq:fim3app-first-order-expansion-1} has no component along $\boldsymbol v_\mu$:
\begin{equation}
 \widehat{\boldsymbol H}^{n,0,(1)}(k)
 =\sum_{\mu'\in\{-\sqrt3,0,\sqrt3\}\setminus\{\mu\}}
 c_{\mu'}^{(1)}(k)\boldsymbol v_{\mu'}.
 \label{eq:fim3app-first-order-normalization}
\end{equation}
Substituting \eqref{eq:fim3app-first-order-expansion-1}, \eqref{eq:fim3app-first-order-expansion-2} and \eqref{eq:fim3app-first-order-normalization} into the first four equations of \eqref{eq:fim3app-kin-forward-error} and collecting the $O(\varepsilon)$ terms in the first three equations and the $O(1)$ terms in the fourth, we obtain
\begin{align}
 &\left[\mathbf I-\frac{\nu}{2}
 \left(\mathbf I+\frac{\mathbf A_h}{C_{M+1}}\right)e^{-\mathrm i k\Delta x}\right]
 \widehat{\boldsymbol H}^{n,1/2,(1)}(k)
 \notag\\
 &\qquad=\sum_{\mu'\in\{-\sqrt3,0,\sqrt3\}\setminus\{\mu\}}
 \left[1-\nu+\frac{\nu}{2}\left(1-\frac{\mu'}{C_{M+1}}\right)e^{\mathrm i k\Delta x}\right]
 c_{\mu'}^{(1)}(k)\boldsymbol v_{\mu'}
 +\frac{a\nu}{2C_{M+1}}\boldsymbol b
 \left[\xi_{1/2}(k)e^{-\mathrm i k\Delta x}-\xi_0(k)e^{\mathrm i k\Delta x}\right],
 \label{eq:fim3app-forward-first-order-coefficients}\\
 &\omega\xi_{1/2}(k)
 =\frac{\nu}{2C_{M+1}}
 \left[\phi_\mu^{(0)}(k)e^{-\mathrm i k\Delta x}-e^{\mathrm i k\Delta x}\right].
 \label{eq:fim3app-forward-leading-moment-coefficient}
\end{align}
By choosing the eigenvectors
\begin{equation}
  \boldsymbol v_{-\sqrt3}=\begin{pmatrix}1\\-\sqrt3\\1\end{pmatrix},
  \qquad
  \boldsymbol v_0=\begin{pmatrix}-2\\0\\1\end{pmatrix},
  \qquad
  \boldsymbol v_{\sqrt3}=\begin{pmatrix}1\\\sqrt3\\1\end{pmatrix},
  \label{eq:fim3app-hydro-eigenvectors}
\end{equation}
we have $\boldsymbol{b} = \boldsymbol{v}_{-\sqrt3} + \boldsymbol{v}_0 + \boldsymbol{v}_{\sqrt3}$, and $\widehat{\boldsymbol{H}}^{n,1/2,(1)}(k)$ and $\xi_{1/2}(k)$ can be solved from the equations \eqref{eq:fim3app-forward-first-order-coefficients} and \eqref{eq:fim3app-forward-leading-moment-coefficient} as
\begin{equation}
  \widehat{\boldsymbol H}^{n,1/2,(1)}(k)
  =\alpha_\mu^{(1)}(k) a \boldsymbol v_\mu
  +\sum_{\mu'\ne\mu}\alpha_{\mu'}^{(1)}(k)\boldsymbol v_{\mu'}, \qquad
  \xi_{1/2}(k) = \frac{1}{2\Delta x}
  \left[\phi_\mu^{(0)}(k)e^{-\mathrm i k\Delta x}-e^{\mathrm i k\Delta x}\right],
  \label{eq:fim3app-xi-half}
\end{equation}
where
\begin{equation}
 \alpha_\mu^{(1)}(k)
 =\frac{\nu(e^{-\mathrm i k\Delta x}\xi_{1/2}(k)
 -e^{\mathrm i k\Delta x}\xi_0(k))}
 {2C_{M+1}\left[1-\dfrac{\nu}{2}
 \left(1+\dfrac{\mu}{C_{M+1}}\right)e^{-\mathrm i k\Delta x}\right]},
 \label{eq:fim3app-alpha}
\end{equation}
and $\alpha_{\mu'}^{(1)}(k)$ for $\mu'\ne\mu$ are not needed in the subsequent calculation. Applying a similar procedure to the backward sweep \eqref{eq:fim3app-kin-backward-error}, we obtain $\widehat{\boldsymbol H}^{n,1,(1)}(k)$ and $\xi_1(k)$ as follows:
\begin{equation}
  \widehat{\boldsymbol H}^{n,1,(1)}(k)
  =\beta_\mu^{(1)}(k) a \boldsymbol v_\mu
  +\sum_{\mu'\ne\mu}\beta_{\mu'}^{(1)}(k)\boldsymbol v_{\mu'}, \qquad
  \xi_1(k) = \frac{1}{2\Delta x}
  \left[\phi_\mu^{(0)}(k)e^{-\mathrm i k\Delta x}-\mathcal G_\mu^{(0)}(k)e^{\mathrm i k\Delta x}\right],
  \label{eq:fim3app-xi1}
\end{equation}
where
\begin{equation}
  \beta_\mu^{(1)}(k)
  =\frac{\psi_\mu^{(0)}(k) \nu(e^{-\mathrm i k\Delta x}\xi_{1/2}(k)
  -e^{\mathrm i k\Delta x}\xi_0(k))}
  {2C_{M+1}\left[1-\dfrac{\nu}{2}
  \left(1+\dfrac{\mu}{C_{M+1}}\right)e^{-\mathrm i k\Delta x}\right]}
  +\frac{\nu(e^{-\mathrm i k\Delta x}\xi_{1/2}(k)
  -e^{\mathrm i k\Delta x}\xi_1(k))}
  {2C_{M+1}\left[1-\dfrac{\nu}{2}
  \left(1-\dfrac{\mu}{C_{M+1}}\right)e^{\mathrm i k\Delta x}\right]}.
\end{equation}
and $\beta_{\mu'}^{(1)}(k)$ for $\mu'\ne\mu$ are insignificant in the present calculation. Here we recall that the definitions of $\xi_0(k)$, $\xi_{1/2}(k)$ and $\xi_1(k)$ are given in \eqref{eq:fim3app-xi1-xi0}, \eqref{eq:fim3app-xi-half} and \eqref{eq:fim3app-xi1}, respectively.

We now consider the hydrodynamic stage, which changes only the first three components. Assume
\begin{displaymath}
  \boldsymbol{E}_j^{n,\ell} = \widehat{\boldsymbol{E}}^{n,\ell}(\varepsilon,k) e^{\mathrm{i} j k \Delta x}, \quad
  \widehat{\boldsymbol{E}}^{n,\ell}(\varepsilon,k) = a [\mathcal{G}_{\mu}^{(0)}(k)]^{\ell+1} \boldsymbol{v}_{\mu} + \varepsilon \widehat{\boldsymbol{E}}^{n,\ell,(1)}(k) + O(\varepsilon^2), \qquad \ell = 0, 1, \ldots, \gamma_2.
\end{displaymath}
Let
\begin{displaymath}
  \widehat{\boldsymbol{E}}^{n,\ell,(1)}(k) = \gamma_{\mu}^{\ell, (1)}(k) a \boldsymbol{v}_{\mu} + \sum_{\mu' \ne \mu} \gamma_{\mu'}^{\ell, (1)}(k) \boldsymbol{v}_{\mu'}, \qquad
  \ell = 0, 1, \ldots, \gamma_2.
\end{displaymath}
The initialization of the hydrodynamic stage gives $\gamma_{\mu}^{0, (1)}(k) = \beta_\mu^{(1)}(k)$, and the error evolution equations \eqref{eq:fim3app-hydro-forward-error}--\eqref{eq:fim3app-hydro-backward-error} yield the recurrence relation:
\begin{displaymath}
  \gamma_{\mu}^{\ell+1, (1)}(k) = \mathcal{G}_{\mu}^{(0)}(k) \gamma_{\mu}^{\ell, (1)}(k) + J_{\mu}(k) \xi_1(k), \qquad
  \ell = 0, 1, \ldots, \gamma_2 - 1,
\end{displaymath}
with
\begin{equation}
  J_\mu(k)
  =-\mathrm i\frac{\nu}{C_{M+1}}\sin k\Delta x
  \left[
  \frac{\psi_\mu^{(0)}(k)}
  {1-\dfrac{\nu}{2}\left(1+\dfrac{\mu}{C_{M+1}}\right)e^{-\mathrm i k \Delta x}}
  +\frac{1}
  {1-\dfrac{\nu}{2}\left(1-\dfrac{\mu}{C_{M+1}}\right)e^{\mathrm i k \Delta x}}
  \right].
  \label{eq:fim3app-J}
\end{equation}
The coefficient $\gamma_{\mu}^{\gamma_2, (1)}(k)$ is the first-order correction to the leading-order amplification factor $\Lambda_\mu^{(0)}(k)$ (see \eqref{eq:fim3analysis-eps-expansion}). Solving this recurrence yields the following expression for $\Lambda_\mu^{(1)}(k)$:
\begin{equation}
  \Lambda_\mu^{(1)}(k) = \gamma_{\mu}^{\gamma_2, (1)}(k) = [\mathcal{G}_{\mu}^{(0)}(k)]^{\gamma_2} \beta_\mu^{(1)}(k) + \frac{1 - [\mathcal{G}_{\mu}^{(0)}(k)]^{\gamma_2}}{1 - \mathcal{G}_{\mu}^{(0)}(k)} J_\mu(k) \xi_1(k).
  \label{eq:fim3app-lambda1-general}
\end{equation}

For the branch $\mu=0$, the Taylor expansion with respect to $k\Delta x$ is tedious but straightforward. The detailed derivation of the final result \eqref{eq:fim3analysis-lambda1-low} is omitted here. We only claim that the requirement $\gamma_2 (k\Delta x)^2 \ll 1$, stated above \eqref{eq:fim3analysis-lambda1-low}, is needed in the following equation:
\begin{equation}
 \frac{1-[\mathcal G_0^{(0)}(k)]^{\gamma_2}}
 {1-\mathcal G_0^{(0)}(k)}
 =\gamma_2+O((k\Delta x)^2),
 \label{eq:fim3app-geometric-low}
\end{equation}
which expands the coefficient in \eqref{eq:fim3app-lambda1-general}.

\bibliographystyle{elsarticle-num} 
\bibliography{reference}

\end{document}